\documentclass[11pt, leqno]{amsart}
\usepackage{amsthm, amsmath}
\usepackage{amssymb} 
\usepackage{amscd}
\usepackage[curve,matrix,arrow,frame,tips]{xy}
\usepackage{verbatim}	
\usepackage{color}
\usepackage{graphicx}	

\newtheorem{thm}{Theorem}[section]
\newtheorem{prop}[thm]{Proposition}
\newtheorem{lemma}[thm]{Lemma}
\newtheorem{defn}[thm]{Definition}

\newtheorem{rmks}[thm]{Remarks}
\newtheorem{rmk}[thm]{Remark}

\newtheorem{cor}[thm]{Corollary}

\newcommand{\cat}[1]{\mathcal #1}
\newcommand{\x}{\times}
\newcommand{\hs}[1]{\hspace{#1pt}}

\newcommand{\tab}{\hs{10}}

\newcommand{\id}{\mathbb{I}}

\newcommand{\defeq}{\stackrel{\rm def}{=}}
\renewcommand{\subset}{\subseteq}
\renewcommand{\supset}{\supseteq}
\newcommand{\Z}{\mathbb{Z}}

\newcommand{\embed}{\hookrightarrow}
\renewcommand{\O}{\cat{O}}

\newcommand{\iso}{\ \tilde{\to}\ }

\newcommand{\codim}[2]{{\rm codim}_{#2}\ #1}

\newcommand{\squarediagramword}[8]{
\begin{center}
$\begin{CD}
#1      @>#5>>  #2     \\
@V#6VV          @VV#7V \\
#3      @>#8>>  #4
\end{CD}$
\end{center}
}
\newcommand{\A}{\mathbb{A}}
\renewcommand{\P}{\mathbb{P}}

\newcommand{\repn}{representation}
\newcommand{\ga}{\alpha}
\newcommand{\gb}{\beta}
\newcommand{\gph}{\phi}
\newcommand{\gp}{\pi}
\newcommand{\gx}{\xi}
\newcommand{\gps}{\psi}

\newcommand{\gs}{\sigma}

\newcommand{\gd}{\delta}
\newcommand{\gm}{\mu}

\newcommand{\gep}{\epsilon}
\newcommand{\gPS}{\Psi}
\newcommand{\gPH}{\Phi}
\newcommand{\gD}{\Delta}
\newcommand{\gO}{\Omega}

\newcommand{\nbundle}[2]{\cat{N}_{#1 \embed #2}}

\newcommand{\claimend}{\tab $\triangle$}
\newcommand{\dual}[1]{#1^{\vee}}

\newcommand{\C}{\mathbb{C}}
\newcommand{\pt}{{\rm Spec}\, k}
\newcommand{\morp}{morphism}

\newcommand{\etale}{\'etale}
\newcommand{\wrt}{with respect to}

\newcommand{\withoutlog}{ithout loss of generality}
\newcommand{\longto}{\longrightarrow}

\newcommand{\singular}[1]{{\rm Sing}(#1)}
\newcommand{\blowup}[2]{{\rm Blow}_{#2}{#1}}
\newcommand{\disjoint}{\amalg}
\newcommand{\spec}[1]{{\rm Spec}\, #1}

\newenvironment{statementslist}{\begin{tabular}[t]{p{15pt}p{380pt}}}{\end{tabular}}
\newcommand{\kernel}[1]{{\rm kernel}\, #1}
\newcommand{\image}[1]{{\rm image}\, #1}
\newcommand{\stricttransform}[1]{\left\langle #1 \right\rangle}

\newcommand{\gradient}[1]{\nabla #1}
\newcommand{\picard}[2]{{\rm Pic}^{#1}(#2)}
\newcommand{\Endo}[1]{{\rm End}\,(#1)}
\newcommand{\homo}{homomorphism}
\renewcommand{\char}[1]{{\rm char\,}#1}

\newcommand{\ratmap}{\dashrightarrow}
\newcommand{\Proj}{{\rm Proj}\ }

\newcommand{\reduced}[1]{#1_{\rm red}}
\newcommand{\rank}[1]{{\rm rk\,#1}}
\newcommand{\cartdiv}[2]{{\rm CaCl}^{#1}(#2)}
\newcommand{\weildiv}[2]{{\rm Cl}^{#1}(#2)}
\newcommand{\torus}{\mathbb{G}_m}

\newcommand{\gsection}[2]{{\rm H^0}(#1, #2)}
\newcommand{\genby}[1]{\langle\, #1\, \rangle}

\newcommand{\equi}{equivariant}
\newcommand{\proj}{projective}
\newcommand{\qproj}{quasi-projective}
\newcommand{\sm}{smooth}
\newcommand{\smcat}{Sm}
\newcommand{\gsmcat}[1]{#1\text{-}Sm}
\newcommand{\gvar}[1]{#1\text{-}Var}
\newcommand{\lazard}{\mathbb{L}}
\newcommand{\glin}[1]{$#1$-linearized}
\newcommand{\cob}[3]{\Omega_{#1}^{#2}(#3)}
\newcommand{\cobcomp}[3]{\hat{\Omega}_{#1}^{#2}(#3)}
\newcommand{\bigcob}[3]{\overline{\Omega_{#1}^{#2}}(#3)}

\renewcommand{\L}{\cat{L}}

\newcommand{\adinvsh}{admissible\ invertible\ shea}
\newcommand{\adtower}{admissible\ tower}
\newcommand{\qadtower}{quasi-admissible\ tower}

\newcommand{\rsncd}{reduced\ strict\ normal\ crossing\ divisor}
\newcommand{\girred}[1]{$#1$-irreducible}
\newcommand{\rou}{root\ of\ unity}

\newcommand{\dpr}{ouble point relation}

\newcommand{\fgl}{formal group law}
\newcommand{\apicard}[2]{{\rm A}\picard{#1}{#2}}
\newcommand{\inv}{invariant}

\newcommand{\geodim}{{\rm geodim}\,}
\newcommand{\nilpnum}[1]{{\rm Nilp}(#1)}
\newcommand{\basicmod}[4]{\lazard {\rm Z}^{#1,#2}_{#3}(#4)}
\newcommand{\Endoinf}[1]{\Endo{#1}_{\rm inf}}
\newcommand{\Endofin}[1]{\Endo{#1}_{\rm fin}}
\newcommand{\bigbasicmod}[4]{\overline{\lazard {\rm Z}^{#1,#2}_{#3}}(#4)}
\newcommand{\gprojsmcat}[1]{#1\text{-}ProjSm}

\begin{document}

\title{Equivariant Algebraic Cobordism and Equivariant Formal Group Laws}
\author{Chun Lung Liu}
\begin{abstract}
We introduce an equivariant algebraic cobordism theory $\cob{}{G}{-}$ for algebraic varieties with $G$-action, where $G$ is a split diagonalizable group scheme over a field $k$. It is done by combining the construction of the algebraic cobordism theory $\Omega(-)$ by F.~Morel and M.~Levine, with the notion of ``$(G,F)$-formal group law'' with respect to a complete $G$-universe and complete $G$-flag $F$ as introduced by M.~Cole, J.~P.~C.~Greenlees and I.~Kriz. In particular, we use their corresponding representing ring $\lazard_G(F)$ in place of the Lazard ring $\lazard$. We show that localization property and homotopy invariance property hold in $\cob{}{G}{-}$. We also prove the surjectivity of the canonical map $\lazard_G(F) \to \cob{}{G}{\pt}$. Moreover, we give some comparison results with $\Omega(-)$, the \equi\ algebraic cobordism theory introduced by J.~Heller and J.~Malag\'{o}n-L\'{o}pez, the equivariant K-theory and Tom Dieck's equivariant cobordism theory (when $k = \C$). In particular, we proved the equivariant Conner-Floyd isomorphism when $\char{k} = 0$. Finally, we show that our definition of $\cob{}{G}{-}$ is independent of the choice of $F$.
\end{abstract}

\address{Department of Mathematics, Rm 202A, Lady Shaw Building, The Chinese University of Hong Kong, Shatin, Hong Kong} 
\email{clliu@math.cuhk.edu.hk}
\date{\today}

\maketitle

\medskip

\bigskip
\centerline{\sc Contents}
\medskip

\noindent \begin{statementslist}
\S 1. & Introduction \\
\S 2. & Notations and assumptions \\
\S 3. & Preliminaries \\
\S 4. & Definition of the \equi\ algebraic cobordism theory \\
\S 5. & Basic properties \\
\S 6. & The \equi\ algebraic cobordism ring of a point \\
\S 7. & Fundamental properties \\
\S 8 & Special theory and some advanced properties \\
\S 9 & Comparison with other algebraic cobordism theories \\
\S 10. & More on $\cob{\rm Tot}{G}{-}$ \\
\S 11. & Comparison with the \equi\ K-theory \\
\S 12. & Realization functor \\
\S 13. & Flag dependency \\
References & 
\end{statementslist}

\bigskip
\medskip

\medskip

\section{Introduction}

In their paper \cite{universal alg cobor}, M.~Levine and F.~Morel define an algebraic cobordism theory $\Omega(-)$, which is an analogue of the complex cobordism theory, in spite of the absence of notion of boundary in algebraic geometry. Roughly speaking, if $X$ is a separated scheme of finite type over the ground field $k$, then we consider elements of the form $(f : Y \to X, \L_1, \ldots, \L_r)$ where $f$ is \proj, $Y$ is an irreducible \sm\ variety over $k$ and $\L_i$ are invertible sheaves over $Y$ (the order of $\L_i$ does not matter and $r$ can be zero). There is a natural notion of isomorphism of elements of this form. Denote the free abelian group generated by isomorphism classes of such elements by $Z(X)$. Let $\underline{\Omega}(X)$ be the quotient of $Z(X)$ by the subgroup corresponding to imposing the \textbf{(Dim)} and \textbf{(Sect)} axioms  (following the notations in \cite{universal alg cobor}). In a nutshell, axiom \textbf{(Dim)} kills all elements of the form $(\id_Y, \L_1, \ldots, \L_n)$ whenever $n > \dim Y$ and axiom \textbf{(Sect)} equalizes the elements $(\id_Y, \L)$ and $(D \embed Y)$ if $D$ is a smooth divisor on $Y$ such that $\L \cong \O_Y(D)$. The algebraic cobordism group $\Omega(X)$ is then defined to be the quotient of $\lazard \otimes_{\Z} \underline{\Omega}(X)$, where $\lazard$ is the Lazard ring, by the $\lazard$-submodule corresponding to imposing the formal group law \textbf{(FGL)}.

This cobordism theory satisfies a number of basic properties and some more advanced properties like the localization property and the homotopy invariance property. Moreover, the cobordism ring $\Omega(\pt)$ is isomorphic to the Lazard ring $\lazard$ when the characteristic of $k$ is 0, which is what we expect from the complex cobordism theory (see Corollary 1.2.11 and Theorem 4.3.7 in \cite{universal alg cobor}).

It is also possible to construct an algebraic cobordism theory via a more geometric approach. Suppose $X$ is a \sm\ variety over $k$. One may consider the abelian group $M(X)^+$ generated by isomorphism classes of \proj\ \morp s $f : Y \to X$, where $Y$ is a \sm\ variety over $k$. A relation called ``double point relation'' is introduced in \cite{alg cobor by DPR} and it is shown that the theory $\omega(-)$ obtained by imposing this relation on $M(-)^+$ is canonically isomorphic to the theory $\Omega(-)$ under the assumption that the characteristic of $k$ is 0 (see Theorem 1 of \cite{alg cobor by DPR}). 

The current paper contributes to the development of equivariant algebraic cobordism theory for varieties with group action. Following the pattern in topology, we can expect to also have several different approaches to defining \equi\ algebraic cobordism theory. For the analogue of one of the homotopy theoretic cobordism theories in the algebraic geometry setup, one can employ Totaro's approximation of $EG$, which leads to a definition given by taking inverse limit of a system of ``good pairs" (see \cite{homo equi alg cobor} for details). Another, possibly equivalent, approach is pursued by Krishna in \cite{homo equi alg cobor 2}.

We are more interested in a geometric approach, i.e., by considering varieties with $G$-action.
One idea is to impose the $G$-action on the d\dpr. This approach is pursued in \cite{geo equi alg cobor}. Due to the lack of transversality in the
\equi\ setting, a generalized version of the double point relation is introduced in \cite{geo equi alg cobor} and an \equi\ algebraic cobordism theory $\cat{U}^G(-)$ is defined accordingly. It is also shown that this generalized double point relation holds in the non-\equi\ theory $\omega(-)$ (see Corollary 3.7 in \cite{geo equi alg cobor}). 

The theory $\cat{U}^G(-)$ has a very strong geometric flavor, but it is insufficient to prove the localization property. This is mainly due to the absence of the (first) Chern class operator for arbitrary \glin{G}\ invertible sheaves in this theory. For this reason, it seems appropriate to define an \equi\ algebraic cobordism theory following the original ideas in \cite{universal alg cobor}, and that is our approach in this paper.

The crucial part in trying to define an \equi\ version of $\Omega(-)$ is on finding the proper notion of ``$G$-\equi\ formal group law'' and its representing ring. Fortunately, this issue is addressed by M.~Cole, J.~P.~C.~Greenlees and I.~Kriz in their paper \cite{equi FGL}. Suppose $G$ is a compact abelian Lie group. For a complete $G$-universe $\cat{U}$ (over $\mathbb{C}$) and complete $G$-flag $F$, they give a definition of $(G,F)$-\equi\ formal group law (definition 12.2 in \cite{equi FGL}) and prove that (Corollary 14.3 in \cite{equi FGL}) there is a corresponding representing ring $\lazard_G(F)$ (For a flag-independent definition, see definition 11.1 in \cite{equi FGL}).

In this paper, we will focus on the following configuration on the group $G$ and ground field $k$ : $G$ would be a split diagonalizable group (product of a finite abelian group $G_f$ and a split torus $G_t$) and $k$ would be a field with characteristic 0 or $p$ where $p$ is relatively prime to the order of $G_f$. Moreover, for technical reason, we will further assume that $k$ contains a primitive $e$-th root of unity where $e$ is the exponent of $G_f$. 

Let us briefly justify our assumptions on $G$ and $k$ here first. We focus on abelian group $G$ simply because the same assumption is imposed in the construction on $\lazard_G(F)$. The other rather technical assumptions on $G$ and $k$ are imposed because we need to guarantee that any finite dimensional $G$-representation over $k$ can be written as the direct sum of 1-dimensional representations (holds automatically over $\mathbb{C}$). It is worth mentioning that such fact plays an essential role in defining the equivariant formal group law and hence $\lazard_G(F)$ (see \cite{equi FGL}).

Now, similar to the theory $\Omega(-)$, one may define an \equi\ algebraic cobordism theory by imposing the \textbf{(Sect)}, \textbf{(Dim)} axioms and the \equi\ formal group law \textbf{(EFGL)} on $\lazard_G(F) \otimes_{\Z} Z^G(X)$, where $Z^G(X)$ is the free abelian group generated by isomorphism classes of elements of the form $(f : Y \to X, \L_1, \ldots, \L_r)$. But the \textbf{(Dim)} axiom would imply that any elements with positive (cohomological) degree will vanish. In our context, this would be completely unnatural because $BG$ is not a zero-dimensional object in general and the homotopical equivariant cobordism ring is $MU^*({\rm pt} \times^G EG) = MU^*(BG).$ Moreover, the canonical map from $\lazard_G(F)$ to the \equi\ algebraic cobordism ring over $\pt$ will not be an iso\morp\ as we want (see Remark \ref{rmk why drop dim axiom}). Therefore, it seems that the only reasonable approach is to drop the \textbf{(Dim)}\ axiom.

Roughly speaking, our \equi\ algebraic cobordism theory $\cob{}{G}{-}$ is defined as follow. For a $G$-variety $X$ over $k$, we define $\basicmod{G}{F}{}{X}$ to be the $\lazard_G(F)$-module generated by infinite sums of the form
$$\sum_{I \geq 0} a_I [f : Y \to X, V^{i_1}_{S_1}(\L_1), V^{i_2}_{S_2}(\L_2), \ldots, V^{i_r}_{S_r}(\L_r)]$$
where $I$ is the multi-index $(i_1, \ldots, i_r)$, $a_I$ are elements in $\lazard_G(F)$ and $V^{i_j}_{S_j}(\L_j)$ is a ``twisted sequence'' of $\L_j$ (see equation (\ref{eqn basic element}) for details). Then $\cob{}{G}{X}$ is defined as the quotient of $\basicmod{G}{F}{}{X}$ by imposing the \textbf{(Sect)} and \textbf{(EFGL)} axioms. It is worth mentioning that, if $F'$ is another complete $G$-flag and $\cob{}{G,F}{-}$, $\cob{}{G,F'}{-}$ are the theories defined upon $F$, $F'$ respectively, then $\cob{}{G, F}{-}$ and $\cob{}{G, F'}{-}$ are canonically isomorphic (see Proposition \ref{prop flag independent}). Hence, our definition is indeed flag-independent.

With the aid of the canonically defined (first) Chern class operator, we are able to show many interesting results in this theory (when $\char{k} = 0$). We show that the canonical map $\lazard_G(F) \to \cob{}{G}{\pt}$ is surjective (Theorem \ref{thm gen by lazard}). Moreover, if the completion map $\lazard_G(F) \to \hat{\lazard}_G(F)$, \wrt\ a canonically defined ideal, is injective, then $\lazard_G(F) \to \cob{}{G}{\pt}$ is an iso\morp\ (Corollary \ref{cor lazard iso theory over a pt 2}). We also prove that the localization property and the homotopy invariance property hold in our theory. 

As in $\cob{}{}{-}$, one might also expect the projective bundle formula to hold in our theory $\cob{}{G}{-}$. But it will then contradict with the fact that $MU_G(\P(\cat{U}))$, where $MU_G(-)$ is Tom Dieck's equivariant cobordism theory, is not a power series ring over $MU_G$. To remedy this situation, we introduce a ``special theory" $\cob{G}{s}{-}$, which can be thought as the middle ground between our \equi\ algebraic cobordism theory and the \equi\ K-theory. We then manage to show that both the projective bundle formula and the extended homotopy property hold in $\cob{G}{s}{-}$. As a consequence, the higher Chern class operators of a \glin{G}\ locally free sheaf can also be defined (see section \ref{sect special theory} for more details).

Furthermore, we establish some interesting comparison results between our \equi\ algebraic cobordism theory and other theories. In particular, we show that the ``forgetful map'' $\cob{}{G}{-} \to \Omega(-)$ is well-defined and it is an iso\morp\ when $G$ is the trivial group (see Proposition \ref{prop iso to non equi theory} and Corollary \ref{cor map to non equi theory}). In addition, when $\char{k} = 0$, we show that there is an abelian group \homo
$$\cob{}{G}{-} \stackrel{\Psi_{\rm Tot}}{\longto} \cob{\rm Tot}{G}{-}$$
where $\cob{\rm Tot}{G}{-}$ is the \equi\ algebraic cobordism theory defined by J.~Heller and J.~Malag\'{o}n-L\'{o}pez using Totaro's approximation of $EG$ as in \cite{homo equi alg cobor} (see Proposition \ref{prop map to homo equi cobor theory}), which can be thought as an analogue of the well-known map in Topology
$$MU_G(-) \to MU(- \x^G EG)$$
where $MU_G(-)$ is Tom Dieck's equivariant cobordism theory. More importantly, as an analogue to Corollary 4.2.12 in \cite{universal alg cobor}, we prove the \equi\ Conner-Floyd isomorphism, i.e., there is a canonical isomorphism
$$R(G)[v,v^{-1}] \otimes_{\lazard_G(F)} \cob{}{G}{-} \to K_0(G;-)[v,v^{-1}]$$
when $\char{k} = 0$ (see Theorem \ref{thm equi Conner Floyd}).

As mentioned in section 13 of \cite{equi FGL 2}, there is a canonical ring \homo
$$\lazard_G(F) \to MU_G$$
and it is conjectured to be an isomorphism. Therefore, we believe our theory $\cob{}{G}{-}$ can be considered as an algebraic analogue of Tom Dieck's equivariant cobordism $MU_G(-)$. This is justified by the realization functor
$$\cob{G}{*}{-} \to MU^{2*}_G(-)$$
when $k = \C$ (Theorem \ref{thm realization fct}).

Even though we work in a less general configuration than the previous approach in defining an equivariant algebraic cobordism theory given by Heller and Malag\'on-L\'opez, we managed to construct a canonical realization functor from our theory to Tom Dieck's theory $MU_G(-)$, which is lacking in $\cob{\rm Tot}{G}{-}$. Moreover, since their approaches are basically algebraic analogues of the equivariant cobordism theory $MU( - \times^G BG)$ in Topology, constructing a realization functor from $\cob{\rm Tot}{G}{-}$ to $MU_G(-)$ would be impractical. Furthermore, notice that under our assumptions on $G$, $k$ and when $\char k = 0$, we have 
$$\lazard_G(F) \stackrel{f}{\longto} \cob{}{G}{\pt} \stackrel{\Psi_{\rm Tot}}{\longto} \cob{\rm Tot}{G}{\pt}$$
(see Proposition \ref{prop map to homo equi cobor theory 2}). As mentioned before, $f$ is surjective and $\Psi_{\rm Tot} \circ f$ is nothing but the completion map \wrt\ a canonically defined ideal. Hence, in a certain sense, our theory $\cob{}{G}{-}$ captures more information than the theory $\cob{\rm Tot}{G}{-}$ defined by Heller and Malag\'on-L\'opez.

As the definition of our theory relies heavily on the equivariant formal group law and the equivariant Lazard ring, generalization of our theory to accept non-abelian group or arbitrary ground field $k$ would not be possible until such generalization is achieved in the context of the equivariant formal group law.

Here is the outline of this paper. In section \ref{sect notation}, we introduce some notations and fix some basic assumptions that we use throughout the paper. In section \ref{sec preliminaries}, we state and prove a number of basic, relatively general facts. Then we give a formal definition of our \equi\ algebraic cobordism theory $\cob{}{G}{-}$ in section \ref{sect defn of theory}. In section \ref{sect basic properties}, we prove some basic properties and show that the \equi\ versions of the double point relation, the blow up relation and the extended double point relation hold in $\cob{}{G}{-}$. 

In section \ref{sect cobor ring}, we investigate the \equi\ algebraic cobordism ring $\cob{}{G}{\pt}$. In particular, we show the following Theorem (Theorem \ref{thm gen by lazard} in the text) :

\medskip

\noindent \textbf{Theorem 1.}\tab {\it Suppose $\char{k} = 0$ and $k$ contains a primitive $e$-th \rou, where $e$ is the exponent of $G_f$. Then the canonical $\lazard_G(F)$-algebra homo\morp 
$$\lazard_G(F) \to \cob{}{G}{\pt},$$
which sends $a$ to $a\, [\id_{\pt}]$, is surjective.}

\medskip

\noindent Under the assumption that $\char{k} = 0$, we prove some more advanced properties, namely, the localization property and the homotopy invariance property in section \ref{sect fundamental properties}. 

In section \ref{sect special theory}, we define the notion ``special theory" $\cob{G}{s}{-}$. We then prove that the projective bundle formula and the extended homotopy property hold in $\cob{G}{s}{-}$ and define the higher Chern class operators of \glin{G}\ locally free sheaves of arbitrary finite ranks accordingly.

In section \ref{sect comparison with other theories}, we compare our theory $\cob{}{G}{-}$ to M.~Levine and F.~Morel's non-\equi\ theory $\Omega(-)$ and J.~Heller and J.~Malag\'{o}n-L\'{o}pez's \equi\ algebraic cobordism theory $\cob{\rm Tot}{G}{-}$ (when $G$ is a split torus). In section \ref{sect more on Totaro cobor}, we extend the definition of $\cob{\rm Tot}{G}{-}$ to allow $G$ to be split diagonalizable, compute the ring structure of $\cob{\rm Tot}{G}{\pt}$ and generalize our results in section \ref{sect comparison with other theories}. As a consequence, we prove the following Theorem (Corollary \ref{cor lazard iso theory over a pt 2} in the text) :

\medskip

\noindent \textbf{Theorem 2.}\tab {\it Suppose $\char{k} = 0$ and $k$ contains a primitive $e$-th \rou, where $e$ is the exponent of $G_f$. If the completion map $\lazard_G(F) \to \hat{\lazard}_G(F)$ is injective, then the canonical ring homo\morp\
$$\lazard_G(F) \to \cob{}{G}{\pt}$$
is an iso\morp.}

\medskip

In section \ref{sect K theory}, we compare our theory to the \equi\ K-theory. To be more precise, we prove the \equi\ Conner-Floyd isomorphism (Theorem \ref{thm equi Conner Floyd} in the text) :

\medskip

\noindent \textbf{Theorem 3.}\tab {\it Suppose $\char{k} = 0$ and $k$ contains a primitive $e$-th \rou, where $e$ is the exponent of $G_f$. Then there is a canonical ring \homo\ $\lazard_G(F) \to R(G)[v,v^{-1}]$, where $R(G)$ is the character ring of $G$. Moreover, there is a canonical, $R(G)[v,v^{-1}]$-module isomorphism
$$\gPS_K : R(G)[v,v^{-1}] \otimes_{\lazard_G(F)} \cob{}{G}{X} \to K_0(G;X)[v,v^{-1}],$$
for any \sm\ $G$-variety $X$, and it commutes with \proj\ push-forward, \sm\ pull-back, (first) Chern class operators and external product.}

It is worth mentioning that in Topology, a much stronger result is obtained. The analogue of $\Psi_K$ will be an isomorphism for any compact Lie group $G$ (see \cite{equi conner floyd}). But since the construction of our theory is based on the equivariant Lazard ring, which is only defined for compact abelian Lie groups, we can not extend the definition of our theory to include non-abelian group $G$ until such extension is available for the equivariant Lazard ring.

We should also point out that, in Topology, the equivariant Conner-Floyd isomorphism holds for both geometric cobordism and homotopical cobordism, i.e.,
$$\tilde{\cat{U}}_{\text{Top}, G}(X) \otimes_{\cat{U}_{\text{Top}, G}} K_G \iso \tilde{K}_G(X)$$
when $G$ is a finite group,
$$\widetilde{MU}_{G}(X) \otimes_{MU_{G}} K_G \iso \tilde{K}_G(X)$$
when $G$ is a compact Lie group (see Theorem A, B in \cite{equi conner floyd}). But unfortunately, in algebraic geometry, the only theory available so far as an analogue to the geometric cobordism theory $\tilde{\cat{U}}_{\text{Top}, G}(-)$ in Topology is the one defined by generalized double point relation, denoted as $\cat{U}^G(-)$ (see \cite{geo equi alg cobor}) and a canonical map from $\cat{U}^G$ to $K_G$ is lacking. 

\medskip

In section \ref{sect realization functor}, we recall the definition of Tom Dieck's equivariant cobordism theory $MU_G(-)$ and the Gysin \homo\ (\proj\ push-forward) and show that there is a canonical realization functor, when $k = \C$ (Theorem \ref{thm realization fct}\ in the text) :

\medskip

\noindent \textbf{Theorem 4.}\tab {\it There is a canonical $\lazard_G(F)$-\homo
$$\Psi_{\rm Top} : \cob{G}{}{X} \to MU_G(X),$$
for any \sm, \proj\ $G$-variety $X$, and it commutes with \proj\ push-forward, \sm\ pull-back, (first) Chern class operators and external product. When $X$ is equidimensional, there is a canonical grading on $\cob{G}{}{X}$ and $\Psi_{\rm Top} : \cob{G}{*}{X} \to MU_G^{2*}(X).$}

\medskip

Finally, we devote the last section to showing that our definition of $\cob{}{G}{-}$ is actually independent of the choice of the complete $G$-flag $F$.

\medskip

\begin{center}
\textbf{Acknowledgements}
\end{center}

I would like to thank M.~Levine for useful conversations and helpful comments.

\bigskip
\bigskip

\bigskip
\bigskip

\section{Notations and assumptions}
\label{sect notation}

In this paper, all schemes are over a ground field $k$. $G$ is a split diagonalizable group scheme, i.e., the product of a finite abelian group scheme, denoted by $G_f$, and a split torus, denoted by $G_t$. We will assume $\char{k}$ is either zero or relatively prime to the order of $G_f$. We will also assume that $k$ contains a primitive $e$-th \rou, where $e$ is the exponent of $G_f$. Hence, any finite dimensional $G$-\repn\ can be written as direct sum of 1-dimensional $G$-\repn s. We call such a pair $(G,k)$ split. 

We denote the category of \sm, \qproj\ schemes over $k$ with $G$-action by $\gsmcat{G}$ and the category of reduced, \qproj\ schemes over $k$ with $G$-action by $\gvar{G}$. The identity \morp\ will be denoted by $\id_X : X \to X$. We will often use the symbol $\gp_i$ to denote the projection of $X_1 \x \cdots \x X_n$ onto its $i$-th component $X_i$ and $\pi_X$ to denote the structure \morp\ $X \to \pt$. If $X$, $Y$ are two objects in $\gvar{G}$, then $X \x Y$ is considered to be in $\gvar{G}$ with $G$ acting diagonally. An object $Y \in \gvar{G}$ is called \girred{G}\ if there exists an irreducible component $Y'$ of $Y$ such that $G \cdot Y' = Y$. The set of isomorphism classes of \glin{G}\ invertible sheaves over $X$ will be denoted by $\picard{G}{X}$ (see definition 1.6 in \cite{GIT} for the definition of $G$-linearized invertible sheaves).

Recall that a complete $G$-universe $\cat{U}$ is a countably infinite $G$-\repn\ that contains countably many copies of any finite-dimensional irreducible $G$-\repn. In particular, we may take
$$\cat{U} =\ \bigoplus_{\substack{V \text{ f.d. irred. $G$-repn.} \\ n \geq 1}} V^{\oplus n}$$
A complete $G$-flag $F$ is a sequence of $G$-\repn s 
$$0 = V^0 \subset V^1 \subset V^2 \subset \cdots$$
such that $\dim V^i/V^{i-1} = 1$ for all $i \geq 1$ and any finite dimensional $G$-\repn\ would be a sub\repn\ of $V^i$ for some $i$. In particular, any 1-dimensional $G$-character will appear as $V^i/V^{i-1}$ infinitely many times.

We will fix a complete $G$-universe $\cat{U}$ and a complete $G$-flag $F$ throughout this paper. We denote the 1-dimensional $G$-characters $V^i/V^{i-1}$ by $\ga_i$. Therefore, 
$$V^i \cong \alpha_1 \oplus \cdots \oplus \alpha_i$$ 
for any $i \geq 1$. For technical reason, we will assume $\ga_1 = \gep$, the trivial character. All $G$-characters are 1-dimensional unless stated otherwise. Each character $\alpha$ defines a 1-dimensional $G$-\repn, and hence a \glin{G}\ invertible sheaf over $\pt$, which will still be denoted by $\alpha$. Moreover, for an object $X \in \gvar{G}$, we will simply denote the sheaf $\pi^*_X \alpha$ by $\alpha$, if there is no confusion.

For a \morp\ $f : X \to Y$ between schemes and a point $y \in Y$, we denote the fiber product $\spec{k(y)} \x_Y X$ by $f^{-1}(y)$ where $k(y)$ is the residue field of $y$ and $\spec{k(y)} \to Y$ is the \morp\ corresponding to $y$. Similarly, if $Z$ is a subscheme of $Y$, then we denote $Z \x_{Y} X$ by $f^{-1}(Z)$. If $A, B$ are both subschemes of $X$, then we denote $A \x_{X} B$ by $A \cap B$.

In this paper, for a \girred{G}\ object $X \in \gvar{G}$, a $G$-prime divisor $D$ on $X$ is a $G$-\inv, \girred{G}, reduced, codimension 1, closed subscheme of $X$. A $G$-\inv\ (Weil) divisor $D$ on $X$ is a linear combination $\sum_i m_i D_i$ where $D_i$ are distinct, $G$-prime divisors on $X$. We call such a divisor \sm\ if all the multiplicities $m_i$ are 1 and $D_i$ are \sm\ and disjoint. We call a $G$-\inv\ divisor $A_1 + \cdots + A_n$ \rsncd\ if each $A_i$ is a \sm\ $G$-\inv\ divisor and, for each $I \subset \{1, \ldots, n\}$, the closed subscheme $\cap_{i \in I}\, A_i$ is \sm\ with codimension $|I|$ in $X$. We say two $G$-\inv\  divisors $A$, $B$ on $X$ are $G$-equivariantly linearly equivalent if $A - B = {\rm div } f$ for some $f \in {\rm H}^0(X,\cat{K}^*)^G$ where $\cat{K}$ is the sheaf of total quotient rings on $X$ (assuming $X$ is regular in codimension 1).

For a locally free sheaf $\cat{E}$ of rank $r$ over a $k$-scheme $X$, the corresponding vector bundle $E$ over $X$ will be given by
$$E \defeq \spec{{\rm Sym}\, \dual{\cat{E}}}.$$
The same applies to the case when $X$ is a $G$-scheme over $k$ and $\cat{E}$ is $G$-linearized.

\bigskip
\bigskip

\section{Preliminaries}
\label{sec preliminaries}

Let us begin by stating some basic facts about objects in $\gvar{G}$ and \glin{G}\ invertible sheaves over such objects.

\begin{prop}
\label{prop stay reduced}
Suppose $f : X \to Y$ is a \sm\ \morp\ between schemes of finite type over $k$. If $Y$ is reduced, then so is $X$.
\end{prop}

\begin{proof}
By Theorem 10.2 in \cite{hartshorne}, the fibers of $f : X \to Y$ are geometrically regular. In particular, the fibers are geometrically reduced and so, reduced. The result then follows from Corollary 3.3.5 in part 2 of \cite{EGA IV} and the facts that $f$ is flat and $Y$ is reduced.
\end{proof}

\begin{prop}
\label{prop equi embed}
For any \morp\ $f : X \to X'$ in $\gvar{G}$, there exist a $G$-\repn\ $V$ and a $G$-\equi\ immersion $i : X \embed \P(V) \x X'$ such that $f = \pi_2 \circ i$. If we further assume $f$ to be \proj, then $i$ will be a closed immersion.
\end{prop}

\begin{proof}
Since $X$ is \qproj, there exists a (not necessarily \equi) immersion $i_0 : X \embed \P^n$. Define $\L \defeq i_0^* \O(1)$ as an (not necessarily \glin{G}) invertible sheaf over $X$. By Theorem 1.6 in \cite{equi completion}, there exists an integer $m$ such that $\L^{\otimes m}$ is $G_t$-linearizable. Fix a $G_t$-linearization of $\L^{\otimes m}$. Then, $\L' \defeq \otimes_{g \in G_f}\ g^* (\L^{\otimes m})$ will be a \glin{G}\ very ample invertible sheaf over $X$. By Proposition 1.7 in \cite{GIT}, there exists an $G$-\equi\ immersion $i_1 : X \embed \P(V)$ for some $G$-\repn\ $V$ such that $i_1^* O(1) \cong \L'$. Then, the map $(i_1, f) : X \to \P(V) \x X'$ will be the \equi\ immersion we want. If $f$ is \proj, then $(i_1, f)$ will be a closed immersion.
\end{proof}

\begin{prop}
\label{prop smooth extension}
Suppose $\char{k} = 0$.

\noindent \begin{statementslist}
{\rm (1)} & Suppose $Y$ is in $\gsmcat{G}$, $X$ is in $\gvar{G}$ and $U \subset X$ is a $G$-\inv\ open subscheme. If $f : Y \to U$ is a \proj\ \morp\ in $\gvar{G}$, then there exist a $G$-\repn\ $V$ and a $G$-\equi\ closed immersion $i : Y \embed \P(V) \x U$ such that its closure in $\P(V) \x X$ is \sm\ and $f = \pi_2 \circ i$. \\
{\rm (2)} & For any $Y \in$ $\gsmcat{G}$, there exist a $G$-\repn\ $V$ and a $G$-\equi\ immersion $i : Y \embed \P(V)$ such that its closure is \sm.
\end{statementslist}
\end{prop}

\begin{proof}
For part (1), see the proof of Proposition 4.14 in \cite{geo equi alg cobor}. For part (2), by Proposition \ref{prop equi embed}, there exists a $G$-\repn\ $V'$ and a $G$-\equi\ immersion $Y \embed \P(V')$. Denote its closure by $\overline{Y}$. By applying part (1) with $U = Y$, $X = \overline{Y}$ and $f = \id_Y$, we have a $G$-\equi\ immersion $Y \embed \P(V'') \x \overline{Y}$ with \sm\ closure. Then the $G$-\equi\ immersion we want is 
$$Y \embed \P(V'') \x \overline{Y} \embed \P(V'') \x \P(V') \embed \P(V),$$ 
where the last \morp\ is the Segre embedding and $V$ is the corresponding $G$-\repn. 
\end{proof}

\newpage

\begin{prop}
\label{prop contain inv aff}
Any $X \in \gvar{G}$ contains a non-empty, $G$-\inv, affine open subscheme.
\end{prop}

\begin{proof}
W\withoutlog, we may assume $X$ is \girred{G}\ and \sm. Then it is the disjoint union of irreducible components permuted by $G_f$. So we may further assume $X$ is irreducible. We may also assume the $G$-action on $X$ is faithful. By considering the complement of the stabilizers $X_g$ where $1 \neq g \in G_f$, we may assume the $G_f$-action on $X$ is free (in particular, proper). 

Since $X$ is \qproj\ and $G_f$ is finite, the geometric quotient 
$$\pi : X \to X/G_f \defeq X'$$
exists as objects in $\gvar{G}$. By Proposition 0.9 in \cite{GIT}, $\pi$ is actually a principal fiber bundle \wrt\ the $G_f$-action. So it is locally trivial in the \etale\ topology. Therefore, the smoothness of $X$ implies the smoothness of $X'$. Moreover, since $G_f$ is affine and the $G_f$-action on $X$ is proper, by Proposition 0.7 in \cite{GIT}, $\pi$ is affine. Therefore, we reduce it to the case when $X$ is an irreducible object in $\gsmcat{G_t}$.

Since $X$ is \sm, it is geometrically regular and in particular, geometrically normal and integral. Therefore, it satisfies property (N) in \cite{equi completion} (see definition 3.4 in \cite{equi completion}). Since $G_t$ is diagonalizable and connected, by Corollary 3.11 in \cite{equi completion}, $X$ is covered by $G_t$-\inv, affine open subschemes as desired.
\end{proof}

\begin{lemma}
\label{lemma same proj bundle}
Suppose $X$ is a Noetherian $G$-scheme and $\cat{S} = \oplus_{d \geq 0}\, S_d$ is a \glin{G}, graded, $\O_X$-algebra such that $\cat{S}_0 = \O_X$, $\cat{S}_1$ is a \glin{G}\ coherent sheaf over $X$ and $\cat{S}$ is locally generated by $\cat{S}_1$. If $\L$ is a sheaf in $\picard{G}{X}$, $p : P \defeq \Proj{\cat{S}} \to X$ and $p' : P' \defeq \Proj{\oplus_{d \geq 0}\, \cat{S}_d \otimes \L^{\otimes d}} \to X$ are the projections, then there is a natural iso\morp\ $\phi : P' \iso P$, commuting with $p$ and $p'$, such that 
$$\O_{P'}(1) \cong \phi^* \O_P(1) \otimes {p'}^*\L.$$
\end{lemma}

\begin{proof}
See Lemma 7.9 in Chapter II in \cite{hartshorne} or Proposition 3.3 in \cite{geo equi alg cobor}.
\end{proof}

\begin{prop}
\label{prop birat given by blowup}
Suppose $X$ and $Y$ are \girred{G}\ objects in $\gvar{G}$.

\noindent \begin{statementslist}
{\rm (1)} & If $f : X \to Y$ is a $G$-\equi, \proj, birational \morp, then there is a $G$-\inv\ closed subscheme $Z \subset Y$ such that $X$ is isomorphic to $\blowup{Y}{Z}$ (blow up of $Y$ along $Z$) and $f$ corresponds to $\pi : \blowup{Y}{Z} \to Y$. \\
{\rm (2)} & If $Y$ is \proj\ and $f : X \ratmap Y$ is a $G$-\equi, rational \morp, then there is a $G$-\inv\ closed subscheme $Z \subset X$ such that $f$ can be extended to a $G$-\equi\ \morp\ $\overline{f} : \blowup{X}{Z} \to Y$.
\end{statementslist}
\end{prop}

\begin{proof}
\noindent (1).\tab This is basically an \equi\ version of Theorem 7.17 in Chapter II in \cite{hartshorne}. By Proposition \ref{prop equi embed}, there exist a $G$-\repn\ $V$ and a $G$-\equi\ immersion $i' : X \embed \P(V)$. Then, $i \defeq (i', f)$ defines a $G$-\equi\ closed immersion $X \embed \P(V) \x Y$ such that $f = \pi_2 \circ i$. Therefore, $X \cong \Proj{\cat{S}}$ for some \glin{G}\ graded $\O_Y$-algebra $\cat{S}$. Let $\L \defeq i^* \O(1)$ and $\cat{S}' \defeq \oplus_{d \geq 0}\, f_* (\L^{\otimes d})$. By composing $i'$ with some $m$-uple embedding of $\P(V)$, for some large $m$, we may assume $\cat{S} \cong \cat{S}'$ as \glin{G}\ graded $\O_Y$-algebras.

W\withoutlog, we may assume there is a $G$-\inv\ hyperplane $H$ in $\P(V)$ that does not contain any irreducible component of $X$. So, if we consider $H \x Y$ as a $G$-\inv\ Cartier divisor on $\P(V) \x Y$, its restriction will also define a $G$-\inv\ Cartier divisor on $X$. Since $\O(1) \cong \O(H \x Y) \otimes \beta$ for some character $\beta$, we have $\L \cong \O(i^* (H \x Y)) \otimes \beta$. Notice that the sheaf associated to a $G$-\inv\ Cartier divisor can always be embedded into the sheaf of total quotient rings. Therefore, we have a $G$-\equi\ embedding $\L \embed \cat{K}_X \otimes \beta$. By Lemma \ref{lemma same proj bundle}, $\Proj{\oplus\, \cat{S}'_d} \cong \Proj{\oplus\, \cat{S}'_d \otimes (\dual{\beta})^{\otimes d}}$. So, by replacing $i'$ by $X \embed \P(V) \iso \P(V \otimes \dual{\beta})$, we may assume $\L \subset \cat{K}_X$. Hence, $f_* \L \subset f_* \cat{K}_X \cong \cat{K}_Y$, where $\cat{K}_X$, $\cat{K}_Y$ are the sheaves of total quotient rings on $X$, $Y$ respectively. 

Since $Y$ is \qproj, by Proposition \ref{prop equi embed}, there exists a very ample sheaf $\cat{M} \in \picard{G}{Y}$. By a similar argument, we may assume $\cat{M} \subset \cat{K}_Y$. Also, for a large enough $n$, $\cat{M}^n \cdot f_* \L = \cat{I}$ for some $G$-\inv\ ideal sheaf of $Y$ because $f_* \L$ is a $G$-\inv, coherent, subsheaf of $\cat{K}_Y$. Again, by Lemma \ref{lemma same proj bundle}, $\Proj{\oplus\, \cat{S}'_d} \cong \Proj{\oplus\, \cat{S}'_d \otimes \cat{M}^{\otimes nd}}$. Hence, it is enough to show $\oplus\, f_* \L^d \otimes \cat{M}^{\otimes nd} \cong \oplus\, \cat{I}^{\otimes d}$ as \glin{G}\ graded $\O_Y$-algebras. But this is true because all sheaves involved are considered as subsheaves of $\cat{K}_Y$ and tensor product becomes product.

\medskip

\noindent (2).\tab Let $U$ be a $G$-\inv\ open subscheme of $X$ such that $f|_U : U \to Y$ is a $G$-\equi\ \morp. Let $W$ be closure of the graph of $f|_U$ inside $X \x Y$. Then $W$ is a \girred{G}\ object in $\gvar{G}$. Moreover, we have a $G$-\equi, \proj, birational \morp\ $p : W \to X$ and a $G$-\equi\ \morp\ $\overline{f} : W \to Y$, which can be considered as an extension of $f$. The result then follows by applying part (1) on $p$.
\end{proof}

Suppose $X$ is a scheme in $\gvar{G}$. Denote the set of $G$-\inv\ Cartier divisors, up to $G$-\equi ly linear equivalence, by $\cartdiv{G}{X}$. If $X$ is regular in codimension 1, then we will denote the set of $G$-\inv\ Weil divisors, up to $G$-\equi ly linear equivalence, by $\weildiv{G}{X}$.

\begin{prop}
\label{prop lb structure}
Suppose $X \in \gvar{G}$ is \girred{G}.

\noindent \begin{statementslist}
{\rm (1)} & The natural map $\cartdiv{G}{X} \to \picard{G}{X}$ is injective. If we further assume $X$ to be locally factorial, then $X$ is regular in codimension 1 and $\weildiv{G}{X} \cong \cartdiv{G}{X}$. \nonumber\\
{\rm (2)} & If $\L$ is in the kernel of the forgetful map $\picard{G}{X} \to \picard{}{X}$, then 

\begin{center}$\displaystyle \L \cong \O_X f \otimes \alpha$\end{center}

for some $f \in \gsection{X}{\cat{K}^*}$ and $G$-character $\alpha$, where $\cat{K}$ is the sheaf of total quotient ring on $X$. \nonumber\\
{\rm (3)} & Any sheaf $\L \in \picard{G}{X}$ can be written as $\L \cong \O_X(D) \otimes \alpha$ for some divisor $D \in \cartdiv{G}{X}$ and $G$-character $\alpha$. \nonumber
\end{statementslist}
\end{prop}

\begin{proof}
\noindent (1). \tab Standard arguments.

\medskip

\noindent (2). \tab Ignoring its $G$-action, $\L$ is a trivial line bundle. So we may write $\L \cong \O_X t$. Suppose the $G_f$ action on $X$ is not faithful. Then there is a non-trivial subgroup $H \subseteq G_f$ which acts trivially on $X$ (but not necessarily so on $\L$). For all $h \in H$, 
$$h \cdot t = \lambda t$$
for some $\lambda \in \O(X)^*$. Since the $H$-action on $X$ is trivial and $h^n \cdot t = t$ where $n$ is the order of $H$, we have $\lambda^n = 1$. By our basic assumptions on $G, k$ in section \ref{sect notation}, $k$ contains a primitive $n$-th \rou, so $\lambda \in k^*$. In other words, the $H$-action on $\L$ is given by $\alpha \in H^*$. Since the restriction map $G^* \to H^*$ is surjective, we may lift $\alpha$ to a $G$-character. Then, the $H$-action on the $\L \otimes \dual{\alpha}$ will be trivial. Therefore, we may assume the $G_f$-action on $X$ is faithful.

Let us consider the case when $G$ is cyclic with order $n$ first. Let $g$ be a generator of $G$. Then the $G$-action on $\L$ will be uniquely determined by an element $\lambda \in \O(X)^*$ where
$$g \cdot t = \lambda t.$$
Moreover, since $g^n \cdot t = t$, we have
$$\lambda (g \cdot \lambda) (g^2 \cdot \lambda) \cdots (g^{n-1} \cdot \lambda) = 1.$$

By Proposition \ref{prop equi embed}, there exists a $G$-\repn\ $V$ such that $X \embed \P(V)$. Pick such immersion so that $\dim V = d$ is minimal. Let $W = \spec k[x_1, \cdots, x_d]$ be an affine, $G$-invariant, open subset of $\P(V)$ such that the $G$-action on $x_i$ are given by $\beta_i \in G^*$. By the minimality of $d$, we have $X \cap W \neq \emptyset$. Let $I \subset k[x_1, \cdots, x_d]$ be the ideal defining the closure $\overline{X \cap W}$ in $W$ and $A \defeq k[x_1, \cdots, x_d] / I$.

\medskip

\noindent Claim 1 : There exists a monomial $\overline{f} \in k[x_1, \cdots, x_d]$ with $G$-action given by $\beta$ such that $\beta(g) = \omega$ is a primitive $n$-th \rou.

Let $\beta \in G^*$ be a generator. Then, for all $i$, $\beta_i = \beta^{a_i}$ for a unique $0 \leq a_i \leq n-1$. If $\text{gcd}(a_1, \ldots, a_d) = m > 1$, then $g^{n/m}$ will act trivially on $k[x_1, \ldots, x_d]$, and then on $X \cap W$. That contradicts with the fact that the $G$-action on $X$ is faithful. Therefore, $\text{gcd}(a_1, \ldots, a_d) = 1$ and the result follows. \claimend

\medskip

\noindent Claim 2 : $\overline{f} \in A$ is a non-zero divisor.

If the monomial $\overline{f}$ lies in $I$, then $\overline{X \cap W} \subseteq \{ \overline{f} = 0 \}$. Since $\{ x_ i = 0\}$ is $G$-\inv\ for all $i$ and $\overline{X \cap W}$ is \girred{G}, $\overline{X \cap W} \subseteq \{ x_i = 0\}$ for some $i$, which contradicts with the fact that $d$ is minimal. Therefore, $0 \neq \overline{f} \in A$.

Since $\overline{X \cap W}$ is \girred{G}, $\overline{X \cap W} = \bigcup_{h \in G} h \cdot X_0$ where $X_0$ is irreducible. Notice that $\overline{f}$ is a non-zero divisor on $\overline{X \cap W}$ if $\overline{f} \neq 0$ on $h \cdot X_0$ for all $h \in G$. Suppose $\overline{f} = 0$ on $h \cdot X_0$ for some $h$. Then, for all $h' \in G$, 
$$0 = h' \cdot \overline{f} = \omega^i \overline{f}$$
on $(h'h) \cdot X_0$ for some $i$. That means $\overline{f} = 0$ on $(h'h) \cdot X_0$ as well. That contradicts with the fact that $0 \neq \overline{f} \in A$. \claimend

\medskip

Now, for all $0 \leq i \leq n-1$, let
$$f_i \defeq \overline{f}^i + \frac{g \cdot \overline{f}^i}{\lambda} + \frac{g^2 \cdot \overline{f}^i}{\lambda(g \cdot \lambda)} + \cdots + \frac{g^{n-1} \cdot \overline{f}^i}{\lambda(g \cdot \lambda) \cdots (g^{n-2} \cdot \lambda)}$$
as elements in $A$. Then it is clear that 
$$g \cdot f_i = \lambda f_i.$$

\medskip

\noindent Claim 3 : $f_i \in A$ is a non-zero divisor for some $0 \leq i \leq n-1$.

Suppose $f_i = 0$ for all $0 \leq i \leq n-1$. Let 
$$c_1 = 1, c_2 = \frac{1}{\lambda}, c_3 = \frac{1}{\lambda(g \cdot \lambda)}, \cdots, c_n = \frac{1}{\lambda(g \cdot \lambda) \cdots (g^{n-2} \cdot \lambda)}.$$
Then, 
\begin{eqnarray*}
0 &=& f_i \\
&=& c_1 \overline{f}^i + c_2 (g \cdot \overline{f}^i) + c_3 (g^2 \cdot \overline{f}^i) + \cdots + c_n (g^{n-1} \cdot \overline{f}^i) \\
&=& c_1 \overline{f}^i + \omega^i c_2  \overline{f}^i + \omega^{2i} c_3 \overline{f}^i + \cdots + \omega^{(n-1)i} c_n \overline{f}^i
\end{eqnarray*}
By claim 2, $\overline{f}$ is a non-zero divisor. So we have a system of equations
$$0 = c_1 + \omega^i c_2 + \omega^{2i} c_3  + \cdots + \omega^{(n-1)i} c_n $$
for all $0 \leq i \leq n-1$.

Let 
$$M \defeq \begin{pmatrix}
1 & 1 & 1 & \cdots & 1 \\
1 & \omega & \omega^2 & \cdots & \omega^{n-1} \\
1 & \omega^2 & \omega^4 & \cdots & \omega^{(n-1)2} \\
\vdots & \vdots & \vdots & \vdots & \vdots \\
1 & \omega^{n-1} & \omega^{2(n-1)} & \cdots & \omega^{(n-1)(n-1)} \\
\end{pmatrix}$$
Since $\omega$ is a primitive $n$-th \rou, we have
$$M M^t = \begin{pmatrix}
n & 0 & 0 & \cdots & 0 \\
0 & 0 & 0 & \cdots & n \\
\vdots & \vdots & \vdots & \vdots & \vdots \\
0 & 0 & n & \cdots & 0 \\
0 & n & 0 & \cdots & 0 \\
\end{pmatrix}$$
By our assumptions on $G, k$ in section \ref{sect notation}, $0 \neq n \in k$. Therefore, $0 \neq \det M \in k$.

Let $C \defeq (c_1\ c_2\ \cdots\ c_n)^t$. Since $MC = 0$, 
$$0 = (\text{adjoint of } M) M C = (\text{det} M) C$$
which draws a contraction ($c_1 = 1$). Hence, $0 \neq f_i \in A$ for some $0 \leq i \leq n-1$. Since $g \cdot f_i = \lambda f_i$ with $\lambda \in A^*$, by a similar argument as in claim 2, $f_i \in A$ is a non-zero divisor. \claimend

\medskip

Let $f \defeq f_i \in A$ given by claim 3 and consider it as an element in $\gsection{X}{\cat{K}^*}$. We can then define a map $\L \cong \O_X t \to \O_X f$ by sending $a t$ to $a f$. It is an isomorphism of invertible sheaf over $X$ because $f$ is a non-zero divisor and it is $G$-\equi\ because $g \cdot f = \lambda f$. That handles the case when $G$ is finite and cyclic.

In general, let $G = G_1 \x G_2$ where $G_1$ is a cyclic group with order $n$ and $L \to X$ be the $G$-\equi\ line bundle corresponding to $\L$. By considering $L \to X$ as a $G_1$-equivariant line bundle, we have 
$$\L \cong \O_X f \otimes \alpha$$ 
for some $f \in \gsection{X}{\cat{K}^*}$ and $\alpha \in G^*$ (since $G^* \to G_1^*$ is surjective). By replacing $\L$ by $\L \otimes \O_X f^{-1} \otimes \dual{\alpha}$, we may assume $\L \cong \O_X t$ with $g \cdot t = t$ for all $g \in G_1$.

Since $G_1$ is finite and $X$ is \qproj\ over $k$, the quotient $X/G_1$ exists as a \qproj\ scheme over $k$. Let $\pi : X \to X/G_1$ be the quotient map and $\L' \defeq \pi_* (\L)^{G_1}$ be the subsheaf of $\pi_* \L$ of $G_1$-\inv\ sections.

\medskip

\noindent Claim 4 : $\L'$ defines a $G$-\equi\ invertible sheaf over $X/G_1$. Moreover, if we denote its corresponding $G$-\equi\ line bundle over $X/G_1$ by $L'$, then $G_1$ acts trivially on $L' \to X/G_1$ and the canonical map $L \to \pi^* L'$ is an isomorphism of $G$-\equi\ line bundle.

Since $g \cdot t = t$ for all $g \in G_1$, the map $\O_X \to \L$, which sends $a$ to $at$, is an isomorphism of $G_1$-\equi\ invertible sheaves and it descends to an isomorphism $\pi_* (\O_X)^{G_1} \to \pi_* (\L)^{G_1} = \L'$. By the definition of (geometric) quotient, $\O_{X/G_1} \cong \pi_* (\O_X)^{G_1}$. Therefore, $\L'$ is, in particular, an invertible sheaf. 

It is clear that $\L'$ has a canonical $G$-action and its corresponding $G$-\equi\ line bundle $L'$ is nothing but $L/G_1$. Hence, the canonical map 
$$L \to \pi^* L' \cong \pi^* (L/G_1)$$
is an isomorphism of $G$-\equi\ line bundles over $X$. \claimend

\medskip

By claim 4 and induction on the order of $G_f$, we may assume 
$$\L' \cong \O_{X/G_1} f \otimes \alpha$$
(as $G$-\equi\ sheaves) for some $f \in \gsection{X/G_1}{\cat{K}^*}$ and $\alpha \in G^*$. Then, by claim 4, 
$$\L \cong \pi^* \L' \cong \pi^* (\O_{X/G_1} f) \otimes \alpha \cong O_X f \otimes \alpha$$
(via the map $\gsection{X/G_1}{\cat{K}^*} \to \gsection{X}{\cat{K}^*}$). That reduces to the case when $G = G_t$. The result then follows from Proposition 1 (with remark 1 and 3) of \cite{equi affine}.

\medskip

\noindent (3). \tab By Proposition \ref{prop equi embed}, there is a $G$-\equi\ immersion $i : X \embed \P(V)$ for some $G$-\repn\ $V$. So $\L$ can be expressed as the difference of two \glin{G}\ very ample sheaves over $X$ and it is enough to prove the statement on such sheaf. By Proposition 1.7 in \cite{GIT}, w\withoutlog, we may assume $\L \cong i ^* \O_{\P(V)}(1)$. We may further assume there is a $G$-\inv\ hyperplane $H \subset \P(V)$ which does not contain any irreducible component of $X$, and hence its restriction on $X$ defines a $G$-\inv\ Cartier divisor $D$.  Since $V$ can be expressed as the direct sum of 1-dimensional $G$-\repn s, $\O_{\P(V)}(1) \cong \O_{\P(V)}(H) \otimes \alpha$ for some $G$-character $\alpha$. Hence,
$$\L \cong i^*\O_{\P(V)}(1) \cong \O_{\P(V)}(H)|_X \otimes \alpha \cong \O_X(D) \otimes \alpha.$$
\end{proof}

\begin{rmk}
{\rm
For part (2) of Proposition \ref{prop lb structure}, when $G$ is not connected, the part $\O_X f$ is necessary to describe the kernel of $\picard{G}{X} \to \picard{}{X}$. For example, if $G = \genby{g}$ is a cyclic group of order 2, $X = \A^1 - \{0\} = \spec{k[x,x^{-1}]}$ with action $g \cdot x = x^{-1}$, $\L = \O_X t$ with action $g \cdot t = x t$, then $\L$ can not be given by any $G$-character because the $G$-\repn s of the fibers over the fixed points -1 and 1 are different (see Remark 1 of \cite{equi affine}). But it can be shown that $\L \cong \O_X f$ with 
$$f \defeq 1 + \frac{g \cdot 1}{x} = 1 + x^{-1} \in \gsection{X}{\cat{K}^*}.$$
}
\end{rmk}

\begin{prop}
\label{prop PBF for equi picard}
Suppose $X \in \gvar{G}$ is locally factorial, $\cat{E}$ is a \glin{G}\ locally free sheaf over $X$ with finite rank and $\pi : \P(\cat{E}) \to X$ is the induced morphism. Then the abelian group \homo
$$\picard{G}{X} \oplus \Z \to \picard{G}{\P(\cat{E})},$$
which sends $(\L, n)$ to $(\pi^* \L) \otimes \O_{\P(\cat{E})}(n)$, is an isomorphism.
\end{prop}

\begin{proof}
Suppose $Y \in \gvar{G}$ is locally factorial. Let $V$ be a $G$-\repn\ with a non-empty open subset $U$ such that $G$ acts freely on $U$ and the principal bundle quotient $U \to U/G$ exists in the category of schemes. By Proposition \ref{prop equi embed}, $Y$ is \qproj\ with a \glin{G}\ action. Then, a principal bundle quotient $Y \x U \to (Y \x U)/G$ exists in the category of scheme (Proposition 23 in \cite{equi intersection}). By the definition of equivariant Chow group in \cite{equi intersection}, 
$$CH_G^n (Y) \defeq CH^n( Y_G )$$
where $Y_G \defeq (Y \x U)/G$ (It is independent of the choices of $U, V$ as long as $V - U$ has sufficiently high codimension). By Theorem 1 in \cite{equi intersection}, we have a natural isomorphism
$$\picard{G}{Y} \cong CH_G^1(Y).$$

Let $E \to X$ be the $G$-\equi\ vector bundle corresponding to $\cat{E}$. By Lemma 1 of \cite{equi intersection}, $E_G \to X_G$ is also a vector bundle and so, $\P(E)_G \cong \P(E_G) \to X_G$ is a projective bundle. Therefore, $\P(\cat{E})_G \to X_G$ is also a projective bundle. Hence,
\begin{eqnarray*}
\picard{G}{\P(\cat{E})} &\cong& CH_G^1(\P(\cat{E})) \\
&=& CH^1(\P(\cat{E})_G) \\
&\cong& CH^1(X_G) \oplus \Z \text{\tab ($X_G$ is irreducible)} \\
&\cong& CH_G^1(X) \oplus \Z \\
&\cong& \picard{G}{X} \oplus \Z,
\end{eqnarray*}
where $(\L, n) \in \picard{G}{X} \oplus \Z$ is identified with $(\pi^* \L) \otimes \O_{\P(\cat{E})}(n)$, as desired.

\end{proof}

\begin{thm}
\label{thm splitting principle}
Suppose $\char{k} = 0$. For any \girred{G}\ $X \in$ $\gsmcat{G}$\ and \glin{G}\ locally free sheaf $\cat{E}$ over $X$ of rank $r$, there exist a $G$-equivariant morphism $f : \tilde{X} \to X$, which is the composition of a series of blow ups along $G$-invariant smooth centers, and a \glin{G}\ invertible subsheaf $\L \embed f^* \cat{E}$ over $\tilde{X}$ such that the sequence
$$0 \to \L \to f^* \cat{E} \to (f^* \cat{E}) / \L \to 0$$
is exact and $(f^* \cat{E}) / \L$ is locally free of rank $r-1$.
\end{thm}

\begin{proof}
It is a generalization of Theorem 6.3 in \cite{geo equi alg cobor}, see section 6.1 in \cite{geo equi alg cobor} for the details.
\end{proof}

\bigskip
\bigskip

\section{Definition of the \equi\ algebraic cobordism theory}
\label{sect defn of theory}

Recall the following notion from \cite{equi FGL} (definition 12.2). A $(G,F)$-\fgl\ over a commutative ring $R$ is a topological $R$-module 
$$R\{\{F\}\} \defeq R\{\{ 1, y(V^1), y(V^2), \ldots \}\}$$ 
with product, coproduct and a $G^*$-action satisfying
\begin{eqnarray}
y(V^i)y(V^j) &=& \sum_{s \geq 0}\,b^{i,j}_s\,y(V^s) \nonumber\\
l_{\ga} y(V^i) &=& \sum_{s \geq 0}\,d(\ga)^i_s\,y(V^s) \nonumber\\
\gD y(V^i) &=& \sum_{s,t \geq 0}\,f^i_{s,t}\,y(V^s) \otimes y(V^t), \nonumber
\end{eqnarray}
\noindent for some elements $b^{i,j}_s$, $d(\ga)^i_s$, $f^i_{s,t}$ in $R$, which are called structure constants, and some other natural properties. In particular, $b^{i,j}_s$ encode a commutative, associative product, $d(\ga)^i_s$ encode a $G^*$-action and $f^i_{s,t}$ encode a commutative, associative coproduct and there are various compatibilities between them. It is also worth mentioning that the choice of the flag $F$ is to let one chooses a topological basis for the ring underlying the equivariant formal group law, so that the structure can be explicitly expressed in terms of the structure constants. According to Corollary 14.3 in \cite{equi FGL}, there is a representing ring $\lazard_G(F)$ for $(G,F)$-\fgl s and it is generated, as a $\Z$-algebra, by the structure constants.

Let us first define our basic object : cycle. A cycle is an expression of the form :
$$[f : Y \to X, \L_1, \ldots, \L_s]$$
where $f : Y \to X$ is a \proj\ \morp\ in $\gvar{G}$, $Y$ is \sm\ and \girred{G}, $\L_j$ are \glin{G}\ invertible sheaves over $Y$ ($s$ can be zero and the order of $\L_j$ does not matter). More specifically, we may call it a cycle with $s$ line bundles. We define its geometric dimension to be $\dim Y$, denoted by $\geodim$. If $s = 0$, we call it a geometric cycle. There is a natural notion of iso\morp\ on such expression.

For a fixed object $X \in \gvar{G}$, we define $\bigbasicmod{G}{F}{s}{X}$ to be the free $\lazard_G(F)$-module generated by isomorphism classes of cycles with $s$ line bundles ($Y, f$ and $\L_j$ can vary) and 
$$\bigbasicmod{G}{F}{}{X} \defeq \prod_{s \geq 0} \bigbasicmod{G}{F}{s}{X}.$$ 

Unfortunately, the $\lazard_G(F)$-module $\bigbasicmod{G}{F}{}{X}$ is too big an object to handle. Instead, we will consider a submodule inside which is generated by elements of a form which will be explained below.

For a \glin{G}\ invertible sheaf $\L$ over some $Y \in \gvar{G}$, integer $i \geq 0$ and a finite subset $S$ of the set of positive integers, we define $V^{i}_{S}(\L)$ as the abbreviation for 
$$\L \otimes \alpha_1 , \L \otimes \alpha_2 , \ldots, \L \otimes \alpha_{i}$$
omitting $\L \otimes \alpha_k$ whenever $k \in S$. So, $V^{i}_{S}(\L)$ is basically the sequence of \glin{G}\ invertible sheaves given by twisting $\L$ by the characters $\alpha_1, \ldots, \alpha_i$, omitting the indices in $S$. For example, 
$$[f : Y \to X, V^{4}_{\{2,3\}}(\L)] = [f : Y \to X, \L\otimes \alpha_1, \L\otimes \alpha_4]$$
which is a cycle with 2 line bundles. We adopt the convention that 
$$[f : Y \to X, V^{i}_{S}(\L)] = 0$$ 
if $\max{S} > i$ (that is, if we omit index which is out of range). When $S$ is empty, we simply denote $V^{i}_{S}(\L)$ as $V^{i}(\L)$. 

Now, for a fixed object $X \in \gvar{G}$, the basic $\lazard_G(F)$-module, denoted by $\basicmod{G}{F}{}{X}$, that we will be working with is defined to be the submodule of $\bigbasicmod{G}{F}{}{X}$ generated by elements of the form
\begin{eqnarray}
\label{eqn basic element}
\sum_{I \geq 0} a_I [f : Y \to X, V^{i_1}_{S_1}(\L_1), V^{i_2}_{S_2}(\L_2), \ldots, V^{i_r}_{S_r}(\L_r)] 
\end{eqnarray}
\noindent where $f : Y \to X$ is a \proj\ \morp\ in $\gvar{G}$, $Y$ is \sm\ and \girred{G}, $\L_j$ are \glin{G}\ invertible sheaves over $Y$ as before, $r \geq 0$, $I$ is the multi-index $(i_1, \ldots, i_r)$, $a_I$ are elements in $\lazard_G(F)$ and $S_j$ are finite subsets of the set of positive integers (The sets $S_j$ are independent of $I$). We call an element of the form (\ref{eqn basic element}) an infinite cycle. 

For instance,
\begin{eqnarray*}
&& \sum_{i\geq 0} a_i [f : Y \to X, V^{i}_{\{1\}}(\L)] \\
&=& a_0 [f : Y \to X, V^{0}_{\{1\}}(\L)] + a_1 [f : Y \to X, V^{1}_{\{1\}}(\L)] + a_2 [f : Y \to X, V^{2}_{\{1\}}(\L)] + \cdots \\
&=& 0 + a_1 [f : Y \to X] + a_2 [f : Y \to X, \L\otimes \alpha_2] + \cdots 
\end{eqnarray*}
Notice that if we take $S_j$ to be empty sets, $a_I = 1$ if $I = (1,1,\ldots, 1)$ and zero otherwise, the infinite cycle 
\begin{eqnarray*}
&& \sum_{I \geq 0} a_I [f : Y \to X, V^{i_1}_{S_1}(\L_1), V^{i_2}_{S_2}(\L_2), \ldots, V^{i_r}_{S_r}(\L_r)] \\
&=& [f : Y \to X, V^{1}(\L_1), V^{1}(\L_2), \ldots, V^{1}(\L_r)] \\
&=& [f : Y \to X, \L_1 \otimes \alpha_1, \L_2 \otimes \alpha_1, \ldots, \L_r \otimes \alpha_1] \\
&=& [f : Y \to X, \L_1, \L_2, \ldots, \L_r]
\end{eqnarray*}
becomes a cycle. In other words, a cycle is an infinite cycle.

We will adopt the convention that
\begin{eqnarray}
&& [f : Y \disjoint Y' \to X, \L_1, \ldots, \L_r] \nonumber\\
&=& [f|_Y : Y \to X, \L_1|_Y, \ldots, \L_r|_Y] + [f|_{Y'} : Y' \to X, \L_1|_{Y'}, \ldots, \L_r|_{Y'}]. \nonumber
\end{eqnarray}

Next, we will define four basic operations in $\basicmod{G}{F}{}{-}$. For a \proj\ \morp\ $g : X \to X'$ in $\gvar{G}$, we define a push-forward 
$$g_* : \basicmod{G}{F}{}{X} \to \basicmod{G}{F}{}{X'}$$
as the restriction of the push-forward $g_* : \bigbasicmod{G}{F}{}{X} \to \bigbasicmod{G}{F}{}{X'}$ which sends $[f : Y \to X, \ldots]$ to $[g \circ f : Y \to X', \ldots]$. 

For a \sm\ \morp\ $g : X' \to X$ in $\gvar{G}$, we define a pull-back
$$g^* : \basicmod{G}{F}{}{X} \to \basicmod{G}{F}{}{X'}$$
as the restriction of the pull-back $g^* : \bigbasicmod{G}{F}{}{X} \to \bigbasicmod{G}{F}{}{X'}$ which sends $[f : Y \to X, \L_1, \ldots]$ to $[f' : Y' \to X', g'^* \L_1, \ldots]$, where $f'$, $g'$ are given by the following Cartesian square :
\squarediagramword{Y'}{Y}{X'}{X}{g'}{f'}{f}{g}

\medskip

The third operation is called infinite Chern class operator. For an object $X \in \gvar{G}$, sheaves $\L_1, \ldots, \L_r \in \picard{G}{X}$, elements $a_I \in \lazard_G(F)$ and $S_j$ as in (\ref{eqn basic element}), we define an operator 
$$\sigma = \sum_I a_I V^{i_1}_{S_1}(\L_1) \cdots V^{i_r}_{S_r}(\L_r) : \basicmod{G}{F}{}{X} \to \basicmod{G}{F}{}{X}$$
by sending $\sum_J b_J [f : Y \to X, V^{j_1}_{T_1}(\cat{M}_1), \ldots, V^{j_s}_{T_s}(\cat{M}_s)]$ to 
$$\sum_{IJ} a_I b_J [f : Y \to X, V^{j_1}_{T_1}(\cat{M}_1), \ldots, V^{j_s}_{T_s}(\cat{M}_s), V^{i_1}_{S_1}(f^*\L_1), \ldots V^{i_r}_{S_r}(f^*\L_r)]$$
where $IJ$ is the multi-index $(i_1, \ldots, i_r, j_1, \ldots, j_s)$. Here we adopt a similar convention that $\sigma = 0$ if $i_k < \max{S_k}$ for some $k$. Notice that if $r = 1$, $S_1 = \emptyset$, $a_{(i)} = \delta^i_1$ (Kronecker delta), then $\sigma = V^1(\L) = c(\L)$, which is the usual (first) Chern class operator as in \cite{universal alg cobor}.

\begin{rmk}
{\rm
If we denote the subring of $\Endo{\basicmod{G}{F}{}{X}}$ generated by all infinite Chern class operators by $\Endoinf{\basicmod{G}{F}{}{X}}$, then any infinite cycle can be written as $f_* \circ \sigma [\id_Y]$ for some \proj\ \morp\ $f : Y \to X$ in $\gvar{G}$, where $Y$ is \sm\ and \girred{G}, and $\sigma \in \Endoinf{\basicmod{G}{F}{}{Y}}$.
}
\end{rmk}

Finally, we define the external product
$$\x : \basicmod{G}{F}{}{X} \x \basicmod{G}{F}{}{X'} \to \basicmod{G}{F}{}{X \x X'}$$
by sending the pair 
$$\left( \sum_I a_I [f : Y \to X, V^{i_1}_{S_1}(\L_1), \ldots, V^{i_r}_{S_r}(\L_r)]\ ,\ \sum_J b_J [f' : Y' \to X', V^{j_1}_{T_1}(\cat{M}_1), \ldots, V^{j_s}_{T_s}(\cat{M}_s)] \right)$$ 
to 
$$\sum_{IJ} a_I b_J [f \x f' : Y \x Y' \to X \x X', V^{i_1}_{S_1}(\pi_1^* \L_1), \ldots, V^{i_r}_{S_r}(\pi_1^* \L_r), V^{j_1}_{T_1}(\pi_2^* \cat{M}_1), \ldots, V^{j_s}_{T_s}(\pi_2^* \cat{M}_s)].$$

\begin{rmk}
{\rm
With this external product, $\basicmod{G}{F}{}{\pt}$ becomes a unitary, associative, commutative $\lazard_G(F)$-algebra and $\basicmod{G}{F}{}{X}$ becomes a $\basicmod{G}{F}{}{\pt}$-module. 
}
\end{rmk}

Then, our \equi\ algebraic cobordism group $\cob{}{G}{-}$ is defined to be the quotient (as $\lazard_G(F)$-modules) of $\basicmod{G}{F}{}{-}$ corresponding to imposing the following two axioms :

\medskip

\noindent \textbf{(Sect)}\tab For all \girred{G}\ $Y \in$ $\gsmcat{G}$\ and $\L \in \picard{G}{Y}$ such that there exists an invariant section $s \in \gsection{Y}{\L}^G$ that cuts out an invariant \sm\ divisor $Z$ on $Y$,
$$[\id_Y,\L] = [Z \embed Y].$$

\medskip

\noindent \textbf{(EFGL)}\tab For all \girred{G}\ $Y \in$ $\gsmcat{G}$, $i,j \geq 0$, character $\ga \in G^*$ and sheaves $\L$, $\cat{M} \in \picard{G}{Y}$,
\begin{eqnarray}
V^i(\L)V^j(\L)[\id_Y] &=& \sum_{s \geq 0}\,b^{i,j}_s\,V^s(\L)[\id_Y] \nonumber\\
V^i(\L \otimes \ga)[\id_Y] &=& \sum_{s \geq 0}\,d(\ga)^i_s\,V^s(\L)[\id_Y] \nonumber\\
V^i(\L \otimes \cat{M})[\id_Y] &=& \sum_{s,t \geq 0}\,f^i_{s,t}\,V^s(\L)V^t(\cat{M})[\id_Y]. \nonumber
\end{eqnarray}

\begin{rmk}
\label{rmk why EFGL}
{\rm
Let us briefly justify our axiom \textbf{(EFGL)}. Since we would like to define our theory $\cob{}{G}{-}$ to be the algebraic analogue of Tom Dieck's equivariant complex cobordism theory $MU_G(-)$, it would be important that the axiom \textbf{(EFGL)} holds in the theory $MU_G(-)$, in an appropriate sense (Readers may refer to section \ref{sect realization functor} for a quick review of $MU_G(-)$). 

As mentioned in section 13 of \cite{equi FGL 2}, $MU_G(\P(\cat{U}))$ (where $\cat{U}$ is the complete $G$-universe as in section \ref{sect notation}) is a $(G,F)$-\fgl\ over the cobordism ring $MU_G$, i.e.,
$$MU_G(\P(\cat{U})) \cong MU_G\{\{F\}\}.$$
Moreover, there is a canonical choice for $y(\epsilon)$ :
$$y(\epsilon) = [\P(\cat{U})^+ \stackrel{0}{\to} M(\xi) = MU(1,G)]$$ 
where $\xi \to \P(\cat{U})$ is the universal equivariant line bundle and $0$ is the zero section. Then, by definition (see section 4 in \cite{equi FGL}),
$$y(\alpha) = l_{\alpha} y(\epsilon) = [\P(\cat{U})^+ \stackrel{\otimes \alpha}{\longto} \P(\cat{U})^+ \stackrel{0}{\to} MU(1,G)]$$
for any 1-dimensional character $\alpha$.

On the other hand, the Euler class (hence, the first Chern class) of $\xi \otimes \alpha$ in $MU_G(-)$ is defined to be
$$e(\xi \otimes \alpha) = c(\xi \otimes \alpha) = [\P(\cat{U})^+ \stackrel{0}{\to} M(\xi \otimes \alpha) \stackrel{a}{\to} MU(1,G)]$$
where $a$ is the map induced by the classifying map of $\xi \otimes \alpha \to \P(\cat{U})$. Therefore, we have
$$y(\alpha) = e(\xi \otimes \alpha) = c(\xi \otimes \alpha)$$
in $MU_G(\P(\cat{U}))$. Since $y(V^i) = y(\alpha_1)y(\alpha_2) \cdots y(\alpha_i)$ (Lemma 13.2 in \cite{equi FGL}), we have $y(V^i) = V^i(\xi)$. Hence, the equation 
$$y(V^i)y(V^j) = \sum_{s \geq 0}\,b^{i,j}_s\,y(V^s)$$
in $MU_G(\P(\cat{U}))$ translates to
$$V^i(\xi) V^j(\xi) = \sum_{s \geq 0}\,b^{i,j}_s\, V^s(\xi).$$
If we denote the classifying map of an arbitrary line bundle $L \to Y$ by $cl_L : Y \to \P(\cat{U})$ and apply $cl_L^*$ to the above equation, we will have
$$V^i(L) V^j(L) = \sum_{s \geq 0}\,b^{i,j}_s\, V^s(L),$$
which is exactly our first equation in the axiom \textbf{(EFGL)}. Similar facts hold for the other two equations in the axiom \textbf{(EFGL)}.
}
\end{rmk}

\begin{rmk}
{\rm
To be more precise, we need to close our relations \wrt\ the four basic operations. In other words, to impose the \textbf{(Sect)} axiom, we need to quotient out the $\lazard_G(F)$ submodule generated by elements of the form
$$f_* \circ \gs \circ g^*\, ([\id_T,\L] - [Z \embed T])$$
where $f : Y \to U$ is \proj, $T \in$ $\gsmcat{G}$\ is \girred{G}, $g : Y \to T$ is \sm, $\gs \in \Endoinf{\basicmod{G}{F}{}{Y}}$ and $\L \in \picard{G}{T}$ is a sheaf such that there exists an invariant section $s \in {\rm H}^0(T,\L)^G$ that cuts out an invariant \sm\ divisor $Z$ on $T$ (see subsection 2.1.3 in \cite{universal alg cobor} for details). The fact that this submodule is closed under the four basic operations follows easily from the basic properties \textbf{(A1)-(A8)} in the coming section.
}
\end{rmk}

\begin{rmk}
\label{rmk why drop dim axiom}
{\rm
Suppose one defines $\cob{}{G}{-}$ by imposing the \textbf{(Dim)}, \textbf{(Sect)} and \textbf{(EFGL)} axioms on $\lazard_G(F) \otimes_{\Z} Z^G(-)$, where $Z^G(X)$ is the free abelian group generated by isomorphism classes of elements of the form $(f : Y \to X, \L_1, \ldots, \L_r)$. Then the canonical map $\lazard_G(F) \to \cob{}{G}{\pt}$, which sends $a$ to $a\, [\id_{\pt}]$, will send the Euler class $e(\alpha) \defeq d(\alpha)^1_0$ to 
$$e(\alpha)[\id_{\pt}] = V^1(\alpha)[\id_{\pt}] = [\id_{\pt}, \alpha]$$
(by Proposition \ref{prop Chern class basic}), which is equal to zero by the \textbf{(Dim)} axiom. Therefore, the canonical map $\lazard_G(F) \to \cob{}{G}{\pt}$ will not be injective as we want.
}
\end{rmk}

\begin{rmk}
\label{rmk basic defn}
{\rm
If $X_j$ are the \girred{G}\ components of $X$, then there is a canonical surjective map
$$\oplus_j\, {i_j}_* : \oplus_j\, \basicmod{G}{F}{}{X_j} \to \basicmod{G}{F}{}{X}$$
where $i_j : X_j \embed X$ are the immersions. Moreover, it respects \textbf{(Sect)} and \textbf{(EFGL)}. Therefore, it defines a canonical surjective map
$$\oplus_j\, {i_j}_* : \oplus_j\, \cob{}{G}{X_j} \to \cob{}{G}{X}.$$
Also, if $X_j$ are disjoint, this map becomes an iso\morp.
}
\end{rmk}

Although it may seem that our definition of $\cob{}{G}{-}$ depends on the choice of the complete $G$-flag $F$, we will see in section \ref{sect flag indep} that our theory is actually independent of this choice.

\bigskip
\bigskip

\section{Basic properties}
\label{sect basic properties}

In our \equi\ algebraic cobordism theory $\cob{}{G}{-}$, following the notation in \cite{universal alg cobor}, we also have basic properties \textbf{(A1) - (A8)}. For properties involving (first) Chern class operators, we will consider infinite Chern class operators instead. In the following list of properties, all objects and \morp s are assumed to be in $\gvar{G}$. 

For a \morp\ $f : X' \to X$ in $\gvar{G}$ and an element $\sigma = \sum_I a_I V^{i_1}_{S_1}(\L_1) \cdots V^{i_r}_{S_r}(\L_r)$ in $\Endoinf{\cob{}{G}{X}}$, we define the ``pull-back'' of $\sigma$ via $f$ as
$$\sigma^f \defeq \sum_I a_I V^{i_1}_{S_1}(f^* \L_1) \cdots V^{i_r}_{S_r}(f^* \L_r),$$
as an element in $\Endoinf{\cob{}{G}{X'}}$.

\medskip

\noindent \begin{tabular}[t]{p{30pt}p{365pt}}
\textbf{(A1)} & If $f : X \to X'$ and $g : X' \to X''$ are both \sm, then
\begin{center}
$(g \circ f)^* = f^* \circ g^*$
\end{center}
and $(\id_X)^*$ is the identity map. \nonumber
\end{tabular}

\noindent \begin{tabular}[t]{p{30pt}p{365pt}}
\textbf{(A2)} & Suppose $f : X \to Z$ is \proj, $g : Y \to Z$ is \sm\ and $f'$, $g'$ are given by the following Cartesian square :
\squarediagramword{X \x_Z Y}{X}{Y}{Z}{g'}{f'}{f}{g}

\medskip

\noindent Then the object $X \x_Z Y$ is in $\gvar{G}$\ and we have $g^* \circ f_* = f'_* \circ g'^*$. \nonumber
\end{tabular}

\noindent \begin{tabular}[t]{p{30pt}p{365pt}}
\textbf{(A3)} & If $f : X \to X'$ is \proj\ and $\sigma$ is an element in $\Endoinf{\cob{}{G}{X'}}$, then
\begin{center}
$f_* \circ \sigma^f = \sigma \circ f_*$.
\end{center} \nonumber
\end{tabular}

\noindent \begin{tabular}[t]{p{30pt}p{365pt}}
\textbf{(A4)} & If $f : X \to X'$ is \sm\ and $\sigma$ is an element in $\Endoinf{\cob{}{G}{X'}}$, then
\begin{center}
$f^* \circ \sigma = \sigma^f \circ f^*.$
\end{center} \nonumber
\end{tabular}

\noindent \begin{tabular}[t]{p{30pt}p{365pt}}
\textbf{(A5)} & If $\sigma$, $\sigma'$ are both in $\Endoinf{\cob{}{G}{X}}$, then
\begin{center}
$\sigma \circ \sigma' = \sigma' \circ \sigma.$
\end{center} \nonumber
\end{tabular}

\noindent \begin{tabular}[t]{p{30pt}p{365pt}}
\textbf{(A6)} & If $f$, $g$ are both \proj, then
\begin{center}
$\x \circ (f_* \x g_*) = (f \x g)_* \circ \x.$
\end{center} \nonumber
\end{tabular}

\noindent \begin{tabular}[t]{p{30pt}p{365pt}}
\textbf{(A7)} & If $f$, $g$ are both \sm, then
\begin{center}
$\x \circ (f^* \x g^*) = (f \x g)^* \circ \x.$
\end{center} \nonumber
\end{tabular}

\noindent \begin{tabular}[t]{p{30pt}p{365pt}}
\textbf{(A8)} & Suppose $x$, $y$ are elements in $\cob{}{G}{X}$, $\cob{}{G}{X'}$ respectively and $\sigma$ is an element in $\Endoinf{\cob{}{G}{X}}$. Then we have
\begin{center}
$\sigma(x) \x y = \sigma^{\pi_1}(x \x y)$
\end{center} 
where $\pi_1 : X \x X' \to X$ is the projection. \nonumber
\end{tabular}

\begin{proof}
The fact that $X \x_Z Y$ is reduced in \textbf{(A2)} follows from Proposition \ref{prop stay reduced}. Everything else can be easily derived from the definitions, similar to \cite{universal alg cobor}.
\end{proof}

In addition to the list above, we also have a number of basic facts for computational purpose. In particular, the double point relation (see definition 0.1 in \cite{alg cobor by DPR}), the blow up relation (see lemma 5.1 in \cite{alg cobor by DPR}) and the extended double point relation (see lemma 5.2 in \cite{alg cobor by DPR}) hold in our theory $\cob{}{G}{-}$. 

For a $G$-character $\alpha$, we will call the element $e(\alpha) \defeq d(\alpha)^1_0 \in \lazard_G(F)$ the Euler class of $\alpha$.  These Euler classes are some very special elements in $\lazard_G(F)$ and we will see in section \ref{sect K theory}\ and \ref{sect realization functor}\ that they correspond exactly to the Euler classes in the \equi\ K-theory and Tom Dieck's equivariant cobordism theory.

\newpage

\begin{prop}
\label{prop Chern class basic}
As an operator on $\cob{}{G}{X}$, 

\noindent \begin{statementslist}
{\rm (1)} & $c(\O_X) = 0.$ \nonumber\\
{\rm (2)} & For any $G$-character $\ga$, we have $c(\ga) = e(\ga)$. \nonumber
\end{statementslist}
\end{prop}

\begin{proof}
Part (1) follows from the \textbf{(Sect)} axiom. For part (2), by definition, 
$$c(\ga) = V^1(\O_X \otimes \ga) = \sum_{s \geq 0}\,d(\ga)^1_s\,V^s(\O_X)$$
by the \textbf{(EFGL)} axiom. Therefore,
\begin{eqnarray*}
c(\ga) &=&  d(\ga)^1_0 + d(\ga)^1_1\,V^1(\O_X) + d(\ga)^1_2\,V^2(\O_X) + \cdots \\
&=& e(\ga) + d(\ga)^1_1\,c(\O_X) + d(\ga)^1_2\,c(\O_X)c(\O_X\otimes \ga_2) + \cdots
\end{eqnarray*}
which is equal to $e(\alpha)$ by part (1).
\end{proof}

\begin{rmk}
\label{rmk FGL}
\rm{
By Lemma 16.7 in \cite{equi FGL}, $f^1_{i,0} = f^1_{0,i} = \delta^i_1$. That means 
$$c(\L \otimes \cat{M}) = c(\L) + c(\cat{M}) + \sum_{s,t \geq 1}\,f^1_{s,t}\,V^s(\L)V^t(\cat{M}).$$
Moreover, the representing ring $\lazard_G(F)$, as a $\Z$-algebra, is generated by $f^1_{s,t}$ and the Euler classes $e(\ga)$ by Theorem 16.1 in \cite{equi FGL}.
}
\end{rmk}

\begin{prop}
\label{prop FGL inverse}
For all $X \in \gvar{G}$ and $\L \in \picard{G}{X}$, there exists $\gs \in \Endoinf{\cob{}{G}{X}}$ such that
$$c(\dual{\L}) = \gs \circ c(\L).$$
\end{prop}

\begin{proof}
By Proposition \ref{prop Chern class basic}\ and Remark \ref{rmk FGL}, we have
$$0 = c(\O_X) = c(\L \otimes \dual{\L}) = c(\L) +  c(\dual{\L}) + \sum_{s,t \geq 1}\,f^1_{s,t}\,V^s(\L)V^t(\dual{\L}).$$
Hence, 
$$c(\dual{\L}) = -c(\L) - \sum_{s,t}\,f^1_{s,t}\,V^s(\L)V^t(\dual{\L}) = (-1 - \sum_{s,t}\,f^1_{s,t}\,V^s_{\{1\}}(\L)V^t(\dual{\L})) \circ c(\L).$$
\end{proof}

\begin{lemma}
\label{lemma FGL remaining terms}
Suppose $Y$ is an object in $\gsmcat{G}$\ and $E_1$, $E_2$ are two invariant divisors on $Y$ such that $E_1 + E_2$ is a \rsncd. Denote the intersection $E_1 \cap E_2$ by $D$. Then, as elements in $\cob{}{G}{D}$, we have
$$\sum_{s,t \geq 1}\,f^1_{s,t}\,V^s_{\{1\}}(\O_D(E_1))\, V^t_{\{1\}}(\O_D(E_2))\,[\id_D] = -[\P_D \to D],$$
where $\P_D \defeq \P(\O_D \oplus \O_D(E_1))$.
\end{lemma}

\begin{proof}
By a similar argument as in the proof of Lemma 3.3 in \cite{alg cobor by DPR}.
\end{proof}

\begin{prop} 
\label{prop double point relation}
Suppose $A$, $B$, $C$ are invariant divisors on $Y \in$ $\gsmcat{G}$\ such that $A + B \sim C$, $C$ is disjoint from $A \cup B$ and $A + B + C$ is a \rsncd. Then, as elements in $\cob{}{G}{Y}$, we have
$$[C \embed Y] = [A \embed Y] + [B \embed Y] - [\P(\O_D \oplus \O_D(A)) \to D \embed Y],$$
where $D \defeq A \cap B$.
\end{prop}

\begin{proof}
W\withoutlog, we may assume $Y$ to be \girred{G}. Then we have
\begin{eqnarray}
[C \embed Y] &=& c(\O(C))[\id_Y] \nonumber\\
&=& c(\O(A + B))[\id_Y] \nonumber\\
&=& c(\O(A))[\id_Y] + c(\O(B))[\id_Y] + \sum_{s,t \geq 1}\,f^1_{s,t}\,V^s(\O(A))V^t(\O(B))[\id_Y] \nonumber
\end{eqnarray}
by Remarks \ref{rmk FGL}. Therefore,
\begin{eqnarray}
[C \embed Y] &=& [A \embed Y] + [B \embed Y] + \sum_{s,t \geq 1}\,f^1_{s,t}\,V^s_{\{1\}}(\O(A))V^t_{\{1\}}(\O(B))\, [D \embed Y] \nonumber\\
&=& [A \embed Y] + [B \embed Y] + i_* ( \sum_{s,t \geq 1}\,f^1_{s,t}\,V^s_{\{1\}}(\O_D(A))V^t_{\{1\}}(\O_D(B))\, [\id_D] ) \nonumber
\end{eqnarray}
where $i : D \embed Y$. The result then follows from Lemma \ref{lemma FGL remaining terms}.
\end{proof}

\begin{prop}
\label{prop blow up relation}
Suppose $Z$ is an \inv\ closed subscheme of $Y$ such that $Z$, $Y$ are both in $\gsmcat{G}$. Then, as elements in $\cob{}{G}{Y}$, we have
$$[\blowup{Y}{Z} \to Y] - [\id_Y] = -[\P_1 \to Z \embed Y] + [\P_2 \to Z \embed Y],$$
where $\P_1 \defeq \P(\O_Z \oplus \dual{\nbundle{Z}{Y}}) \to Z$ and $\P_2 \defeq \P(\O \oplus \O(1)) \to \P(\dual{\nbundle{Z}{Y}}) \to Z$ with $\O(1)$ being the tautological line bundle over $\P(\dual{\nbundle{Z}{Y}})$.
\end{prop}

\begin{proof}
By the same argument as in the proof of Lemma 5.1 in \cite{alg cobor by DPR}.
\end{proof}

\begin{prop}
\label{prop extended double point relation}
Suppose $A$, $B$, $C$ are invariant divisors on $Y \in$ $\gsmcat{G}$\ such that $A + B \sim C$ and $A + B + C$ is a \rsncd. Then, as elements in $\cob{}{G}{Y}$, we have
$$[C \embed Y] = [A \embed Y] + [B \embed Y] - [\P_1 \to Y] + [\P_2 \to Y] - [\P_3 \to Y],$$
where $D \defeq A \cap B$,\tab $E \defeq A \cap B \cap C$ and
\begin{eqnarray}
\P_1 &\defeq & \P(\O_D \oplus \O_D(A)) \to D, \nonumber\\
\P_2 &\defeq & \P(\O \oplus \O(1)) \to \P(\O_E(-B) \oplus \O_E(-C)) \to E, \nonumber\\
\P_3 &\defeq & \P(\O_E \oplus \O_E(-B) \oplus \O_E(-C)) \to E \nonumber
\end{eqnarray}
with $\O(1)$ being the tautological line bundle over $\P(\O_E(-B) \oplus \O_E(-C))$.
\end{prop}

\begin{proof}
By the same argument as in the proof of Lemma 5.2 in \cite{alg cobor by DPR}.
\end{proof}

\begin{rmk}
{\rm
From now on, we will refer to Proposition \ref{prop double point relation}, \ref{prop blow up relation} and \ref{prop extended double point relation} as the double point relation, the blow up relation and the extended double point relation respectively.
}
\end{rmk}

\bigskip
\bigskip

\section{The \equi\ algebraic cobordism ring of a point}
\label{sect cobor ring}

In this section, we will show that the \equi\ algebraic cobordism group $\cob{}{G}{X}$, as a $\lazard_G(F)$-module, is generated by geometric cycles. Moreover, we will show that the canonical $\lazard_G(F)$-algebra homo\morp 
$$\lazard_G(F) \to \cob{}{G}{\pt}$$
is surjective. Since we need to employ the embedded desingularization theorem (see \cite{embed desing thm}) and the weak factorization theorem (Theorem 0.3.1 in \cite{weak factor thm}), we will assume $\char{k} = 0$ throughout this section.

First of all, we have the following result, which is an analogue of the \textbf{(Nilp)} axiom in \cite{universal alg cobor} (see Remark 2.2.3 in \cite{universal alg cobor}).

\begin{prop}
\label{prop Nilp axiom}
Suppose $\char{k} = 0$. For any \girred{G}\ $Y \in$ $\gsmcat{G}$, $\L \in \picard{G}{Y}$ and finite set $S$ as in (\ref{eqn basic element}), 
$$V^n_S(\L)[\id_Y] = 0$$
as elements in $\cob{}{G}{Y}$, for sufficiently large $n$.
\end{prop}

\begin{proof}
We will proceed by induction on $\dim Y$. If $\dim Y = 0$, then, by Proposition \ref{prop lb structure}, $\L \cong \beta$ for some character $\beta$. Take $N$ to be an integer such that $N > \max{S}$ and $\alpha_N \cong \dual{\beta}$. By definition,
$$V^n_S(\L)[\id_Y] = V^{N-1}_S(\L) \circ c(\L \otimes \alpha_N) \circ \cdots \circ c(\L \otimes \alpha_n) [\id_Y] = 0$$
because $\L \otimes \alpha_N \cong \O_Y$.

Suppose $\dim Y > 0$. By Proposition \ref{prop lb structure}, $\L \cong \O(\sum_i \pm D_i) \otimes \beta$ for some \inv\ $G$-prime divisors $D_i$ and character $\beta$. Apply the embedded desingularization theorem on $\cup_i D_i \embed Y$, we got a map $\pi : \tilde{Y} \to Y$ which is the composition of a series of blow ups along \inv\ \sm\ centers such that the strict transforms of $D_i$, denoted by $\stricttransform{D_i}$, are \sm. For simplicity, assume $\pi$ is given by a single blow up along $Z \subset Y$. By the blow up relation,
$$[\tilde{Y} \to Y] - [\id_Y] = -[\P_1 \to Z \to Y] + [\P_2 \to Z \to Y].$$
Apply $V^n_S(\L)$ on both sides, we have
\begin{eqnarray}
&& \pi_* \circ V^n_S(\pi^* \L)[\id_{\tilde{Y}}] - V^n_S(\L)[\id_Y] \nonumber\\
&=& - V^n_S(\L)[\P_1 \to Z \to Y] + V^n_S(\L)[\P_2 \to Z \to Y] \nonumber\\
&=& - i_* \circ {p_1}_* \circ p_1^* \circ V^n_S(\L|_Z)[\id_Z] + i_* \circ {p_2}_* \circ p_2^* \circ V^n_S(\L|_Z)[\id_Z] \nonumber
\end{eqnarray}
where $i : Z \embed Y$ is the immersion and $p_i : \P_i \to Z$ are the projections. By the induction assumption, it is enough to consider $V^n_S(\pi^* \L)[\id_{\tilde{Y}}]$. Since $\pi^* \L \cong \O(\sum_i \pm \stricttransform{D_i} + \sum_j \pm E_j) \otimes \beta$ where $E_j$ are the strict transforms of the exceptional divisors, which are also \inv\ and \sm, w\withoutlog, we may assume $D_i$ are \sm.

By the induction assumption, there exist integers $N_i$ such that $V^n_S(\L|_{D_i})[\id_{D_i}] = 0$ for all $n \geq N_i$. Now take $N$ to be an integer which is greater than $\max{S}$ and all $N_i$, and also $\alpha_N \cong \dual{\beta}$. Then, for all $n \geq N$,
\begin{eqnarray}
V^n_S(\L)[\id_Y] &=& V^n_S(\O(\sum_i \pm D_i) \otimes \beta)[\id_Y] \nonumber\\
&=& V^{N-1}_S(\L) \circ c(\O(\sum_i \pm D_i)) \circ \sigma [\id_Y] \nonumber\\
&& \text{for some $\sigma \in \Endoinf{\cob{}{G}{Y}}$} \nonumber\\
&=& V^{N-1}_S(\L) \circ \sum_i \sigma_i \circ c(\O(D_i))[\id_Y], \nonumber
\end{eqnarray}
for some $\sigma_i$, by Remark \ref{rmk FGL} and Proposition \ref{prop FGL inverse}. Hence,
$$V^n_S(\L)[\id_Y] = V^{N-1}_S(\L) \circ \sum_i \sigma_i [D_i \embed Y] = 0$$
because $N-1 \geq N_i$.
\end{proof}

\begin{rmk}
\label{rmk nilpnum}
{\rm
Suppose, for each \girred{G}\ $Y \in$ $\gsmcat{G}$\ and $\L \in \picard{G}{Y}$ (up to isomorphism), we fix a choice of $G$-prime divisors $D_i$ and character $\beta$ such that $\L \cong \O(\sum_i \pm D_i) \otimes \beta$. We pick $D_i$ to be \sm\ if possible. If not, we further fix a choice of centers $Z_j$ while applying the embedded desingularization theorem on $\cup_i D_i \embed Y$. 

If $\dim Y = 0$, we define a number $\nilpnum{Y,\L,S}$ to be the minimum positive integer $n$ such that $\alpha_n = \dual{\beta}$ and $n > \max S$. If $\dim Y > 0$, then we define $\nilpnum{Y,\L,S}$ to be the minimum positive integer $n$ such that $\alpha_n = \dual{\beta}$ and
$$n > \nilpnum{Z_j, \L |_{Z_j}, S}, \nilpnum{\stricttransform{D_i}, \L |_{\stricttransform{D_i}}, S}, \nilpnum{E_k, \L |_{E_k}, S}$$
where $Z_j$ are the centers of desingularizations, $\stricttransform{D_i}$ are the strict transforms of $D_i$ and $E_k$ are the exceptional divisors (well-defined because $\dim D_i, Z_j, E_k < \dim Y$).

As suggested by the proof, for any \girred{G}\ $Y \in$ $\gsmcat{G}$, $\L \in \picard{G}{Y}$ and finite set $S$ as in (\ref{eqn basic element}), we have
$$V^{\nilpnum{Y,\L, S}}_S(\L)[\id_Y] = 0$$ 
as elements in $\cob{}{G}{Y}$.
}
\end{rmk}

\begin{cor}
\label{cor gen by cycle}
Suppose $\char{k} = 0$. For all $X \in \gvar{G}$, the $\lazard_G(F)$-module $\cob{}{G}{X}$ is generated by cycles.
\end{cor}

\begin{proof}
The module $\cob{}{G}{X}$ is generated by infinite cycles, which can be written as $f_* \circ \sigma [\id_Y]$ for some $\sigma \in \Endoinf{\cob{}{G}{Y}}$. By Proposition \ref{prop Nilp axiom}, it is indeed a finite sum.
\end{proof}

Next, we will employ a technique called ``reduction of tower'', which is similar to the technique discussed in section 6.3 in \cite{geo equi alg cobor}. In spite of the assumption that $G$ is a finite abelian group scheme in section 6.3 in \cite{geo equi alg cobor}, most arguments work for our more general setup too. Therefore, we will not give separate proofs here. Recall the following definitions from \cite{geo equi alg cobor}.

\begin{defn}
\label{defn ad tower and sheaf}
{\rm
Suppose $Y$ is an object in $\gsmcat{G}$. A \morp\ $\P \to Y$ in $\gsmcat{G}$\ is called a \qadtower\ over $Y$ with length $n$ if it can be factored into
$$\P = \P_n \to \P_{n-1} \to \cdots \to \P_1 \to \P_0 = Y$$  
such that, for all $0 \leq i \leq n-1$, $\P_{i+1} = \P(\cat{E}_i)$ where $\cat{E}_i$ is the direct sum of sheaves which is either the pull-back of a \glin{G}\ locally free sheaves over $Y$, or the pull back of $\O_{\P_j}(m)$ for some integer $m$ and $1 \leq j \leq i$. Note that $\cat{E}_i$ is not necessarily the direct sum of invertible sheaves.

A sheaf $\L \in \picard{G}{Y}$ is called admissible if there exist \sm, $G$-prime divisors $D_1, \ldots, D_k$ on $Y$ and character $\beta$ such that

\begin{center}
$\L \cong \O_Y(\sum_{i = 1}^k m_i D_i) \otimes \beta$
\end{center}

\noindent for some integers $m_i$. Denote the subgroup of $\picard{G}{Y}$ generated by \adinvsh ves by $\apicard{G}{Y}$. Also, define the group of \adinvsh ves over $\P_i$ by
$$\apicard{G}{\P_i} \defeq \apicard{G}{Y} + \Z \O_{\P_1}(1) + \cdots + \Z \O_{\P_i}(1).$$ 
We then call a \qadtower\ $\P \to Y$ admissible if all sheaves involved in the construction are \adinvsh ves.
}\end{defn}

\begin{lemma}
\label{lemma tower twisting}
Suppose $Y \in$ $\gsmcat{G}$\ is \girred{G}\ and $D$ is a \sm\ $G$-prime divisor on $Y$. Furthermore, suppose $\P \to Y$ is an \adtower\ with length $n$ and $(i+1)$-th level $\P_{i+1} = \P(\oplus_{j=1}^r \L_j)$, and $\cat{M}_1, \ldots, \cat{M}_s$ are sheaves in $\picard{G}{\P}$. Then there exist an \adtower\ $\P' \to Y$ with length $n$, \qadtower s $Q_0$, $Q_1$, $Q_2$, $Q_3 \to D$ and \glin{G}\ invertible sheaves $\cat{M}_k'$, $\cat{M}^j_k$ such that
$$\P' = \P_n' \to \cdots \to \P_{i+1}' \to \P_i \to \cdots \to \P_0 = Y$$
where $\P_{i+1}' = \P((\oplus_{j=1}^{r-1} \L_j) \oplus \L_r(D))$, $\dim \P = \dim \P' = \dim Q_j$ and we have the following equality in $\cob{}{G}{Y}$ :
\begin{eqnarray}
&& [\P' \to Y, \cat{M}_1', \ldots, \cat{M}_s'] - [\P \to Y, \cat{M}_1, \ldots, \cat{M}_s] \nonumber\\
&=& [Q_0 \to D \embed Y, \cat{M}^0_1, \ldots, \cat{M}^0_s] - [Q_1 \to D \embed Y, \cat{M}^1_1, \ldots, \cat{M}^1_s] \nonumber\\
&& +\ [Q_2 \to D \embed Y, \cat{M}^2_1, \ldots, \cat{M}^2_s] - [Q_3 \to D \embed Y, \cat{M}^3_1, \ldots, \cat{M}^3_s]. \nonumber
\end{eqnarray}
\end{lemma}

\begin{proof}
Since this result is very similar to Lemma 6.10 in \cite{geo equi alg cobor}, we will only give a sketch of proof here. First of all, we construct an \adtower\ $\hat{\P} \to Y$ with length $n$ by defining $\hat{\P}_{i+1} \defeq \P((\oplus_{j=1}^r \L_j) \oplus \L_r(D))$ and all higher levels are constructed in the same manner as $\P$. We then have $\P \embed \hat{\P}$. 

By Proposition \ref{prop PBF for equi picard} and the fact that $\O_{\hat{\P}_j}(1) |_{\P_j} \cong \O_{\P_j}(1)$ for all $0 \leq j \leq n$, the map $\picard{G}{\hat{\P}} \to \picard{G}{\P}$ is surjective. So there are \glin{G}\ invertible sheaves $\hat{\cat{M}_1}, \ldots, \hat{\cat{M}_s}$ over $\hat{\P}$ that extends $\cat{M}_1, \ldots, \cat{M}_s$ respectively. Then, we construct the \adtower\ $\P' \to Y$ and the \qadtower\ $Q_0 \to D \embed Y$ by restricting $\hat{\P}$ via $\P'_{i+1} \embed \hat{\P}_{i+1}$ and  $D \embed Y$ respectively.
 
By Lemma 6.9 in \cite{geo equi alg cobor} (still holds for our group $G$), we have $Q_0 + \P \sim \P'$ as invariant divisors on $\hat{\P}$ and the sum of them is a \rsncd. By the extended double point relation, we have
\begin{eqnarray}
\label{eqn 29}
[\P' \embed \hat{\P}] &=& [Q_0 \embed \hat{\P}] + [\P \embed \hat{\P}] - [(Q_0 \cap \P) \x_{\hat{\P}} P^1 \to \hat{\P}] \\
&& +\ [(Q_0 \cap \P \cap \P') \x_{\hat{\P}} P^2 \to \hat{\P}] - [(Q_0 \cap \P \cap \P') \x_{\hat{\P}} P^3 \to \hat{\P}]  \nonumber
\end{eqnarray}
where 
\begin{eqnarray}
P^1 &\defeq& \P(\O \oplus \O(Q_0)) \to \hat{\P}, \nonumber\\
P^2 &\defeq& \P(\O \oplus \O(1)) \to \P(\O(-\P) \oplus \O(-\P')) \to \hat{\P}, \nonumber\\
P^3 &\defeq& \P(\O \oplus \O(-\P) \oplus \O(-\P')) \to \hat{\P}. \nonumber
\end{eqnarray}
We then denote $(Q_0 \cap \P) \x_{\hat{\P}} P^1$, $(Q_0 \cap \P \cap \P') \x_{\hat{\P}} P^2$ and $(Q_0 \cap \P \cap \P') \x_{\hat{\P}} P^3$ by $Q_1$, $Q_2$ and $Q_3$ respectively. They are all \qadtower s over $D$ (by Lemma 6.9 in \cite{geo equi alg cobor}). Hence, the result follows by applying (first) Chern class operators $c(\hat{\cat{M}_1}), \ldots, c(\hat{\cat{M}_s})$ on equation (\ref{eqn 29}) and pushing it down to $\cob{}{G}{Y}$.
\end{proof}

\begin{rmk}
\label{rmk twisting tower same structure}
{\rm
It can be seen from the proof that the only difference on the structures of $\P$ and $\P'$ are at the level $i+1$, i.e., the \glin{G}\ invertible sheaves used in the definitions of $\P_j$ and $\P'_{j}$ are the same (by identifying $\O_{\P_k}(m)$ and $\O_{\P'_k}(m)$) whenever $j \neq i+1$.
}
\end{rmk}

For simplicity of notation, for the rest of this paper, we will write $[Y, \L_1, \ldots, \L_r]$ for $[Y \to \pt, \L_1, \ldots, \L_r]$.

\begin{defn}
{\rm
Define $\cob{}{G}{\pt}'$ to be the $\lazard_G(F)$-subalgebra of $\cob{}{G}{\pt}$ generated by elements of the form 
$$[\P, \L_1, \ldots,\L_r]$$
where $\P$ is an \adtower\ over $\pt$.
}\end{defn}

\begin{prop}
\label{prop tower reduction}
Suppose $Y \in$ $\gsmcat{G}$\ is \girred{G}, $\P \to Y$ is a \qadtower\ and $\cat{M}_1, \ldots, \cat{M}_s$ are sheaves in $\picard{G}{\P}$. Then there exist elements $x_i \in \cob{}{G}{\pt}'$, \proj\ \morp s $Y_i \to Y$ in $\gsmcat{G}$\ with $\dim Y_i \leq \dim Y$ and sheaves $\cat{M}^i_j$ such that
$$[\P \to Y, \cat{M}_1, \ldots, \cat{M}_s] = \sum_i x_i\, [Y_i \to Y, \cat{M}^i_1, \ldots]$$
as elements in $\cob{}{G}{Y}$. Moreover, $\dim \P = \geodim x_i + \dim Y_i$. If there is no invertible sheaf on the left hand side, i.e., $s = 0$, then the same result holds with no invertible sheaves on the right hand side.
\end{prop}

\begin{proof}
Here is a sketch of the proof (See Proposition 6.17 in \cite{geo equi alg cobor} for details). We will prove the statement by induction on the dimension of $Y$. We will handle the induction step first. Suppose $\dim Y \geq 1$. Let $\cob{}{G}{Y}'$ be the subgroup of $\cob{}{G}{Y}$ generated by elements of the form
$$[\P' \to Y' \to Y, \cat{M}_1', \ldots, \cat{M}_s']$$
where $Y' \in$ $\gsmcat{G}$\ is \girred{G}\ with dimension less than $\dim Y$, $\dim \P = \dim \P'$ and $\P' \to Y'$ is a \qadtower. So, elements in $\cob{}{G}{Y}'$ will be handled by the induction assumption. Let $\P \to Y$ be a \qadtower\ with length $n$. If $n = 0$, then we are done. Suppose $n \geq 1$.

\medskip

\noindent Step 1 : Reduction to a \qadtower\ constructed only by \glin{G}\ invertible sheaves.

Suppose $\gp : \tilde{Y} \to Y$ is the composition of a series of blow up along \sm\ \inv\ centers and let $\tilde{\P} \defeq \P \x_Y \tilde{Y}$. By the blow up relation, it can be shown that the difference
$$\gp_*[\tilde{\P} \to \tilde{Y}, \gp'^*\cat{M}_1, \ldots, \gp'^*\cat{M}_s] - [\P \to Y, \cat{M}_1, \ldots, \cat{M}_s],$$
where $\pi' : \tilde{\P} \to \P$, lies inside $\cob{}{G}{Y}'$. Therefore, we may blow up $Y$ along any \inv\ \sm\ center if necessary.

If $\P_i = \P(\cat{E}' \oplus \cat{E})$ such that rank $\cat{E} > 1$ (in particular, it comes from $Y$), then, by Theorem \ref{thm splitting principle}, we may assume we have a splitting 
$$0 \to \L \to \cat{E} \to \cat{E}/\L \to 0$$ 
of \glin{G}\ locally free sheaves over $Y$. Define 
$$\hat{\P}_i \defeq \P(\cat{E}' \oplus \cat{E} \oplus \L) \text{\tab and \tab} \P_i' \defeq \P(\cat{E}' \oplus (\cat{E}/\L) \oplus \L).$$ 
and construct towers $\hat{\P}$, $\P' \to Y$ in a similar manner as in the proof of Lemma \ref{lemma tower twisting}. Also define sheaves $\hat{\cat{M}}_i \in \picard{G}{\hat{\P}}$ and $\cat{M}'_i \in \picard{G}{\P'}$ similarly. Hence, by Lemma 6.12 in \cite{geo equi alg cobor} (still holds for our group $G$), we have
$$\P_i = \P(\cat{E}' \oplus \cat{E}) \sim \P(\cat{E}' \oplus (\cat{E}/ \L) \oplus \L) = \P'_i$$
as invariant smooth divisors on $\hat{\P}_i$, which implies that $\O_{\hat{\P}}(\P) = \O_{\hat{\P}}(\P')$ as \glin{G}\ invertible sheaves over $\hat{\P}$. Therefore,
\begin{eqnarray*}
[\P \to \hat{\P}, \cat{M}_1, \ldots, \cat{M}_s] &=& c(\O_{\hat{\P}}(\P))[\id_{\hat{\P}}, \hat{\cat{M}}_1, \ldots, \hat{\cat{M}}_s] \\ 
&=& c(\O_{\hat{\P}}(\P'))[\id_{\hat{\P}}, \hat{\cat{M}}_1, \ldots, \hat{\cat{M}}_s] \\ 
&=& [\P' \to \hat{\P}, \cat{M}_1', \ldots, \cat{M}_s']
\end{eqnarray*}
By pushing the elements forward via the map $\hat{\P} \to Y$, we have  
$$[\P \to Y, \cat{M}_1, \ldots, \cat{M}_s] = [\P' \to Y, \cat{M}_1', \ldots, \cat{M}_s']$$
as elements in $\cob{}{G}{Y}$ and the result follows by applying this argument repeatedly until all sheaves involved are of rank 1.

\medskip

\noindent Step 2 : Reduction to an \adtower.

For each $\L \in \picard{G}{Y}$ used in the construction of $\P$, by Proposition \ref{prop lb structure}, there are $G$-prime divisors $D_{\L,i}$ on $Y$ and character $\gb$ such that $\L \cong \O_Y(\sum_i \pm D_{\L,i}) \otimes \gb.$ By applying the embedded desingularization Theorem on $\cup_{\L,i}\, D_{\L,i} \embed Y$, we may assume $D_{\L,i}$ are all \sm.

\medskip

\noindent Step 3 : Reduction to an \adtower\ with $\P_1 = \P(\gp_Y^* \cat{E}_1)$ where $\cat{E}_1$ is a \glin{G}\ locally free sheaf over $\pt$.

Consider the first level $\P_1 = \P(\oplus^r_{j=1} \L_j)$. Since the sheaves $\L_j$ are admissible, they can be expressed by \inv\ \sm\ divisors on $Y$. By lemma \ref{lemma tower twisting} (and remark \ref{rmk twisting tower same structure}), for any \inv\ \sm\ divisor $D$ on $Y$, we can twist $[\P \to Y, \cat{M}_1, \ldots, \cat{M}_s]$ to $[\P' \to Y, \cat{M}_1', \ldots, \cat{M}_s']$ so that $\P_1' = \P((\oplus_{j \neq k}\ \L_j) \oplus \L_k(\pm D))$ and the difference will be given by elements inside $\cob{}{G}{Y}'$. Hence, we may assume $\L_j \cong \gb_j$ for all $j$. The result then follows by defining $\cat{E}_1 \defeq \oplus_{j=1}^r \gb_j$

\medskip

\noindent Step 4 : Finish the induction step.

By applying the argument in step 3 on all levels, we have an \adtower
$$Q = Q_n = \P(\cat{E}_n) \to \cdots \to \P(\cat{E}_1) \to Q_0 = \pt$$ 
such that $\P \cong Y \x Q.$ By Proposition \ref{prop PBF for equi picard}, for each $i$, the invertible sheaf $\cat{M}_i$ over $\P$ is isomorphic to $\gp_1^* \cat{M}_{Y,i} \otimes \gp_2^* \cat{M}_{Q,i}$ for some \glin{G}\ invertible sheaves $\cat{M}_{Y,i}$, $\cat{M}_{Q,i}$ over $Y$, $Q$ respectively. For simplicity, assume $s = 1$,
\begin{eqnarray}
&& [\P \to Y, \cat{M}] \nonumber\\
&=& \gp_* c(\cat{M})[\id_{\P}] \nonumber\\
&& \text{where $\gp : \P \to Y$} \nonumber\\
 &=& \gp_* c(\gp_1^* \cat{M}_Y \otimes \gp_2^* \cat{M}_Q)[\id_{\P}] \nonumber\\
 &=& \gp_* c(\gp_1^* \cat{M}_Y)[\id_{\P}] + \gp_* c(\gp_2^* \cat{M}_Q)[\id_{\P}] +\, \gp_* \sum_{s,t \geq 1}\,f^1_{s,t}\,V^s(\gp_1^* \cat{M}_Y) V^t(\gp_2^* \cat{M}_Q)\,[\id_{\P}], \nonumber
\end{eqnarray}
which is a finite sum by Proposition \ref{prop Nilp axiom}. Hence,
\begin{eqnarray}
&& [\P \to Y, \cat{M}] \nonumber\\
&=& [Q][\id_Y, \cat{M}_Y] + [Q, \cat{M}_Q][\id_Y] + \sum_{s,t \geq 1}\,f^1_{s,t}\, \gp_* (V^t(\cat{M}_Q)[\id_Q] \x V^s(\cat{M}_Y)[\id_Y]) \nonumber\\
&=& x_1\, [\id_Y, \cat{M}_Y] + x_2\, [\id_Y] + \sum_{s,t \geq 1}\,x_{s,t}\, V^s(\cat{M}_Y)[\id_Y] \nonumber
\end{eqnarray} 
by letting $x_1 \defeq [Q]$, $x_2 \defeq [Q, \cat{M}_Q]$ and $x_{s,t} \defeq f^1_{s,t}\, [Q, V^t(\cat{M}_Q)]$. That finishes the induction step (when $\dim Y > 0$).

\medskip

\noindent Step 5 : $\dim Y = 0$ case.

In this case, $\P \cong Y \x Q$ for some \adtower\ $Q$ over $\pt$. Apply step 4 and we are done.
\end{proof}

Our first goal in this section is to show that $\cob{}{G}{X}$, as a $\lazard_G(F)$-module, is generated by geometric cycles. To this end, we need two more technical Lemmas.

\begin{lemma}
\label{lemma div cut down dim}
Suppose $Y \in$ $\gsmcat{G}$\ is \girred{G}\ and $D$ is an invariant divisor on $Y$ which can be expressed as the sum of \sm, $G$-prime divisors on $Y$. Then $c(\O(D))[\id_Y]$ is equal to the finite sum of elements of the form $a\,[Y' \to Y, \ldots]$ where $a$ is in $\lazard_G(F)$ and $\dim Y' < \dim Y$.
\end{lemma}

\begin{proof}
Let $D = \sum_i \pm D_i$ where $D_i$ are \sm, $G$-prime divisors. By Remark \ref{rmk FGL} and Proposition \ref{prop FGL inverse}, for some $\sigma_i \in \Endoinf{\cob{}{G}{Y}}$,
$$c(\O(D))[\id_Y] = c(\O(\sum_i \pm D_i)) [\id_Y] = \sum_i \sigma_i \circ c(\O(D_i)) [\id_Y] = \sum_i {j_i}_* \circ \sigma_i^{j_i}  [\id_{D_i}]$$
where $j_i : D_i \embed Y$ are the immersions. The result then follows from Proposition \ref{prop Nilp axiom}.
\end{proof}

\begin{lemma}
\label{lemma lb structure over adtower}
Suppose 
$$\P = \P_n = \P(\cat{E}_n \oplus \L_n) \to \cdots \to \P(\cat{E}_1 \oplus \L_1) = \P_1 \to \P_0 = \pt$$
is an \adtower\ over $\pt$ with length $n$. Denote the \inv\ \sm\ divisors $\P(\cat{E}_i)$ on $\P_i$ by $H_i$. Then, for all $\L \in \picard{G}{\P}$, there exist integers $m_1, \ldots, m_n$ and character $\gb$ such that
$$\L \cong \O(m_1 H_1 + \cdots + m_n H_n) \otimes \gb.$$
\end{lemma}

\begin{proof}
By Proposition \ref{prop PBF for equi picard}, we may assume $\L \cong \O_{\P_i}(1)$. By induction on $n$, we may further assume $\L \cong \O_{\P_n}(1)$. Clearly, 
$$\O_{\P}(H_n) \cong \dual{\L_n} \otimes \O_{\P}(1).$$
Therefore, 
$$\L \cong \O_{\P}(1) \cong \O_{\P}(H_n) \otimes \L_n$$
and the result follows from the induction assumption ($\L_n \in \picard{G}{\P_{n-1}}$).
\end{proof}

For the rest of this section, for a \girred{G}\ object $Y \in$ $\gsmcat{G}$, we will say $\beta$ is a twisting character of a sheaf $\L \in \picard{G}{Y}$ if $\L \cong \O_Y(D) \otimes \beta$ for some \inv\ divisor $D$ on $Y$. By Proposition \ref{prop lb structure}, twisting character always exists. Notice that it may not be unique though. For example, suppose $G$ is a non-trivial group and $\ga$ is a non-trivial character. Let $Y = \P(\epsilon \oplus \ga)$ and $D = \P(\epsilon), D' = \P(\ga)$ be two invariant divisors of $Y$. Then, it is clear that 
$$\O_Y(D') \otimes \epsilon \cong \O_Y(1) \cong \O_Y(D) \otimes \ga.$$
In particular, $\epsilon$ and $\ga$ are both twisting characters of $\O_Y(1)$.

\begin{thm}
\label{thm gen by geo cycle}
Suppose $\char{k} = 0$ and $k$ contains a primitive $e$-th \rou, where $e$ is the exponent of $G_f$. For any $X \in \gvar{G}$, the $\lazard_G(F)$-module $\cob{}{G}{X}$ is generated by geometric cycles. 

More precisely, any element $[Y \to X, \L_1, \ldots, \L_r] \in \cob{}{G}{X}$, such that $Y$ is \girred{G}, can be expressed in the following form :
$$\left( \prod^r_{i=1} e(\beta_i) \right) \cdot [Y \to X] + \sum_j a_j\, [\P_j] [Y_j \to X] + \sum_k a_k'\, [Y'_k \to X]$$
where $\beta_i$ is a twisting character of $\L_i$, $a_j, a_k'$ are elements in $\lazard_G(F)$, $\P_j$ are \adtower s over $\pt$, $Y_j$, $Y'_k \in$ $\gsmcat{G}$\  are \girred{G}, $\dim Y_j$, $\dim Y_k' < \dim Y$ and $\dim Y = \dim \P_j + \dim Y_j$.
\end{thm}

\begin{proof}
By Corollary \ref{cor gen by cycle}, it is enough to consider cycles. Since $[f : Y \to X, \L_1, \ldots, \L_r] = f_* [\id_Y, \L_1, \ldots, \L_r]$, it is enough to show the statement on elements of the form $[\id_Y, \L_1, \ldots, \L_r]$ where $Y \in$ $\gsmcat{G}$\ is \girred{G}. We will proceed by induction on $d \defeq \dim Y$. Within this proof, for a cycle $x \in \cob{}{G}{-}$, we will say $x \equiv 0$ if $\geodim{x} < d$. Notice that any geometric cycle of geometric dimension $< d$ will be absorbed by the sum $\sum_k a_k'\, [Y'_k \to X]$. In particular, terms of the form
$$a_j\, [\P_j] [Y_j \to X]$$
with $\dim \P_j + \dim Y_j < \dim Y$ will be handled by considering   
$$a_j\, [\P_j] [Y_j \to X] = a_j\, [\P_j \x Y_j \to X].$$

If $\dim Y = 0$, then 
$$[\id_Y, \L_1, \ldots, \L_r] = [\id_Y, \gb_1, \ldots, \gb_r] = e(\gb_1) \cdots e(\gb_r)[\id_Y]$$
and we are done.

Suppose $\dim Y > 0$. Let $\L_i' \defeq \L_i \otimes \dual{\beta_i}$. Then,
\begin{eqnarray}
[\id_Y, \L_1, \ldots, \L_r] &=& c(\L_r)[\id_Y, \L_1, \ldots, \L_{r-1}] \nonumber\\
&=& c(\L_r' \otimes \beta_r)[\id_Y, \L_1, \ldots, \L_{r-1}] \nonumber\\
&=& e(\beta_r)[\id_Y, \L_1, \ldots, \L_{r-1}] + \sum_{j \geq 1} d(\beta_r)^1_j\, V^j(\L_r')[\id_Y, \L_1, \ldots, \L_{r-1}] \nonumber\\
&=& e(\beta_r)[\id_Y, \L_1, \ldots, \L_{r-1}] + \sigma \circ c(\L_r')[\id_Y] \nonumber
\end{eqnarray}
for some $\sigma \in \Endoinf{\cob{}{G}{Y}}$. Inductively, we have
$$[\id_Y, \L_1, \ldots, \L_r] = \left( \prod^r_{i=1} e(\beta_i) \right) \cdot [\id_Y] + \sum^r_{i=1} \sigma_i \circ c(\L_i')[\id_Y]$$
for some $\sigma_i$. Therefore, it suffices to prove that if $\L$ has a trivial twisting character, then
$$\sigma \circ c(\L)[\id_Y] = \sum_j a_j\, [\P_j] [Y_j \to Y] + \sum_k a_k'\, [Y'_k \to Y]$$
where $a_j$, $a_k'$, $\P_j$, $Y_j$ and $Y'_k$ are as described in the statement.

By the blow up relation, for any \girred{G}, \inv, \sm\ closed subscheme $Z \subset Y$, we have
$$[\gp : \blowup{Y}{Z} \to Y, \gp^*\L] - [\id_Y, \L] = - [p_1 : \P_1 \to Z \embed Y, p_1^*\L] + [p_2 : \P_2 \to Z \embed Y, p_2^*\L]$$
where $\P_1$, $\P_2 \to Z$ are both \qadtower s. By proposition \ref{prop tower reduction}, 
\begin{eqnarray}
\label{eqn4}
[\P_1 \to Z \embed Y, p_1^*\L] = \sum_j x_j\, [Y_j \to Z \embed Y, \ldots]
\end{eqnarray}
where $x_j \in \cob{}{G}{\pt}'$, $\dim Y_j \leq \dim Z < d$ and $d = \dim \P_1 = \geodim{x_j} + \dim Y_j$. 

Let us consider elements of the form $[\P, \cat{M}_1, \ldots, \cat{M}_s]$ where $\P$ is an \adtower\ over $\pt$ with length $n$ and dimension $d' \leq d$. By Lemma \ref{lemma lb structure over adtower}, we have 
$$\cat{M}_s \cong \O(m_1 H_1 + \cdots + m_n H_n) \otimes \gb$$ 
for some integers $m_i$, character $\gb$ and \inv\ \sm\ divisors $H_i$ on $\P_i$ ($i$-th level of $\P$). Therefore,
\begin{eqnarray}
&& [\P,\cat{M}_1, \ldots, \cat{M}_s] \nonumber\\
&=& {\gp_{\P}}_* [\id_{\P},\cat{M}_1, \ldots, \cat{M}_s] \nonumber\\
&=& {\gp_{\P}}_* \circ c(\O(m_1 H_1 + \cdots + m_n H_n) \otimes \gb)\,[\id_{\P},\cat{M}_1, \ldots, \cat{M}_{s-1}] \nonumber\\
&=& e(\gb)[\P, \cat{M}_1, \ldots, \cat{M}_{s-1}] + {\pi_{\P}}_* \circ \gs \circ c(\O(m_1 H_1 + \cdots + m_n H_n))\,[\id_{\P}], \nonumber
\end{eqnarray}
for some $\gs \in \Endoinf{\cob{}{G}{\P}}$. By Proposition \ref{prop Nilp axiom} and Lemma \ref{lemma div cut down dim}, 
$${\pi_{\P}}_* \circ \gs \circ c(\O(m_1 H_1 + \cdots + m_n H_n))\,[\id_{\P}] \equiv 0$$
mod elements with geometric dimension $< d'$ (by abuse of notation). By repeating this process, we have $[\P,\cat{M}_1, \ldots, \cat{M}_s] \equiv a\, [\P]$ for some $a \in \lazard_G(F)$. Hence, equation (\ref{eqn4}) becomes
$$[\P_1 \to Z \embed Y, p_1^*\L] \equiv \sum_j a_j\, [\P_j] [Y_j \to Y, \ldots].$$
By the induction assumption and the fact that the product of two \adtower s over $\pt$ is again an \adtower\ over $\pt$, we can replace $[Y_j \to Y, \ldots]$ by $[Y_j \to Y]$. The same equation holds for $\P_2$. Therefore, we have
$$[\blowup{Y}{Z} \to Y, \gp^*\L] - [\id_Y, \L] \equiv \sum_j a_j\, [\P_j] [Y_j \to Y]$$
where $\dim Y_j < d$ and $\dim \P_j + \dim Y_j = d$. 

By the same argument used in step 2 in the proof of Proposition \ref{prop tower reduction}, there is a map $\pi : \tilde{Y} \to Y$ given by a series of blow ups along \girred{G}, \inv, \sm\ centers such that $\pi^* \L \cong \O(\sum_i \pm D_i)$ for some \sm, $G$-prime divisors $D_i$ on $\tilde{Y}$. By Lemma \ref{lemma div cut down dim}, we have $[\id_{\tilde{Y}}, \pi^* \L] \equiv 0$. Therefore, 
$$[\id_Y, \L] = \sum_j a_j\, [\P_j] [Y_j \to Y] + \sum_k a_k'\, [Y_k' \to Y, \ldots]$$
for some $a_k'$ and $Y_k'$ as described in the statement. The result then follows by applying $\gs$ on both sides and the induction assumption.
\end{proof}

Next, we would like to show that the canonical map $\lazard_G(F) \to \cob{}{G}{\pt}$ is surjective. To that end, we need a better understanding of the function field of a \girred{G}\ object $Y \in$ $\gsmcat{G}$\ when $G$ is a split torus.

\begin{lemma}
\label{lemma fcn field for torus action}
Suppose $G$ is a split torus of rank $r$, $Y \in$ $\gsmcat{G}$\ is irreducible and the $G$-action on $Y$ is faithful. Then
$$k(Y) \cong k(Y)^G(t_1, \ldots, t_r)$$
where $k(Y)$ is the function field of $Y$, $t_i$ are algebraically independent over $k(Y)^G$ and the $G$-actions on $t_i$ are given by 1-dimensional characters.
\end{lemma}

\begin{proof}
W\withoutlog, we may assume $Y$ is \proj\ (by Proposition \ref{prop smooth extension}). By Proposition \ref{prop equi embed}, we can embed $Y$ into some $\P(V')$ with minimal $\dim \P(V')$. Take an \inv\ affine part $V = \spec{k[x_1, \ldots, x_m, y_1, \ldots, y_n]} \subset \P(V')$ where $G$ acts on $x_i$ trivially and the $G$-actions on $y_i$ are given by non-trivial 1-dimensional characters. Consider the \inv\ open subset $U \subset V$ with coordinates $x_i, y_j$ being non-zero, i.e.,
\begin{eqnarray*}
U &=& \spec{A} \\
&\defeq& \spec{k[x_1, \ldots, x_m, y_1, \ldots, y_n][x_1^{-1}, \ldots, x_m^{-1}, y_1^{-1}, \ldots, y_n^{-1}]}
\end{eqnarray*}
Since $\dim \P(V')$ is minimal, $Y \cap U$ is non-empty. Also, since the action on $Y$ is faithful, so is the action on $U$. Moreover, by Nagata's theorem, $U/G = \spec{A^G}$ exists as a variety over $k$ and so is $(Y \cap U)/G$. 

Within this proof, a ring of Laurent polynomials over $k$ such that the $G$-action on each variable is given by a 1-dimensional character will be called "of Laurent form". For example, $A$ is of Laurent form.

\medskip

\noindent Claim 1 : There are monomials $f_i(y_1, \ldots, y_n) \in A$, which are algebraically independent over $k(U/G)$, such that
$$A = A^G[f_1, \ldots, f_r][f_1^{-1}, \ldots, f_r^{-1}].$$

Consider the case when $r = 1$ first. Fix an iso\morp\ $G \cong \torus$. Then the actions on each $y_i$ has weight $a_i$. We claim that we can pick such $f_1$ with weight $\gcd(a_1, \ldots, a_n)$ and $A^G$ will be of Laurent form.

If $n = 1$, we just take $f_1 \defeq y_1$ and 
$$A^G = k[x_1, \ldots, x_m][x_1^{-1}, \ldots, x_m^{-1}]$$
is of Laurent form. 

Suppose $n > 1$. Apply the induction assumption on 
$$B = k[x_1, \ldots, x_m, y_1, \ldots, y_{n-1}][x_1^{-1}, \ldots, x_{m}^{-1}, y_1^{-1}, \ldots, y_{n-1}^{-1}],$$ 
we have 

\begin{tabular}{ll}
(a) & $B = B^G[f_1][f_1^{-1}]$ \\
(b) & $f_1$ is a monomial in terms of $y_1, \ldots, y_{n-1}$ \\
(c) & $f_1 \text{ has weight } b \defeq \gcd(a_1, \ldots, a_{n-1})$ \\ 
(d) & $B^G \text{ is of Laurent form}$ \\
(e) & $f_1, y_n \text{ are algebraically independent over  $B^G$}$ 
\end{tabular}

\noindent Thus,
$$A  = B[y_n][y_n^{-1}] = B^G[f_1, y_n][f_1^{-1}, y_n^{-1}] \subseteq A$$
Therefore, $A  = B^G[f_1, y_n][f_1^{-1}, y_n^{-1}]$.

Now, by the division algorithm, there are integers $q, r$ with $0 \leq r < a_n$ such that $b = q a_n + r$. Let $g \defeq f_1 y_n^{-q}$. Then, on one hand,
$$A  = B^G[f_1, y_n][f_1^{-1}, y_n^{-1}] \supseteq B^G[g, y_n][g^{-1}, y_n^{-1}].$$
On the other hand, since $f_1 = g y_n^q$, we have
$$A  = B^G[f_1, y_n][f_1^{-1}, y_n^{-1}] \subseteq B^G[g, y_n][g^{-1}, y_n^{-1}].$$
Notice that the weights of $f_1,y_n$ are $b,a_n$ respectively, while the weights of $g,y_n$ are $r, a_n$ respectively. By repeated applications, we got a pair of monomials $g,g'$ with weights $c,0$ respectively where $c = \gcd(a_1, \ldots, a_n)$ and $A = B^G[g,g'][g^{-1}, {g'}^{-1}].$ Since $g'$ has weight zero, it is indeed in $A^G$. So, 
$A = A^G[g][g^{-1}].$ It should also be clear that $g,g'$ are algebraically independent over $B^G$ and $A^G = B^G[g'][g'^{-1}]$.

In conclusion, we have

\begin{tabular}{ll}
(a) & $A = A^G[g][g^{-1}]$ \\
(b) & $g$ is a monomial in terms of $y_1, \ldots, y_{n}$ \\
(c) & $g \text{ has weight } \gcd(a_1, \ldots, a_n)$ \\ 
(d) & $A^G \text{ is of Laurent form}$ \\
(e) & $g \text{ is algebraically independent over $A^G$}$ 
\end{tabular}

\noindent That handles the $r = 1$ case.

The general case then follows by considering the quotients
$$U \to U/ \torus \to U/ \torus^2 \to \cdots \to U / \torus^r = U/G$$ 
and repeated applications of the $r = 1$ case at each level. \claimend

\medskip

Let $I \subset A^G$ be the ideal defining the closed subvariety $(Y \cap U)/G \subset U/G$. Since $Y \cap U \cong (Y \cap U)/G \x_{U/G} U$, by claim 1, 
$$Y \cap U \cong \spec{(A^G/I)[f_1, \ldots, f_r][f_1^{-1}, \ldots, f_r^{-1}]}.$$ 
The result then follows by defining $t_i \defeq f_i$.
\end{proof}

We are going to employ a technique which we call "deformation of coefficients". Suppose 
$$Y = \Proj k[x,y,z]\ /\ (f(x,y,z))$$
where $f = x^2 + yz$, is an irreducible object in $\gsmcat{G}$ (trivial group action). Let 
$$C = \Proj A \defeq \Proj k[c_1, \ldots, c_6,t]$$ 
and 
$$\gp : T \defeq \text{BiProj\ } A[x,y,z]\ /\ (f(x,y,z)) \to C$$
where 
$$f = c_1 x^2 + c_2 xy + c_3 xz + c_4 y^2 + c_5 yz + c_6 z^2$$ 
(trivial group actions). Then, $(a_1; \ldots ; a_6 ; t) \defeq (1; 0 ;0;0;1;0;1)$ is a point in $C$ with fiber $T_{(a_1 ; \ldots ; a_6 ;1)} = Y$. 

Let $L \subset C$ be a line which contains the point $(a_1 ; \ldots; a_6;1)$ and another point $(b_1; \ldots; b_6;1)$ such that $T_{(b_1; \ldots; b_6;1)}$ is \sm. By restriction, we have $\gp : T \x_C L \to L$. By considering $D \defeq (a_1; \ldots; a_6;1)$ and $D' \defeq (b_1; \ldots; b_6;1)$ as linearly equivalent divisors on $L$, we have
\begin{eqnarray*}
\gp^* \O_L(D) [\id_{T \x_C L}] &=& \gp^* \O_L(D') [\id_{T \x_C L}] \\
\ [T_{(a_1; \ldots; a_6;1)} \to T \x_C L] &=& [T_{(b_1; \ldots; b_6;1)} \to T \x_C L]
\end{eqnarray*}
as elements in $\cob{}{G}{T \x_C L}$. By pushing them down to $\pt$, we have
$$[Y] = [T_{(a_1; \ldots; a_6;1)}] = [T_{(b_1; \ldots; b_6;1)}] = [\Proj k[x,y,z]\ /\ (f(x,y,z))]$$
as elements in $\cob{}{G}{\pt}$, with generic choice on the coefficients of $f$ (as long as the object it defined is \sm).

We are now ready to prove our first main result in this paper.

\begin{thm}
\label{thm gen by lazard}
Suppose $\char{k} = 0$ and $k$ contains a primitive $e$-th \rou, where $e$ is the exponent of $G_f$. Then the canonical $\lazard_G(F)$-algebra homo\morp 
$$\lazard_G(F) \to \cob{}{G}{\pt},$$
which sends $a$ to $a\, [\id_{\pt}]$, is surjective.
\end{thm}

\begin{proof}
For simplicity, we will say $x \equiv 0$ if $x$ is in the image of the canonical map. By Theorem \ref{thm gen by geo cycle}, it is enough to consider elements of the form $[Y]$ where $Y$ is \girred{G}. We will proceed by induction on its geometric dimension. Suppose $\dim Y \defeq n \geq 1$. By the blow up relation, for any \girred{G}, \inv, \sm\ closed subscheme $Z \subset Y$,
$$[\blowup{Y}{Z}] - [Y] = - [\P_1 \to Z \to \pt] + [\P_2 \to Z \to \pt]$$
for some \qadtower s $\P_i \to Z$. Suppose the statement is true for elements of the form $[\P]$ where $\P$ is an \adtower\ over $\pt$ of dimension $n$. By Proposition \ref{prop tower reduction} and the induction assumption, we have $[\blowup{Y}{Z}] \equiv [Y]$.  By the \equi\ weak factorization theorem (Theorem 0.3.1) in \cite{weak factor thm}, whenever two \proj\ schemes $Y$, $Y' \in$ $\gsmcat{G}$\ are \equi ly birational, they define the same equivalence class, i.e., $[Y] \equiv [Y']$.

\medskip

\noindent Step 1 : Reduction to an \adtower\ over $\pt$.

W\withoutlog, $G$ acts on $Y$ faithfully. Then, there is a subgroup $H \subset G_f$ and an irreducible object $X \in  \gsmcat{(H \x G_t)}$ such that $Y \cong (G_f/H) \x X$. Also, the $(H \x G_t)$-action on $X$ is faithful. 

Since $H$ is abelian, we can write
$$H \cong H_1 \x \cdots \x H_q$$
where $H_i$ is a cyclic group of order $M_i$. Consider the field extensions $k(X)$ over $k(X)^{G_t}$ and $k(X)^{G_t}$ over $k$. By Lemma \ref{lemma fcn field for torus action}, 
$$k(X) \cong k(X)^{G_t}(t_1, \ldots, t_r)$$
where $r$ is the rank of $G_t$, $t_i$ are algebraically independent over $k(X)^{G_t}$ with actions given by non-trivial $G_t$-characters $\beta_i$. By claim 1 in the proof of Theorem 6.22 in \cite{geo equi alg cobor}, 
$$k(X)^{G_t} \cong k(x_1, \ldots, x_{n-r})[x_{n-r+1}, v_1, \ldots, v_q]\ /\ (f,v_1^{M_1} - g_1, \ldots, v_q^{M_q} - g_q)$$
for some $f$, $g_i \in k[x_1, \ldots, x_{n-r+1}]$ such that the $H$-action on $x_i$ is trivial and $H_j$ acts on $v_i$ non-trivially if and only if $i = j$. 

Now, we will follow the same arguments as in the proof of Theorem 6.22 in \cite{geo equi alg cobor}, which we will briefly explain here. Let 
$$X' \defeq {\rm Proj}\ k[t_0, \ldots, t_r] \x {\rm Proj}\ k[x_0, \ldots, x_{n-r+1}, v_1, \ldots, v_q]\ /\ (f,v_1^{M_1} - g_1, \ldots, v_q^{M_q} - g_q)$$
Then, $k(X') \cong k(X)$, but $X'$ may not be \sm. By resolution of singularities, we obtain $\stricttransform{X'}$ which is \equi ly birational to $X'$ and is \sm. Furthermore, by deformation of coefficients, we can deform $X'$ to $X''$ which is defined in the same way, but with generic coefficients for $f,g_i$. This deformation will at the same time deforms $\stricttransform{X'}$ to $\stricttransform{X''}$. Therefore, $\stricttransform{X'}, X''$ and $\stricttransform{X''}$ are all \sm\ and we have
$$[X] \equiv [\stricttransform{X'}] = [\stricttransform{X''}] \equiv [X'']$$
as elements in $\cob{}{G}{\pt}$. In short, we may assume 
$$X \cong {\rm Proj}\ k[t_0, \ldots, t_r] \x {\rm Proj}\ k[x_0, \ldots, x_{n-r+1}, v_1, \ldots, v_q]\ /\ (f,v_1^{M_1} - g_1, \ldots, v_q^{M_q} - g_q)$$
where the $(H \x G_t)$-action on $t_0$ is trivial and $f$, $g_i$ are generic homogeneous polynomials with degree $d$ and $M_i$ respectively (as long as $X$ is \sm).

Let 
$$W \defeq (G_f/H) \x {\rm Proj}\ k[t_0, \ldots, t_r] \x {\rm Proj}\ k[x_0, \ldots, x_{n-r+1}, v_1, \ldots, v_q]\ /\ (v_1^{M_1} - g_1, \ldots, v_q^{M_q} - g_q).$$ 
By Lemma 6.21 in \cite{geo equi alg cobor}, $W$ is \sm.

\medskip

\noindent Claim 1 : W\withoutlog, we may assume $D' \defeq \{x_{n-r+1} = 0\}$ is an \inv\ \sm\ divisor on $W$.

Consider the $G$-\equi\ map 
$$p : W \to {\rm Proj}\ k[x_0, \ldots, x_{n-r+1}] = \P^{n+r-1}$$ 
given by projection. The result then follows from the fact that $G$ acts on $\P^{n+r-1}$ trivially and the pull-back of a generic hyperplane is \sm\ ($\char k = 0$). \claimend

\medskip

Therefore, $D \defeq \{f = 0\}$ and $D'$ are both \inv\ \sm\ divisors on $W$. Then we have
\begin{eqnarray}
[Y] &=& [(G_f/H) \x X] \nonumber\\
&=& {\gp_W}_* \circ  c(\O(D))\,[\id_W] \nonumber\\
&=& {\gp_W}_* \circ c(\O(d\, D'))\,[\id_W]  \nonumber\\
&=& {\gp_W}_* \circ \sigma \circ c(\O(D'))\,[\id_W] \nonumber
\end{eqnarray}
for some $\sigma \in \Endoinf{\cob{G}{}{W}}$. Hence, $[Y] \equiv a\,[D']$ for some $a \in \lazard_G(F)$, by Theorem \ref{thm gen by geo cycle}, the induction assumption and the assumption that the statement holds for \adtower s over $\pt$. So, w\withoutlog, we may assume 
$$X \cong {\rm Proj}\ k[t_0, \ldots, t_r] \x {\rm Proj}\ k[x_0, \ldots, x_{n-r}, v_1, \ldots, v_q]\ /\ (v_1^{M_1} - g_1, \ldots, v_q^{M_q} - g_q).$$
Let $X_k \cong {\rm Proj}\ k[t_0, \ldots, t_r] \x {\rm Proj}\ k[x_0, \ldots, x_{n-r}, v_1, \ldots, v_k]\ /\ (v_1^{M_1} - g_1, \ldots, v_k^{M_k} - g_k).$

\medskip

\noindent Claim 2 : For generic $g_i$, the $(H \x G_t)$-schemes $X_k$, where $1 \leq k \leq q$, are all \sm. Moreover, the sum of the \inv\ divisors $\{g_i = 0\}$ on ${\rm Proj}\ k[t_0, \ldots, t_r] \x \Proj{k[x_0, \ldots, x_{n-r}]}$ is a \rsncd.

The first part follows from Lemma 6.21 in \cite{geo equi alg cobor}\  with the fact that the choice of $g_i$ there is actually generic. The second part is just Bertini's Theorem. \claimend

\medskip

Now, let $W$ be the $G$-scheme 
$$(G_f/H) \x {\rm Proj}\ k[t_0, \ldots, t_r] \x {\rm Proj}\ k[x_0, \ldots, x_{n-r}, v_1, \ldots, v_q]\ /\ (v_1^{M_1} - g_1, \ldots, v_{q-1}^{M_{q-1}} - g_{q-1}),$$
which may not be \sm. By claim 2, we can consider $Y$ as an \inv\ \sm\ divisor $\{v_q^{M_q} = g_q\}$ on $W$.

\medskip

\noindent Claim 3 : The singular locus of $W$ is given by $\{x_0 = \cdots = x_{n-r} = v_1 = \cdots = v_{q-1} = 0\}$.

W\withoutlog, we may assume that $G_f = H$, $r=0$ and $k$ is algebraically closed. Consider the \proj\ space $\Proj{k[x_0, \ldots, x_n]} = \P^n$ with trivial action. By claim 2, the \inv\ divisors $\{g_i = 0\}$ are \sm\ and the sum of them is a \rsncd. Suppose $p \defeq (a_0 ; \cdots; a_n; b_1 ; \cdots ; b_q)$ is a $k$-rational point in $\singular{W}$. By reordering, suppose the vectors $\gradient{(v_1^{M_1} - g_1)}, \ldots, \gradient{(v_k^{M_k} - g_k)}$, where $1 \leq k \leq q-1$, are linearly dependent at $p$ and $k$ is the smallest among such choices. Then the coordinates $b_1, \ldots, b_k$ are necessarily all zero ($\char k = 0$), which implies that $g_1(p) = \cdots = g_k(p) = 0$. Now, assume $a_0, \ldots, a_n$ are not all zero. Then, $\overline{p} \defeq (a_0 ; \cdots; a_n)$ will be a $k$-rational point in $\P^n$ which lies in the intersection of the divisors $\{g_1 = 0\}, \ldots, \{g_k = 0\}$. In addition, the vectors $\gradient{g_1}, \ldots, \gradient{g_k}$ will be linearly dependent at $\overline{p}$, which contradicts with the choice of $g_i$. Therefore, $a_0, \ldots, a_n$ are all zero. Since $p$ is in $W$, the coordinates $b_1, \ldots, b_{q-1}$ are all zero too.  \claimend

\medskip

Apply resolution of singularities on $W$ to obtain $\tilde{W}$. By claim 3, $Y$ is disjoint from $\singular{W}$ and hence can be considered as an \inv\ \sm\ divisor on $\tilde{W}$. Now, consider the \inv\ divisor $D \defeq \{v_q = 0\}$ on $W$, which is \sm\ by claim 2. For the same reason as $Y$, it can also be considered as a divisor on $\tilde{W}$. Then, as before, we have
$$[Y] = {\gp_{\tilde{W}}}_* \circ c(\O(Y)) [\id_{\tilde{W}}] = {\gp_{\tilde{W}}}_* \circ c(\O(M_q\, D)) [\id_{\tilde{W}}] \equiv a\, [D]$$
for some $a \in \lazard_G(F)$. Since $[D] = [(G_f/ H) \x X_{q-1}]$, by repeating the same argument, we may assume 
$$Y \cong (G_f/H) \x X_0 = (G_f/H) \x {\rm Proj}\ k[t_0, \ldots, t_r] \x \P^{n-r} = (G_f/H) \x \P(V) \x \P^{n-r}$$ 
for some $G_t$-\repn\ $V$. The result then follows from our assumption that the statement is true for admissible towers. That finishes step 1.

\medskip

\noindent Step 2 : Reduction to an element of the form $[\P(V)]$ where $V$ is a $(n+1)$-dimensional $G$-\repn.

We will now consider the \adtower s over $\pt$. Suppose 
$$\P = \P_m \to \P_{m-1} \to \cdots \to \P_0 = \pt$$
is an \adtower\ over $\pt$ with dimension $n$ and length $m$. W\withoutlog, we may assume $\dim \P_m = n > \dim \P_{m-1}$. We will proceed by induction on $m$ and the dimension of its $(m-1)$-th level. If $m = 1$, then $\P = \P(V)$ for some $G$-\repn\ $V$ and we are done. If $\dim \P_{m-1} = 0$, then $\P_{m-1} \cong \pt$ and we are also done.

Suppose $m \geq 2$ and $\dim \P_{m-1} \geq 1$. Following the notation in Lemma \ref{lemma lb structure over adtower}, $\P_i = \P(\cat{E}_i \oplus \L_i)$ and $D_i \defeq \P(\cat{E}_i)$. Then, by Lemma \ref{lemma lb structure over adtower}, there exist integers $d_1, \ldots, d_{m-1}$ and character $\gb$ such that 
$$\L_m \cong \O(d_1 D_1 + \cdots + d_{m-1} D_{m-1}) \otimes \gb.$$
Moreover, for $1 \leq i \leq m-1$, if we define $\hat{\P} \defeq \P(\cat{E}_m \oplus \L_m \oplus \L_m(D_i))$ and $\P' \defeq \P(\cat{E}_m \oplus \L_m(D_i))$, then, by Lemma 6.9 in \cite{geo equi alg cobor} (still holds for our group $G$), the sum of the \inv\ divisors $\P$, $\P'$ and $\hat{\P}|_{D_i}$ on $\hat{\P}$ is a \rsncd\ and $\hat{\P}|_{D_i} + \P \sim \P'$. Hence, by applying the \textbf{(EFGL)} axiom on $\hat{\P}$ and pushing everything down to $\pt$, we have
$$[\P'] = [\hat{\P}|_{D_i}] + [\P] + {\gp_{\hat{\P}}}_* \circ \gs \circ c(\O(\hat{\P}|_{D_i})) \circ c(\O(\P))\,[\id_{\hat{\P}}]$$
for some $\gs \in \Endoinf{\cob{}{G}{\hat{\P}}}$. Therefore,
$$[\P'] = [\hat{\P}|_{D_i}] + [\P] + {\gp_{\hat{\P}}}_* \circ \gs\, [\P|_{D_i} \embed \hat{\P}] \equiv [\hat{\P}|_{D_i}] + [\P]$$
by Theorem \ref{thm gen by geo cycle} and the induction assumption. Moreover, notice that 
$$\hat{\P}|_{D_i} = \P(\cat{E}_m \oplus \L_m \oplus \L_m(D_i)) \to \P_{m-1}|_{D_i}$$ 
is an \adtower\ over $\pt$ with dimension $n$ and length $m$, but with $\dim \P_{m-1}|_{D_i} = \dim \P_{m-1} - 1$. By the induction assumption, $[\hat{\P}|_{D_i}] \equiv 0$ and so,
$$[\P(\cat{E}_m \oplus \L_m)] = [\P]  \equiv [\P'] = [\P(\cat{E}_m \oplus \L_m(D_i))]$$
By repeated applications, we may assume $\L_m \cong \gb_m$ for some character $\gb_m$.

Apply the same argument on the other line bundles used in defining $\P_m$ and we will obtain
$$\P_m = \P((\O_{\P_{m-1}} \otimes \gb_1) \oplus \cdots \oplus (\O_{\P_{m-1}} \otimes \gb_p)).$$ 
Then, $[\P] = [\P(\gb_1 \oplus \cdots \oplus \gb_p)]\, [\P_{m-1}]$ with $\dim \P_{m-1} < n$ and we are done.

\medskip

\noindent Step 3 : Reduction to $n = 0$ case.

It remains to consider elements of the form $[\P(V)]$ where $V$ is a $(n+1)$-dimensional $G$-\repn. The following proof is an analogue of the proof of Lemma 4.2.3 in \cite{universal alg cobor}. 

Let $X \defeq \P(\epsilon \oplus \beta_1 \oplus \cdots \oplus \beta_{n})$ and $D \defeq \P(\epsilon \oplus \beta_1 \oplus \cdots \oplus \beta_{n-1})$ for some characters $\beta_1, \ldots, \beta_n$. We then consider $D$ as an \inv\ \sm\ divisor on $X$. Let $\hat{\P} \defeq \P(\O \oplus \O(D)) \to X$. As before, the sum of the \inv\ divisors $A \defeq \hat{\P}|_D$, $B \defeq \P(\O_X)$ and $C \defeq \P(\O_X(D))$ on $\hat{\P}$ is a \rsncd\ and $A + B \sim C$. Therefore, we have
$$[\P(\O_X(D))] = [\hat{\P}|_D] + [\P(\O_X)] + {\gp_{\hat{\P}}}_* \circ \sigma\, [\P(\O_X)|_D \embed \hat{\P}]$$
for some $\sigma$. Again, by Theorem \ref{thm gen by geo cycle} and the induction assumption on dimension, we have
$$[X] = [\P(\O_X(D))] \equiv [\hat{\P}|_D] + [\P(\O_X)] = [\P(\O_D \oplus \O_D(D)) \to D \to \pt] + [X].$$
Hence, 
$$[\P(\O \oplus \O(1) \otimes \dual{\beta_n}) \to \P(\epsilon \oplus \beta_1 \oplus \cdots \oplus \beta_{n-1}) \to \pt] \equiv 0$$
because $\O_X(D) \cong \O(1) \otimes \dual{\beta_n}$.

\medskip

\noindent Claim 4 : $[\P(\O \oplus (\oplus_{i=0}^{p-1}\, \O(1) \otimes \dual{\beta_{n-i}})) \to \P(\epsilon \oplus \beta_1 \oplus \cdots \oplus \beta_{n-p}) \to \pt] \equiv 0$ for all $1 \leq p \leq n$.

We just showed the statement for $p = 1$. For $1 \leq p \leq n - 1$, let $X \defeq \P(\epsilon \oplus \beta_1 \oplus \cdots \oplus \beta_{n-p})$ and $D \defeq \P(\epsilon \oplus \beta_1 \oplus \cdots \oplus \beta_{n-p-1})$. Apply a similar argument on 
$$\hat{\P} \defeq \P(\O \oplus (\oplus_{i=0}^{p-1}\, \O(1) \otimes \dual{\beta_{n-i}})) \oplus \O(D)) \to X,$$ 
we have
\begin{eqnarray}
&& [\P((\oplus_{i=0}^{p-1}\, \O(1) \otimes \dual{\beta_{n-i}}) \oplus \O(D)) \to X \to \pt] \nonumber\\
&\equiv& [\hat{\P}|_D \to D \to \pt] + [\P(\O \oplus (\oplus_{i=0}^{p-1}\, \O(1) \otimes \dual{\beta_{n-i}})) \to X \to \pt]. \nonumber
\end{eqnarray}
On one hand, 
\begin{eqnarray}
&& [\P((\oplus_{i=0}^{p-1}\, \O(1) \otimes \dual{\beta_{n-i}}) \oplus \O(D)) \to X \to \pt] \nonumber\\
&=& [\P((\oplus_{i=0}^{p-1}\, \dual{\beta_{n-i}}) \oplus \dual{\beta_{n-p}}) \x X] \nonumber\\
&=& [\P((\oplus_{i=0}^{p-1}\, \dual{\beta_{n-i}}) \oplus \dual{\beta_{n-p}})]\, [X] \equiv 0 \nonumber
\end{eqnarray}
by the induction assumption on dimension. On the other hand, by the induction assumption on $p$,
\begin{eqnarray}
&& [\hat{\P}|_D \to D \to \pt] + [\P(\O \oplus (\oplus_{i=0}^{p-1}\, \O(1) \otimes \dual{\beta_{n-i}})) \to X \to \pt] \nonumber\\
&\equiv& [\hat{\P}|_D \to D \to \pt] \nonumber\\
&=& [\P(\O \oplus (\oplus_{i=0}^p\, \O(1) \otimes \dual{\beta_{n-i}})) \to \P(\epsilon \oplus \beta_1 \oplus \cdots \oplus \beta_{n-p-1}) \to \pt]. \nonumber
\end{eqnarray}
\claimend

\medskip

By setting $p = n$ in claim 4, we obtain
$$0 \equiv [\P(\O \oplus (\oplus_{i=0}^{n-1}\, \O(1) \otimes \dual{\beta_{n-i}})) \to \P(\epsilon) \to \pt] = [\P(\epsilon \oplus (\oplus_{i=0}^{n-1}\, \dual{\beta_{n-i}}))].$$
The result then follows because $\beta_i$ are arbitrary.

\medskip

\noindent Step 4 : $n = 0$ case.

Finally, it remains to show the statement when $\dim Y = 0$. W\withoutlog, $G$ acts on $Y$ faithfully. In that case, $G_t$ is necessarily trivial. By the same arguments as in step 1, we reduce it to the case when $[Y] = [G/H]$. Let $G_i$ be cyclic groups of order $M_i$ such that $G/H \cong G_1 \x \cdots \x G_q.$ Then, $[G/H] = [G_1] \cdots [G_q]$ where $[G_i]$ are considered as elements in $\gsmcat{G}$\ equipped with the natural $G$-action. Let $g_i$ be a generator of $G_i$ and $\gb_i : G \to G/H \to k$ be the $G$-character such that $\gb_i(g_i) = \gx_i$, where $\gx_i$ is a primitive $M_i$-th \rou, and $\gb_i(g_j) = 1$ if $i \neq j$. Let $W \defeq \P(\gep \oplus \gb_i)$. Then, for some $\gs \in \Endoinf{\cob{}{G}{W}}$ and $\gs' \in \Endoinf{\cob{}{G}{\P(\gb_i)}}$,
$$[G_i] = {\gp_W}_* \circ c(\O(M_i))[\id_W] = {\gp_W}_* \circ \gs \circ c(\O(1))[\id_W] = {\gp_{\P(\gb_i)}}_* \circ \gs'\, [\id_{\P(\gb_i)}] \equiv 0.$$
That finishes the proof of Theorem \ref{thm gen by lazard}.
\end{proof}

\bigskip
\bigskip

\section{Fundamental properties}
\label{sect fundamental properties}

As in the algebraic cobordism theory $\Omega(-)$ in \cite{universal alg cobor}, we also have the localization property and homotopy invariance property, when $\char{k} = 0$ (see section 3.2 and 3.4 in \cite{universal alg cobor}). Since we need to use the embedded desingularization theorem and the weak factorization theorem, we will assume $\char{k} = 0$ for the rest of this section.

Let us start with the localization property. Since $\basicmod{G}{F}{}{-}$ is generated by infinite cycles, which is relatively difficult to work with, our strategy is to first define another \equi\ algebraic cobordism theory $\cob{}{G}{-}_{\rm fin}$ (which does not involve infinite sum), prove that $\cob{}{G}{-}_{\rm fin}$ and $\cob{}{G}{-}$ are canonically isomorphic and then show that the localization property holds in $\cob{}{G}{-}_{\rm fin}$.

For an object $X \in \gvar{G}$, let 
$$\basicmod{G}{F}{}{X}_{\rm fin} \defeq \oplus_{s \geq 0}\ \bigbasicmod{G}{F}{s}{X}.$$ 
We then define the basic operations : \proj\ push-forward, \sm\ pull-back, external product as in $\basicmod{G}{F}{}{-}$, and the (first) Chern class operator
$$c(\L) [f : Y \to X, \L_1, \ldots, \L_r] \defeq [f : Y \to X, \L_1, \ldots, \L_r, f^* \L],$$ 
which is also consistent with our definition of infinite Chern class operator in $\basicmod{G}{F}{}{-}$. For simplicity of notation, let $\Endofin{\basicmod{G}{F}{}{X}_{\rm fin}}$ be the $\lazard_G(F)$-subalgebra of $\Endo{\basicmod{G}{F}{}{X}_{\rm fin}}$ generated by (first) Chern class operators.

Then we define $\basicmod{G}{F}{}{-}_{\rm B}$ as the quotient of $\basicmod{G}{F}{}{-}_{\rm fin}$ by imposing the following axioms (closed \wrt\ the four basic operations : push-forward, pull-back, external product and the (first) Chern class operator) :

\medskip

\noindent \textbf{(Blow)} For all \girred{G}\ $Y \in $ $\gsmcat{G}$ and \inv, \sm\ closed subscheme $Z \subset Y$, 
\begin{eqnarray}
[\blowup{Y}{Z} \to Y] - [\id_Y] &=& -\ [\P(\O \oplus \dual{\nbundle{Z}{Y}}) \to Z \embed Y] \nonumber\\
&& +\ [\P(\O \oplus \O(1)) \to \P(\dual{\nbundle{Z}{Y}}) \to Z \embed Y]. \nonumber
\end{eqnarray}

\medskip

Next, we define $\basicmod{G}{F}{}{-}_{\rm N}$ as the quotient of $\basicmod{G}{F}{}{-}_{\rm B}$ by imposing :

\medskip

\noindent \textbf{(Nilp)} For all \girred{G}\ $Y \in$ $\gsmcat{G}$, $\L \in \picard{G}{Y}$ and set $S$ as in (\ref{eqn basic element}),
$$V^{\nilpnum{Y,\L, S}}_S(\L)[\id_Y] = 0$$
where $\nilpnum{Y,\L, S}$ is the positive integer given in Remark \ref{rmk nilpnum}.

\medskip

Finally, we define $\cob{}{G}{-}_{\rm fin}$ as the quotient of $\basicmod{G}{F}{}{-}_{\rm N}$ by imposing the \textbf{(Sect)} and \textbf{(EFGL)} axioms as in the definition of $\cob{}{G}{-}$. Notice that the equations in the \textbf{(EFGL)} axiom are all finite sums because of the \textbf{(Nilp)} axiom.

\begin{lemma}
\label{lemma fin and inf theory iso}
Suppose $\char{k} = 0$. The two theories $\cob{}{G}{-}$ and $\cob{}{G}{-}_{\rm fin}$ are canonically isomorphic.
\end{lemma}

\begin{proof}
Let $\Phi : \basicmod{G}{F}{}{X}_{\rm fin} \to \cob{}{G}{X}$ be the natural map that sends cycles to cycles. It clearly commutes with the four basic operations. This map descends to a map $\Phi : \cob{}{G}{-}_{\rm fin} \to \cob{}{G}{X}$ because the blow up relation holds in $\cob{}{G}{-}$ and by Remark \ref{rmk nilpnum}.

The inverse map $\Psi : \basicmod{G}{F}{}{X} \to \cob{}{G}{X}_{\rm fin}$ is also natural (infinite cycle becomes a finite sum of cycles because of the \textbf{(Nilp)} axiom), which clearly descends to a map $\Psi : \cob{}{G}{X} \to \cob{}{G}{X}_{\rm fin}$. Moreover, $\Psi \circ \Phi$ is the identity and $\Phi$ is surjective by Corollary \ref{cor gen by cycle}.
\end{proof}

Now we can proceed to prove the localization property in $\cob{}{G}{-}_{\rm fin}$. We first need the following Lemma.

\begin{lemma}
\label{lemma sheaf trivial over open subsch}
Assume that $\char{k} = 0$. Suppose $X$ is an object in $\gvar{G}$, $Z \in \gvar{G}$ is an invariant closed subscheme of $X$, $U \defeq X - Z$ is the complement, $i : Z \embed X$ and $j : U \embed X$ are the immersions.

\noindent \begin{statementslist}
{\rm (1)} & If $X$ is smooth and \girred{G},

\begin{center}$\displaystyle Z = D_1 \cup D_2 \cup \cdots \cup D_n \cup Z_0$\end{center}

where $D_i$ are distinct $G$-prime divisors on $X$ and $\codim{Z_0}{X} \geq 2$, then the following sequence is exact

\begin{center}$\displaystyle \Z^n \stackrel{a}{\longto} \picard{G}{X} \stackrel{j^*}{\longto} \picard{G}{U} \longto 0,$\end{center}

where $a$ sends $(c_1, \ldots, c_n)$ to $\O_X(\sum_{i = 1}^n c_i D_i)$. \\
{\rm (2)} & If $\L$ is a sheaf in $\picard{G}{X}$ such that $\L|_U \cong \O_U$, then, for all $x \in \cob{}{G}{X}_{\rm fin}$, the element $c(\L)(x)$ lies inside $i_*\, \cob{}{G}{Z}_{\rm fin}$.
\end{statementslist}
\end{lemma}

\begin{proof}
\noindent (1).\tab First of all, by Proposition \ref{prop equi embed}, $X$ is \qproj\ with a \glin{G}\ action. Let $W$ be an object in $\gsmcat{G}$ such that $G$ acts on $W$ freely and the principal bundle quotient $W \to W / G$ exists in the category of schemes. Then, by Proposition 23 in \cite{equi intersection}, the principal bundle quotient $X \x W \to (X \x W)/G$ exists in the category of schemes. So it is locally trivial in the \etale\ topology. Therefore, $(X \x W)/G$ is smooth (because $X \x W$ is).

Now, by Theorem 1 in \cite{equi intersection} and the fact that $X$ is locally factorial, 
$$\picard{G}{X} \cong CH_G^1(X)$$
where $CH_G(X)$ is the equivariant Chow group of $X$ (see section 2.2 in \cite{equi intersection}). By definition,  $CH_G^1(X) = CH^1((X \x W)/G)$ for some $W \in \gsmcat{G}$ as mentioned above. Since $(X \x W)/G$ is smooth (in particular, locally factorial), we have 
$$CH^1((X \x W)/G) \cong \picard{}{(X \x W)/G}$$
and the same results hold if $X$ is replaced by $U$. Therefore, the sequence 
$$\Z^n \stackrel{a}{\longto} \picard{G}{X} \stackrel{j^*}{\longto} \picard{G}{U} \longto 0$$
is equivalent to 
$$\Z^n \stackrel{b}{\longto} \picard{}{(X \x W)/G} \stackrel{((j \x \id_W)/G)^*}{\longto} \picard{}{(U \x W)/G} \longto 0,$$
where $b$ sends $(c_1, \ldots, c_n)$ to $\O(\sum_{i=1}^n c_i D_i' )$ with $D_i' \defeq (D_i \x W)/G$ (considered as prime divisors on $(X \x W)/G$), which is clearly exact.

\medskip

\noindent (2).\tab W\withoutlog, $x = [f : Y \to X, \ldots]$ where $Y \in$ $\gsmcat{G}$\ is \girred{G}. Consider the following Cartesian diagram in the category $\gvar{G}$ :
\squarediagramword{\reduced{(Y|_Z)}}{Y}{Z}{X}{i'}{f'}{f}{i}
\noindent (``red'' stands for the reduced structure) Then, $f'$ is \proj, $i'$ is a closed immersion and $Y|_U = Y - \reduced{(Y|_Z)}$ with $(f^* \L)|_{Y|_U} \cong \O_{Y|_U}$. Moreover, 
$c(\L)(x) = f_*\, c(f^*\L)[\id_Y, \ldots]$. Therefore, it is enough to show that $c(f^*\L)[\id_Y, \ldots]$ lies inside $i'_*\,  \cob{}{G}{\reduced{(Y|_Z)}}_{\rm fin}$. In other words, we may assume $X$ is \sm, \girred{G}\ and $x = [\id_X, \ldots]$. By \textbf{(A3)}, we may further assume $x = [\id_X]$.

By part (1), 
$$\L \cong \O_X(\sum_i \pm D_i)$$
for some $G$-prime divisors on $X$ such that $D_i \subset Z$. Therefore,
$$c(\L)[\id_X] = c(\O(\sum_i \pm D_i))[\id_X] = \sum_i \sigma_i \circ c(\O(D_i)) [\id_X]$$
for some $\gs_i \in \Endoinf{\cob{}{G}{X}_{\rm fin}}$ (the infinite Chern class operator acts as a finite sum in $\cob{}{G}{-}_{\rm fin}$). So we may further assume $\L \cong \O_X(D)$ for some $G$-prime divisor on $X$ such that $D \subset Z$.

By applying the embedded desingularization theorem on $D \embed X$, we have a projective, birational map $\pi : \tilde{X} \to X$ in $\gsmcat{G}$ and some \sm, $G$-\inv\ exceptional divisors $E_i$ on $\tilde{X}$ such that the strict transform $\stricttransform{D} \subset \tilde{X}$ is \sm, $\pi(\stricttransform{D}), \pi(E_i) \subseteq Z$ and 
$$\pi^* \L \cong \O_{\tilde{X}}(\stricttransform{D} + \sum_i \pm E_i).$$ 
Then, by the \textbf{(Blow)} axiom, the difference
$$c(\L)[\pi : \tilde{X} \to X] - c(\L)[\id_X]$$
lies inside $i_*\, \cob{}{G}{Z}_{\rm fin}$. Furthermore, 
\begin{eqnarray*}
c(\L)[\pi : \tilde{X} \to X] &=& \pi_* \circ c(\pi^* \L)[\id_{\tilde{X}}] \\
&=& \pi_* \circ c(\O(\stricttransform{D} + \sum_i \pm E_i))[\id_{\tilde{X}}] \\
&=& \pi_* \circ \sigma [\stricttransform{D} \embed \tilde{X}] + \sum_i \pi_* \circ \sigma_i [E_i \embed \tilde{X}],
\end{eqnarray*}
for some $\sigma, \sigma_i \in \Endoinf{\cob{}{G}{\tilde{X}}_{\rm fin}}$, which clearly lies inside $i_*\, \cob{}{G}{Z}_{\rm fin}$. Hence, $c(\L)[\id_X] \in i_*\, \cob{}{G}{Z}_{\rm fin}$ and we are done.
\end{proof}

We are now ready to prove the localization property.

\begin{thm}
\label{thm localization property}

Suppose $\char{k} = 0$ and $k$ contains a primitive $e$-th \rou, where $e$ is the exponent of $G_f$. 

Let $X$ be an object in $\gvar{G}$, $Z \in \gvar{G}$ be an invariant closed subscheme of $X$ and $U$ be the complement $X - Z$. Then the following sequence is exact
$$\cob{}{G}{Z} \stackrel{i_*}{\longto} \cob{}{G}{X} \stackrel{j^*}{\longto} \cob{}{G}{U} \longto 0$$
where $i : Z \embed X$ and $j : U \embed X$ are the immersions.
\end{thm}

\begin{proof}
By Lemma \ref{lemma fin and inf theory iso}, we can consider $\cob{}{G}{-}_{\rm fin}$ instead. Our proof is similar to the proof of Theorem 3.2.7 in \cite{universal alg cobor}. First of all, the composition $j^* \circ i_*$ is clearly zero. Moreover, for any element $[f : Y \to U, \L_1, \ldots, \L_r]$ in $\cob{}{G}{U}_{\rm fin}$, by Proposition \ref{prop smooth extension}, there exists an \equi\ \proj\ map $\overline{f} : \overline{Y} \to X$ such that $\overline{Y} \in$ $\gsmcat{G}$\ and $\overline{f}|_U = f$. Also, the restriction map $\picard{G}{\overline{Y}} \to \picard{G}{Y}$ is surjective (by Lemma \ref{lemma sheaf trivial over open subsch}). So, the map $j^*$ is surjective. Hence, it is enough to show $\kernel{j^*} \subset \image{i_*}$. W\withoutlog, we may assume $X$, $U$ and $Z$ are all non-empty. By Proposition \ref{prop equi embed}, we can embedded $X$ into some $\P(V)$. It will be enough to show the exactness of the sequence
$$\cob{}{G}{\overline{X}-U} \longto \cob{}{G}{\overline{X}} \longto \cob{}{G}{U} \longto 0$$
where $\overline{X}$ is the closure of $X$ inside $\P(V)$. So, we may further assume $X$ to be \proj.

Consider the following commutative diagram :

\medskip

\begin{center}
$\begin{CD}
\basicmod{G}{F}{}{Z}_{\rm fin} @>{i_*}>> \basicmod{G}{F}{}{X}_{\rm fin} @>{j^*}>> \basicmod{G}{F}{}{U}_{\rm fin}  \\
@V{\gph_{Z,1}}VV @V{\gph_{X,1}}VV @V{\gph_{U,1}}VV \\
\basicmod{G}{F}{}{Z}_{\rm B} @>{i_*}>> \basicmod{G}{F}{}{X}_{\rm B} @>{j^*}>> \basicmod{G}{F}{}{U}_{\rm B}  \\
@V{\gph_{Z,2}}VV @V{\gph_{X,2}}VV @V{\gph_{U,2}}VV  \\
\basicmod{G}{F}{}{Z}_{\rm N} @>{i_*}>> \basicmod{G}{F}{}{X}_{\rm N} @>{j^*}>> \basicmod{G}{F}{}{U}_{\rm N} \\
@V{\gph_{Z,3}}VV @V{\gph_{X,3}}VV @V{\gph_{U,3}}VV  \\
\cob{}{G}{Z}_{\rm fin} @>{i_*}>> \cob{}{G}{X}_{\rm fin} @>{j^*}>> \cob{}{G}{U}_{\rm fin}
\end{CD}$
\end{center}

\medskip

\noindent where $\phi_{-,1}$, $\phi_{-,2}$, $\phi_{-,3}$ are the quotient maps.

\medskip

\noindent Claim 1 : 
$$j^*\ \kernel{\gph_{X,3}} \supset \kernel{\gph_{U,3}}.$$

The kernel of $\gph_{U,3}$ is generated by elements corresponding to the \textbf{(Sect)} and \textbf{(EFGL)} axioms. For the \textbf{(EFGL)} axiom, to be closed \wrt\ the basic operations (push-forward, pull-back, external product and the (first) Chern class operator), we start with elements of the form

\begin{eqnarray*}
&& V^i(\L)V^j(\L)[\id_T] - \sum_{s \geq 0}\,b^{i,j}_s\,V^s(\L)[\id_T] \\
&& V^i(\L \otimes \ga))[\id_T] - \sum_{s \geq 0}\,d(\ga)^i_s\,V^s(\L)[\id_T] \\
&& V^i(\L \otimes \cat{M})[\id_T] - \sum_{s,t \geq 0}\,f^i_{s,t}\,V^s(\L)V^t(\cat{M})[\id_T]
\end{eqnarray*}
where $T \in$ $\gsmcat{G}$\ is \girred{G}, $i,j \geq 0$, $\ga$ is a character and $\L$, $\cat{M}$ are sheaves in $\picard{G}{T}$. Pull them back along some \sm\ \morp\ $g : Y \to T$, apply $\gs \in \Endofin{\basicmod{G}{F}{}{Y}_{\rm N}}$ and then push them forward along some \proj\ \morp\ $f : Y \to U$ to obtain

\begin{eqnarray}
\label{eqn5}
&& f_* \circ \gs \circ g^*\, (V^i(\L)V^j(\L)[\id_T] - \sum_{s \geq 0}\,b^{i,j}_s\,V^s(\L)[\id_T]) \\
&& f_* \circ \gs \circ g^*\, (V^i(\L \otimes \ga))[\id_T] - \sum_{s \geq 0}\,d(\ga)^i_s\,V^s(\L)[\id_T]) \nonumber\\
&& f_* \circ \gs \circ g^*\, (V^i(\L \otimes \cat{M})[\id_T] - \sum_{s,t \geq 0}\,f^i_{s,t}\,V^s(\L)V^t(\cat{M})[\id_T]) \nonumber
\end{eqnarray}
These are the elements in the kernel of $\gph_{U,3}$ corresponding to imposing the \textbf{(EFGL)} axiom (see subsection 2.1.3 in \cite{universal alg cobor} for details). For simplicity, we will only handle equation (\ref{eqn5}). The other two will follow from similar arguments.

Since $g$ is \sm, $Y$ is in $\gsmcat{G}$. W\withoutlog, we may assume $Y$ is also \girred{G}. Extend $f : Y \to U$ to $\overline{f} : \overline{Y} \to X$ as before. Consider the following commutative diagram (rows are not necessarily exact) :

\begin{center}
$\begin{CD}
\basicmod{G}{F}{}{\reduced{(\overline{Y}|_Z)}}_{\rm N} @>{i_*}>> \basicmod{G}{F}{}{\overline{Y}}_{\rm N}  @>{j^*}>> \basicmod{G}{F}{}{Y}_{\rm N}  \\
@V{\overline{f}_*}VV @V{\overline{f}_*}VV @V{f_*}VV  \\
\basicmod{G}{F}{}{Z}_{\rm N} @>{i_*}>> \basicmod{G}{F}{}{X}_{\rm N}  @>{j^*}>> \basicmod{G}{F}{}{U}_{\rm N}
\end{CD}$
\end{center}
\noindent (By abuse of notation, we will denote the immersions $Y \embed \overline{Y}$ and $\reduced{(\overline{Y}|_Z)} \embed \overline{Y}$ by $j$ and $i$ respectively as well). By some diagram chasing, it is enough to lift the element
$$\gs \circ g^*\, (V^i(\L)V^j(\L)[\id_T] - \sum_{s \geq 0}\,b^{i,j}_s\,V^s(\L)[\id_T])$$ 
to an element in $\basicmod{G}{F}{}{\overline{Y}}_{\rm N}$ which is congruent to zero mod $\image{i_*} + \kernel{\gph_{\overline{Y},3}}$. By extending $\gs$ to $\overline{\gs}$ as before, we may further assume $\sigma = \id$.

Since $g^* \circ V^i(\L) = V^i(g^* \L) \circ g^*$ and $g^*[\id_T] = [\id_Y]$, we may assume $g = \id_Y$. By extending $\L \in \picard{G}{Y}$ to $\overline{\L} \in \picard{G}{\overline{Y}}$, the lifting we want is
$$V^i(\overline{\L}) V^j(\overline{\L})[\id_{\overline{Y}}] - \sum_{s \geq 0}\,b^{i,j}_s\,V^s(\overline{\L}) [\id_{\overline{Y}}].$$

For the \textbf{(Sect)} axiom, by similar arguments, it is enough to lift an element of the form $[\id_Y,\L] - [D \embed Y],$ where $\L$ is a sheaf in $\picard{G}{Y}$ and $D$ is an invariant \sm\ divisor on $Y$ cut out by some invariant section $s \in \gsection{Y}{\L}^G$, to an element in $\basicmod{G}{F}{}{\overline{Y}}_{\rm N}$ which is congruent to zero.

Let $\overline{D}$ be the closure of $D$ in $\overline{Y}$. Apply the embedded desingularization Theorem on $\overline{D} \embed \overline{Y}$ to obtain the following commutative diagram :
\squarediagramword{\stricttransform{D}}{\tilde{Y}}{\overline{D}}{\overline{Y}}{}{\gp}{\gp}{}
\noindent Let $E_k$ be the strict transforms of the exceptional divisors. 

On one hand, we have
\begin{eqnarray}
\pi_* \circ c(\gp^* \O_{\overline{Y}}(\overline{D}))\,[\id_{\tilde{Y}}] &=& \pi_* \circ c(\O_{\tilde{Y}}(\,\stricttransform{D} + \sum_k\, m_k E_k\,))\,[\id_{\tilde{Y}}] \nonumber\\
&& \text{for some integers $m_k$} \nonumber\\
&\equiv& \pi_* \circ c(\O(\stricttransform{D}))\,[\id_{\tilde{Y}}] + \sum_k\, \pi_* \circ \gs_k \circ c(\O(E_k))\,[\id_{\tilde{Y}}] \nonumber\\
&& \text{for some $\gs_k \in \Endofin{\basicmod{G}{F}{}{\tilde{Y}}_{\rm N}}$} \nonumber\\
&\equiv& [\stricttransform{D} \to \overline{Y}] + \sum_k \pi_* \circ \gs_k\, [E_k \embed \tilde{Y}], \nonumber
\end{eqnarray}
which is congruent to $[\stricttransform{D} \to \overline{Y}]$ because all $E_k$ lie over $\reduced{(\overline{Y}|_Z)}$.

On the other hand, by the \textbf{(Blow)} axiom, we have
$$[\tilde{Y} \to \overline{Y}, \gp^* \O_{\overline{Y}}(\overline{D})] \equiv [\id_{\overline{Y}}, \O_{\overline{Y}}(\overline{D})]$$
because the towers created by each blow up lie over the \sm\ center, which lies over $\reduced{(\overline{Y}|_Z)}$. Hence, we have
$$[\id_{\overline{Y}}, \O_{\overline{Y}}(\overline{D})] - [\stricttransform{D} \to \overline{Y}] \equiv 0,$$
which is the lifting we want. \claimend

\medskip

\noindent Claim 2 : For all \girred{G}\ $Y \in$ $\gsmcat{G}$, $\L \in \picard{G}{Y}$ and $S$, there exist \proj\ maps $f_i$, \sm\ maps $g_i$, $\sigma_i$ and $\L_i \in \picard{G}{Z_i}$ such that
$$V^{\nilpnum{Y,\L,S}}_S(\L)[\id_Y] = \pi_* \circ V^{\nilpnum{Y,\L,S}}_S(\pi^* \L) [\id_{\tilde{Y}}] + \sum_i {f_i}_* \circ \sigma_i \circ g_i^* \circ V^{\nilpnum{Z_i,\L_i,S}}_S(\L_i)[\id_{Z_i}]$$
as elements in $\basicmod{G}{F}{}{Y}_{\rm B}$, where $\pi : \tilde{Y} \to Y$ and $Z_i$ are given as part of the definition of $\nilpnum{Y,\L,S}$ (see Remark \ref{rmk nilpnum}) and $\dim Z_i < \dim Y$. 

\medskip

See the proof of Proposition \ref{prop Nilp axiom}. \claimend

\medskip

\noindent Claim 3 :
$$j^*\ \kernel{\{\gph_{X,3} \circ \gph_{X,2}\}} \supset \kernel{\gph_{U,2}}.$$

We need to lift an element of the form 
$$f_* \circ \gs \circ g^*\, V^r_S(\L)[\id_T],$$ 
where $f : Y \to U$, $\sigma$, $g : Y \to T$ are as in claim 1, $Y$ and $T$ are both \girred{G}\ and \sm, $\L$ is a sheaf in $\picard{G}{T}$, $S$ is a finite set as in (\ref{eqn basic element}) and $r = \nilpnum{T,\L,S}$, to an element in $\basicmod{G}{F}{}{X}_{\rm B}$ which is congruent to zero mod $\image{i_*} + \kernel{\{\gph_{X,3} \circ \gph_{X,2} \}}$. We will prove this by induction on $\dim T$. 

By the definition of $\nilpnum{T,\L,S}$ (see Remarks \ref{rmk nilpnum}) and claim 2, we may assume $\L \cong \O(\sum_i m_i D_i) \otimes \beta$ such that $D_i$ are distinct \inv\ \sm\ divisors, $\alpha_r \cong \dual{\beta}$ and $r > \nilpnum{D_i, \L|_{D_i}, S}$ for all $i$.

By Proposition \ref{prop smooth extension}, we can embedded $T$ into some $\P(V)$ such that its closure $\overline{T}$ is \sm. By the embedded desingularization theorem, we may also assume $\overline{D_i}$ are \sm. Let $\overline{\L} \defeq \O(\sum_i m_i \overline{D_i}) \otimes \beta \in \picard{G}{\overline{T}}$ be an extension of $\L$. By the same argument as in the proof of claim 1, we can extend $f : Y \to U$ to $\overline{f} : \overline{Y} \to X$ and it is enough to lift $g^* \, V^r_S(\L)[\id_T]$ to an element in $\basicmod{G}{F}{}{\overline{Y}}_{\rm B}$ which is congruent to zero. It is worth mentioning that even though the choices made in the definitions of $\nilpnum{T,\L,S}$ and $\nilpnum{\overline{T},\overline{\L},S}$ can be completely incoherent, all we need in this proof are the existence of extensions and smoothness of some objects.

Since $g \circ j^{-1} : \overline{Y} \dashrightarrow \overline{T}$ is an \equi\ rational map and $\overline{T}$ is \proj, by Proposition \ref{prop birat given by blowup}, we can blow up $\overline{Y}$ along some \inv\ center to obtain $Y'$ so that $g \circ j^{-1}$ lifts to an \equi, regular map. By resolution of singularities, we may assume $Y'$ is \sm. Also, $j$ lifts to an immersion $Y \embed Y'$. Therefore, we may assume $Y' = \overline{Y}$ and $g : Y \to \overline{T}$ has an extension $\overline{g} : \overline{Y} \to \overline{T}$.

If $\dim T = 0$, then $\L \cong \gb$ and the element $V^r_S(\gb)[\id_{\overline{Y}}]$ will be a lifting of $g^* \, V^r_S(\L)[\id_T]$. Also, 
$$V^r_S(\gb)[\id_{\overline{Y}}] = V^{r-1}_S(\gb) \circ c(\gb \otimes \dual{\gb})[\id_{\overline{Y}}] \equiv 0$$
(because of the \textbf{(Sect)} axiom).

Suppose $\dim T > 0$. First of all, we lift $g^* \, V^r_S(\L)[\id_T]$ to $V^r_S(\overline{g}^* \overline{\L})[\id_{\overline{Y}}]$. Since $g$ is \sm, $g^{-1} D_i$ are also distinct \inv\ \sm\ divisors on $Y$. In other words, 
$$\singular{\overline{g}^{-1} \overline{D_i}} \subseteq \overline{Y} - Y$$ 
Apply the embedded desingularization Theorem on $\cup_i\, \overline{g}^{-1} \overline{D_i} \embed \overline{Y}$ to obtain $\tilde{Y}$ and denote the composition $\tilde{Y} \to \overline{Y} \to \overline{T}$ by $\tilde{g}$. Then it is clear that the immersion $Y \embed \overline{Y}$ lifts to an immersion $Y \embed \tilde{Y}$, the map $\tilde{g}$ extends $g$ and, for all $i$,
$$\tilde{g}^* \O_{\overline{T}}(\overline{D_i}) = \O_{\tilde{Y}}(\overline{g^{-1} D_i} + \sum_k m_{i,k} D_{i,k})$$
for some integers $m_{i,k}$ and \inv\ \sm\ divisors $D_{i,k}$ (exceptional divisors) on $\tilde{Y}$ which lie over $\overline{Y} - Y$. Moreover, $\overline{g^{-1} D_i}$, the closure of $g^{-1} D_i$ in $\tilde{Y}$, is also \sm. Also, as before, it is enough to lift $g^* \, V^r_S(\L)[\id_T]$ to an element in $\basicmod{G}{F}{}{\tilde{Y}}_{\rm B}$ which is congruent to zero.

Let $r'$ be the max of $\nilpnum{D_i, \L|_{D_i}, S}$. Then, for some $\gs_1$, we have
\begin{eqnarray}
V^r_S(\tilde{g}^* \overline{\L})[\id_{\tilde{Y}}] &=& \gs_1 \circ V^{r'}_S(\tilde{g}^* \overline{\L}) \circ c(\tilde{g}^* \overline{\L} \otimes \dual{\gb})\, [\id_{\tilde{Y}}] \nonumber\\
&=& \gs_1 \circ V^{r'}_S(\tilde{g}^* \overline{\L}) \circ c(\O(\, \sum_i m_i\, \overline{g^{-1} D_i} + \sum_{i,k}\, m_i m_{i,k} D_{i,k}))\, [\id_{\tilde{Y}}], \nonumber\\
&\equiv & \sum_i \gs_{2,i} \circ V^{r'}_S(\tilde{g}^* \overline{\L})\, [\overline{g^{-1} D_i} \embed \tilde{Y}] + \sum_{i,k} \gs_{3,i,k}\, [D_{i,k} \embed \tilde{Y}] \nonumber
\end{eqnarray}
for some $\gs_{2,i}$, $\gs_{3,i,k}$. Denote $\overline{g^{-1} D_i}$ by $A_i$ for simplicity. Since  $[D_{i,k} \embed \tilde{Y}]$ lies inside $\image i_*$,

\begin{eqnarray}
V^r_S(\tilde{g}^* \overline{\L})[\id_{\tilde{Y}}] &\equiv& \sum_i \gs_{2,i} \circ V^{r'}_S(\tilde{g}^* \overline{\L})\, [A_i \embed \tilde{Y}] \nonumber\\
&=& \sum_i {h_i}_* \circ \gs_{4,i} \circ V^{r'}_S(\tilde{g}^* \overline{\L}|_{A_i})\, [\id_{A_i}], \nonumber
\end{eqnarray}
for some $\gs_{4,i}$, where $h_i$ are the immersions $A_i \embed \tilde{Y}$. Notice that ${h_i}_* \circ \gs_{4,i} \circ V^{r'}_S(\tilde{g}^* \overline{\L}|_{A_i})\, [\id_{A_i}]$ is a lifting of 
$${(h_i|_Y)}_* \circ \gs_{5,i} \circ g^*\, V^{r'}_S(\L|_{D_i})\, [\id_{D_i}]$$
where $\sigma_{5,i}$ is the restriction of $\sigma_{4,i}$ over $g^{-1} D_i$. Since $r' \geq \nilpnum{D_i, \L|_{D_i}, S}$, by the induction assumption ($\dim D_i = \dim T - 1$), it has a lifting which is congruent to zero. 

To summarize, we lift $a \defeq g^* \, V^r_S(\L)[\id_T]$ to $b \defeq V^r_S(\tilde{g}^* \overline{\L})[\id_{\tilde{Y}}]$, which is congruent to the sum of $c_i \defeq {h_i}_* \circ \gs_{4,i} \circ V^{r'}_S(\tilde{g}^* \overline{\L}|_{A_i})\, [\id_{A_i}]$. Then, by the induction assumption, $j^*(c_i)$ has a lifting $d_i$ which is congruent to zero. Hence, $b - \sum_ i c_i + \sum_i d_i$ is the lifting we want. \claimend

\medskip

\noindent Claim 4 :
$$j^*\ \kernel{\{\gph_{X,3} \circ \gph_{X,2} \circ \gph_{X,1}\}} \supset \kernel{\gph_{U,1}}.$$

It is enough to lift an element of the form 
$$g^*([\blowup{T}{W} \to T] - [\id_T] + [Q_1 \to W \embed T] - [Q_2 \to W \embed T]),$$
where $Y$, $T$ and $g : Y \to T$ are as in claim 1, $W$ is an \inv, \sm\ closed subscheme of $T$, $Q_1 = \P(\O \oplus \dual{\nbundle{W}{T}}) \to W$ and $Q_2 = \P(\O \oplus \O(1)) \to \P(\dual{\nbundle{W}{T}}) \to W$, to an element in $\basicmod{G}{F}{}{\overline{Y}}_{\rm fin}$ which is congruent to zero mod $\image{i_*} + \kernel{\{\gph_{\overline{Y},3} \circ \gph_{\overline{Y},2} \circ \gph_{\overline{Y},1} \}}$.

Since $g$ is \sm, we may assume $g = \id_Y$. In that case, denote the closure of $W$ in $\overline{Y}$ by $\overline{W}$. By applying the embedded desingularization theorem on $\overline{W} \embed \overline{Y}$, we may assume $\overline{W}$ is \sm. Then we can take
$$[\blowup{\overline{Y}}{\overline{W}} \to \overline{Y}] - [\id_{\overline{Y}}] + [Q_1' \to \overline{W} \embed \overline{Y}] - [Q_2' \to \overline{W} \embed \overline{Y}],$$
where $Q_1' \defeq \P(\O \oplus \dual{\nbundle{\overline{W}}{\overline{Y}}}) \to \overline{W}$ and $Q_2' \defeq \P(\O \oplus \O(1)) \to \P(\dual{\nbundle{\overline{W}}{\overline{Y}}}) \to \overline{W}$, as our lifting. \claimend

\medskip

By claims 1, 3, 4 and some diagram chasing, if we can show that
$$\kernel{j^*} \subset \image{i_*} + \kernel{\{\gph_{X,3} \circ \gph_{X,2} \circ \gph_{X,1}\}}$$
in the sequence
$$\basicmod{G}{F}{}{Z}_{\rm fin} \stackrel{i_*}{\longto} \basicmod{G}{F}{}{X}_{\rm fin} \stackrel{j^*}{\longto} \basicmod{G}{F}{}{U}_{\rm fin} \longto 0,$$
then we will have $\kernel{j^*} \subset \image{i_*}$ in the sequence
$$\cob{}{G}{Z}_{\rm fin} \stackrel{i_*}{\longto} \cob{}{G}{X}_{\rm fin} \stackrel{j^*}{\longto} \cob{}{G}{U}_{\rm fin} \longto 0$$
and we are done.

\medskip

\noindent Claim 5 : It is enough to consider elements of the form $x - x'$ where $x$, $x'$ are cycles and $j^*(x) = j^*(x')$.

Notice that $\basicmod{G}{F}{}{-}_{\rm fin}$ is a free $\lazard_G(F)$-module with a basis given by iso\morp\ classes of cycles. If an element in $\basicmod{G}{F}{}{X}_{\rm fin}$ lies inside $\kernel{j^*}$, we may assume it is of the form $\sum_{i=1}^n a_n x_n$ where $a_i$ are elements in $\lazard_G(F)$, $x_i$ are cycles, $j^*(x_i)$ belongs to the same iso\morp\ class for all $i$ and $\sum_{i=1}^n a_n = 0$. We can then rearrange the terms in the following way :
$$\sum_{i=1}^n a_n x_n = a_1(x_1 - x_n) + \cdots + a_{n-1}(x_{n-1} - x_n) + (a_1 + \cdots + a_{n-1} + a_n) x_n.$$
That proves the claim. \claimend

\medskip

By claim 5, it is enough to consider elements of the form
$$[f : Y \to X, \L_1, \ldots, \L_r] - [f' : Y' \to X, \L_1', \ldots, \L_r']$$
where $Y$, $Y' \in$ $\gsmcat{G}$\ are both \girred{G}\ and there is an \equi\ isomorphism $\gps : Y|_U \to Y'|_U$ such that $f' \circ \gps = f$ and $\gps^* (\L_i'|_{Y'|_U}) \cong \L_i|_{Y|_U}$ for all $i$. Let $Y'' \subset Y \x_X Y'$ be the closure of the graph of $\gps$. By resolution of singularities, we may assume $Y''$ to be \sm\ and we have a commutative diagram
\squarediagramword{Y''}{Y}{Y'}{X}{\gm}{\gm'}{f}{f'}
with \equi, \proj, birational maps $\gm$, $\gm'$ which are isomorphisms over $U$. By weak factorization Theorem, $\gm : Y'' \to Y$ can be factored into a series of blow ups or blow downs along \inv\ \sm\ centers. By the \textbf{(Blow)} axiom, 
$$[Y'' \to X, \gm^* \L_1, \ldots, \gm^* \L_r] - [Y \to X, \L_1, \ldots, \L_r] \equiv 0$$
mod $\image{i_*} + \kernel{\{\gph_{X,3} \circ \gph_{X,2} \circ \gph_{X,1}\}}$. Similarly, 
$$[Y'' \to X, \gm'^* \L_1', \ldots, \gm'^* \L_r'] - [Y' \to X, \L_1', \ldots, \L_r'] \equiv 0.$$
That reduces the case to $Y = Y'$ and $f = f'$, with sheaves $\L_i$, $\L_i' \in \picard{G}{Y}$ such that $\L_i \cong \L_i'$ over $U$. In this case,
\begin{eqnarray}
[f : Y \to X, \L_1, \ldots, \L_r] &=& f_* \circ c(\L_1) \circ \cdots c(\L_r) [\id_Y] \nonumber\\
&=& f_* \circ c(\L_1) \circ \cdots c(\L_{r-1}) \circ c(\L_r' \otimes (\L_r \otimes \dual{\L_r'})) [\id_Y] \nonumber\\
&\equiv& f_* \circ c(\L_1) \circ \cdots c(\L_{r-1}) \circ c(\L_r')[\id_Y] + f_* \circ \gs \circ c(\L_r \otimes \dual{\L_r'}) [\id_Y] \nonumber\\
&& \text{for some $\gs$} \nonumber\\
&\equiv& f_* \circ c(\L_1) \circ \cdots c(\L_{r-1}) \circ c(\L_r')[\id_Y] \nonumber
\end{eqnarray}
by Lemma \ref{lemma sheaf trivial over open subsch}. Hence, by repeating the same argument for all $\L_i$, we have
$$f_* \circ c(\L_1) \circ \cdots c(\L_r) [\id_Y] \equiv f_* \circ c(\L_1') \circ \cdots c(\L_r') [\id_Y],$$
as we want. That finishes the proof of Theorem \ref{thm localization property}.
\end{proof}

The next property we would like to show is the homotopy invariance property. We will handle the surjectivity first.

\begin{prop}
\label{prop homo invar surj}
Suppose $X$ is an object in $\gvar{G}$, $V$ is a finite dimensional $G$-\repn\ and $\pi : X \x V \to X$ is the projection. Then the induced map
$$\pi^* : \cob{}{G}{X} \to \cob{}{G}{X \x V}$$
is surjective.
\end{prop}

\begin{proof}
W\withoutlog, we may assume $\dim V = 1$. By Theorem \ref{thm gen by geo cycle}, it is enough to consider elements of the form  $[f : Y \to X \x V]$ where $Y \in$ $\gsmcat{G}$\ is \girred{G}. By Proposition \ref{prop equi embed} and the localization property, we may assume $X$ is \proj.

Consider the \proj\ map $\pi_2 \circ f : Y \to V$. Since $Y$ is \girred{G}, $\image{\pi_2 \circ f}$ is a reduced, \girred{G}, closed subscheme of $V$. So, the image is either $V$ or of dimension 0.

\medskip

\noindent Case 1 : Dimension of $ \image{\pi_2 \circ f} = 0$.

Let $f_1 : Y \to X$ and $f_2 : Y \to V$ be the maps $\pi_1 \circ f$ and $\pi_2 \circ f$ respectively. Since $f_2 : Y \to \image{f_2}$ is \proj, $Y$ is \proj. Consider the map $\phi : Y \x V \to X \x V$ defined by sending $(y,v)$ to $(f_1(y), f_2(y)-v)$. It is clear that $\phi$ is \equi. Notice that sending $(y,v)$ to $(y,f_2(y)-v)$ defines an iso\morp\ $g : Y \x V \iso Y \x V$ and $\phi = (f_1 \x \id_V) \circ g$. Since $Y$ is \proj, $f_1$ is \proj\ and so is $\phi$.

Notice that
$$[Y \x 0 \embed Y \x V] = c(\O(Y \x 0))[\id_{Y \x V}],$$
in which we consider $Y \x 0$ as an \inv\ divisor on $Y \x V$. Since the divisor $Y \x 0$ is the pull-back (along $\pi_2 : Y \x V \to V$) of the divisor $[0]$ on $V$ and $\O_V([0]) \cong \beta$ for some character $\gb$, we have
$$[Y \x 0 \embed Y \x V] = c(\beta)[\id_{Y \x V}] = e(\beta)[\id_{Y \x V}]$$
by Proposition \ref{prop Chern class basic}. Apply $\phi_*$ on both sides ($\phi$ is \proj), we have
$$[f : Y \to X \x V] = e(\beta)[\phi : Y \x V \to X \x V] = e(\beta)\pi^* [f_1 : Y \to X]$$
and we are done.

\medskip

\noindent Case 2 : $\image{\pi_2 \circ f} = V$.

We will proceed by induction on $\dim Y$. If $\dim Y = 0$, then $\pi_2 \circ f$ can not be surjective. Suppose $\dim Y > 0$. Consider the following Cartesian square :
\squarediagramword{Z}{V \x \A^1}{Y}{V}{p_1}{p_2}{m}{\pi_2 \circ f}
\noindent where $\A^1$ is equipped with trivial $G$-action and $m(x,y) \defeq xy$. Denote the fiber of $\pi_2 \circ f$ over zero by $Y_0$ and the fibers of $\pi_2 \circ p_1 : Z \to \A^1$ over $0, 1$ by $Z_0$, $Z_1$ respectively. Notice that $\pi_2 \circ f$, $\pi_2 \circ p_1$ are both flat and $Z$ is equidimensional.

\medskip

\noindent Claim 1 : There exist \inv\ closed subschemes $D_i \subset Y_0$, $D_i' \subset Z_0$ and integers $m_i$ such that $D_i$, $D_i'$ are $G$-prime divisors on $Y$, $Z$ respectively, $D_i' \cong D_i \x V$ as $G$-schemes, 
$$(\pi_2 \circ f)^* \O_V([0]) \cong \O_Y(\sum_i m_i D_i),$$
$$(\pi_2 \circ p_1)^* \O_{\A^1}([0]) \cong \O_Z(\sum_i m_i D_i')$$
and $\singular{Z} \subset Z_0$ with dimension $< \dim D_i$.

\medskip

Let $\spec{R}$ be an irreducible, affine open subscheme of $Y$ and $m^* : k[t] \to k[x,y]$ be the ring homo\morp\ corresponding to $m$. Then, $\pi_2 \circ f$ corresponds to an injective ring homo\morp\ $k[t] \to R$ which sends $t$ to some non-zero element $a \in R$. Therefore, $Z$ is locally given by $\spec{R[x,y]\, /\, (a - xy)}$. Let $P_i$ be the minimal prime ideals of $R / (a)$. 

Observe that $(\pi_2 \circ f)^*[0]$ defines a Cartier divisor, which induces a Weil divisor. If we take $D_i$ to be the divisor given by $P_i \subset R$ and $m_i$ to be the length of $R_{P_i} / (a)$, then $(\pi_2 \circ f)^*[0] = \sum_i m_i D_i$. Similarly, since the minimal prime ideals of 
$$R[x,y]\, /\, (a - xy, y) \cong R[x] / (a)$$
are $Q_i \defeq P_i \cdot (R[x] / (a))$, if we take $\tilde{Q}_i$ to be the preimage of $Q_i$ of the map 
$$R[x,y]\, /\, (a - xy) \to R[x,y]\, /\, (a - xy,y) \cong R[x] / (a)$$
and $D_i'$ to be the divisor given by $\tilde{Q}_i$, then $D_i' \cong D_i \x V$ and $(\pi_2 \circ p_1)^*[0] = \sum_i m_i D_i'$.

Over $V \x (\A^1 - 0)$, $Z$ is locally given by $R[x,y][y^{-1}]\, /\, (a - xy) \cong R[y][y^{-1}]$, which is regular. So, $p_1^{-1}(V \x (\A^1-0))$ is \sm. For the same reason, $p_1^{-1}((V - 0) \x \A^1)$ is also \sm. Hence, $\singular{Z} \subset p_1^{-1}(0 \x 0)$, which is a closed subscheme of $Z$ with codimension 2 ($p_1$ is flat). \claimend

\medskip

By resolution of singularities, there exists an \equi, \proj, birational map $Z' \to Z$ such that $Z' \in$ $\gsmcat{G}$. Then, we apply the embedded desingularization on $\cup_i \stricttransform{D_i'} \embed Z'$ to obtain $Z'' \to Z'$. Denote the composition $Z'' \to Z' \to Z$ by $q$. We then have 
\begin{eqnarray}
\label{eqn 30}
(\pi_2 \circ p_1 \circ q)^* \O_{\A^1}([0]) = \O_{Z''}(\sum_i m_i \stricttransform{\stricttransform{D_i'}} + \sum_j \pm \stricttransform{E_j}) 
\end{eqnarray}
for some exceptional divisors $E_j$. Resolve the singularities of $D_i$ to obtain $\tilde{D_i} \in$ $\gsmcat{G}$. 

Now, consider the following commutative diagram :

\medskip

\begin{center}
$\begin{CD}
\stricttransform{\stricttransform{D'_i}} @>>> Z'' @>>> Z' @>>> Z \\
@V{\phi}VV @. @. @| \\
\tilde{D_i} \x V @>>> D_i \x V @>{\rm iso.}>> D_i' @>>> Z 
\end{CD}$
\end{center}

\medskip

\noindent where $\phi$ is a \equi, birational map (may not be regular). Let $g : Z \to X \x V \x \A^1$ be the map defined by sending $z$ to $(\pi_1 \circ f \circ p_2(z) , p_1(z))$. It is \proj\ because $p_1$ is \proj. Precomposing it with $q$, we got a \proj\ map $h : Z'' \to X \x V \x \A^1$. By Proposition \ref{prop smooth extension}, we can extend it to a \proj\ map $\overline{h} : \overline{Z''} \to X \x V \x \P^1$ (trivial action on $\P^1$). Let $p : \overline{Z''} \to X \x V$ be the \proj\ map given by composing $\overline{h}$ with projection. Then we have
\begin{eqnarray}
[f : Y \to X \x V] &=& [Y \cong (\pi_3 \circ \overline{h})^{-1}(1) \embed \overline{Z''} \stackrel{p}{\longto} X \x V] \nonumber\\
&=& p_* (\, c((\pi_3 \circ \overline{h})^* \O_{\P^1}([1]))\, [\id_{\overline{Z''}}]\, ) \nonumber\\
&=& p_* (\, c((\pi_3 \circ \overline{h})^* \O_{\P^1}([0]))\, [\id_{\overline{Z''}}]\, ) \nonumber\\
&=& p_* (\, c(\O_{\overline{Z''}}(\sum_i m_i \stricttransform{\stricttransform{D_i'}} + \sum_j \pm \stricttransform{E_j}))\, [\id_{\overline{Z''}}]\, ) \nonumber
\end{eqnarray}
by equation (\ref{eqn 30}) ($Z''$, $\overline{Z''}$ have the same fibers over $\A^1$). Hence,
$$[f : Y \to X \x V] = \sum_i p_* \circ \sigma_i\, [\stricttransform{\stricttransform{D_i'}} \embed \overline{Z''}] + \sum_j p_* \circ \sigma_j\, [\stricttransform{E_j} \embed \overline{Z''}]$$
for some $\sigma_i$, $\sigma_j \in \Endoinf{\cob{}{G}{\overline{Z''}}}$. Notice that $\dim Y = \dim \stricttransform{\stricttransform{D_i'}} = \dim \stricttransform{E_j}$.

For simplicity, we will say $x \equiv 0$ if $x$ is an element in $\image{\pi^*} \subset \cob{}{G}{X \x V}$. By Proposition \ref{prop Nilp axiom}, Theorem \ref{thm gen by geo cycle} and the induction assumption, 
$$[f : Y \to X \x V] \equiv \sum_i p_* a_i\, [\stricttransform{\stricttransform{D_i'}} \embed \overline{Z''}] + \sum_j p_* a_j'\, [\stricttransform{E_j} \embed \overline{Z''}]$$
for some $a_i$, $a_j' \in \lazard_G(F)$. In other words, it is enough to show 
$$[\stricttransform{\stricttransform{D_i'}} \to Z \to X \x V] \equiv [\stricttransform{E_j} \to Z \to X \x V] \equiv 0.$$

By the weak factorization Theorem, the rational map $\phi : \stricttransform{\stricttransform{D_i'}} \ratmap \tilde{D_i} \x V$ can be considered as the composition of a series of blow ups or blow downs along \inv\ \sm\ centers. For simplicity, assume $\phi$ is given by a single blow up along some \inv\ \sm\ center $C$. Then, by the blow up relation, the difference $[\stricttransform{\stricttransform{D_i'}} \to X \x V] - [\tilde{D_i} \x V \to X \x V]$ is given by elements of the form $[\P \to C \to X \x V]$ where $\P \to C$ is a \qadtower. Since $\dim Y = \dim \stricttransform{\stricttransform{D_i'}} = \dim \P > \dim C$, we have $[\P \to C \to X \x V] \equiv 0$ by Proposition \ref{prop tower reduction} and the induction assumption. Hence,
$$[\stricttransform{\stricttransform{D_i'}} \to X \x V] \equiv [\tilde{D_i} \x V \to X \x V] = [\tilde{D_i} \to D_i \embed Y \to X] \x [\id_V] \equiv 0.$$

Since $\stricttransform{E_j}$ is the strict transform of the exceptional divisor $E_j$ corresponding to a certain blow up in the process of blowing up $Z$ to obtain $Z'$, or blowing up $Z'$ to obtain $Z''$, there is an object $C \in$ $\gsmcat{G}$\ such that $E_j \to C$ is a \qadtower, $\dim E_j > \dim C$ and the map $E_j \to X \x V$ factors through $E_j \to C$. Hence, by the same argument as before, 
$$[\stricttransform{E_j} \to E_j \to X \x V] \equiv [E_j \to X \x V] = [E_j \to C \to X \x V] \equiv 0.$$
That finishes the proof of Proposition \ref{prop homo invar surj}.
\end{proof}

\begin{thm}
\label{thm homotopy invariance}
Suppose $\char{k} = 0$ and $k$ contains a primitive $e$-th \rou, where $e$ is the exponent of $G_f$. 

For any object $X \in \gvar{G}$ and finite dimensional $G$-\repn\ $V$, the map
$$\pi^* : \cob{}{G}{X} \to \cob{}{G}{X \x V}$$
induced by the projection $\pi : X \x V \to X$ is an iso\morp.  
\end{thm}

\begin{proof}
By Proposition \ref{prop homo invar surj}, it is enough to prove the injectivity of $\pi^*$. W\withoutlog, we may assume $\dim V = 1$. By the localization property, we have the following exact sequence :
$$\cob{}{G}{X} \stackrel{i_*}{\longrightarrow} \cob{}{G}{X \x \P(V \x \A^1)} \stackrel{j^*}{\longrightarrow} \cob{}{G}{X \x V} \to 0$$
where $\A^1$ is equipped with trivial $G$-action, $i : X \cong X \x \infty \embed X \x \P(V \x \A^1)$ and $j : X \x V \embed X \x \P(V \x \A^1)$ are the immersions. 

Let $p : X \x \P(V \x \A^1) \to X$ be the projection and $\phi : \cob{}{G}{X \x \P(V \x \A^1)} \to \cob{}{G}{X}$ be the map $p_* \circ c(\L)$ where $\L \defeq \pi_2^* \O_{\P(V \x \A^1)}([0])$. Then, for any element $x \in \cob{}{G}{X}$,
\begin{eqnarray}
\phi \circ i_*(x) &=& p_* \circ c(\L) \circ i_*(x) \nonumber\\
&=& p_* \circ i_* \circ c(\L|_{X \x \infty}) (x) \nonumber\\
&=& p_* \circ i_* \circ c(\O_X) (x) = 0. \nonumber
\end{eqnarray}
Therefore, $\phi$ defines a map $\cob{}{G}{X \x V} \to \cob{}{G}{X}$. Now, for any element $[Y \to X] \in \cob{}{G}{X}$,
\begin{eqnarray}
\phi \circ \pi^*[Y \to X] &=& \phi\, [Y \x V \to X \x V] \nonumber\\
&=& p_* \circ c(\L) [Y \x \P(V \x \A^1) \to X \x \P(V \x \A^1)] \nonumber\\
&=& p_* (\, [Y \to X] \x c(\O_{\P(V \x \A^1)}([0])) [\id_{\P(V \x \A^1)}]\, ) \nonumber\\
&=& p_* (\, [Y \to X] \x [0 \embed \P(V \x \A^1)]\, ) \nonumber\\
&=& [Y \to X]. \nonumber
\end{eqnarray}
Hence, $\pi^*$ is injective by Theorem \ref{thm gen by geo cycle}.
\end{proof}

\bigskip
\bigskip

\section{Special theory and some advanced properties}
\label{sect special theory}

As in the non-equivariant configuration in \cite{universal alg cobor}, we would like the projective bundle formula and the extended homotopy property to hold in our theory $\cob{G}{}{-}$ as well. But notice that the proof of the extended homotopy property in \cite{universal alg cobor} relies on the projective bundle formula and the projective bundle formula ultimately comes from fact that in Topology, we have
$$MU(\P^{\infty}) \cong MU[[y]]$$
where $y \defeq c_1(\O_{\P^{\infty}}(1))$, which is only true without group action (see Theorem 9.6 in \cite{equi FGL 2} for the equivariant analogue). Therefore, we should not expect the projective bundle formula to hold in our theory $\cob{G}{}{-}$. But there is a remedy. 

In this section, we will define a notion called special theory, denoted as $\cob{G}{s}{-}$. Then, we will show the projective bundle formula and extended homotopy property in $\cob{G}{s}{-}$. That will then allow us to define the higher Chern class operators of a \glin{G}\ locally free sheaf of arbitrary finite rank. The importance of the higher Chern class operators will be justified by one of our main Theorems : the \equi\ Conner-Floyd isomorphism (see the proof of Theorem \ref{thm equi Conner Floyd}).

First of all, we call a $(G,F)$-\fgl\ over $R$ special if, for all $\alpha \in G^*$,
$$u(\alpha) \defeq d(\alpha)^1_1 \in R \text{ is a unit and } d(\alpha)^1_s = 0 \text{ for } s \geq 2.$$
Then, there is a universal ring 
$$\lazard_G^s(F) \defeq S^{-1} \lazard_G(F) / (d(\alpha)^1_s | s \geq 2),$$
where $S$ is the multiplicative set generated by $d(\alpha)^1_1$, and a canonical ring \homo
$$\lazard_G(F) \to \lazard_G^s(F).$$

For a multiplicative $(G,F)$-\fgl, by Lemma 2.1 in \cite{multi FGL}, 
$$l_{\ga}\, y(\gep) = y(\ga) = e(\ga) + (1 - v e(\ga))y(\gep).$$
In other words, $d(\ga)^1_0 = e(\ga)$, $d(\ga)^1_1 = 1 - v e(\ga)$ and $d(\ga)^1_s = 0$ when $s \geq 2$. Moreover, by Proposition 4.5 in \cite{multi FGL},
$$\lazard_G^m(F)[v^{-1}] \cong R(G)[v,v^{-1}]$$
with $d(\ga)^1_1 = 1 - v e(\ga) = \dual{\alpha}$ (a unit). Hence, we can conclude that there is a canonical surjective ring \homo
$$\lazard_G^s(F) \to \lazard_G^m(F)[v^{-1}] \cong R(G)[v,v^{-1}].$$

Furthermore, since the additive $(G,F)$-\fgl\ corresponds to setting $v=0$ in the multiplicative $(G,F)$-\fgl\ and $d(\ga)^1_1 = 1 - v e(\ga) = 1$ is a unit, there is also a canonical surjective ring \homo
$$\lazard_G^s(F) \to \lazard_G^a(F).$$
In conclusion, the special theory will cover the \equi\ K-theory and the \equi\ algebraic cobordism theory associated to the additive $(G,F)$-\fgl, at the very least.

Now, let us define our special theory as
$$\cob{G}{s}{-} \defeq \lazard_G^s(F) \otimes_{\lazard_G(F)} \cob{G}{}{-}$$
with objects in $\gvar{G}$. The four basic operations is clearly well-defined in $\cob{G}{s}{-}$. More importantly, by the \textbf{(EFGL)} axiom, we will have
\begin{equation}
\label{eqn special FGL}
c(\L \otimes \alpha) = d(\ga)^1_0 + d(\alpha)^1_1 c(\L) = e(\ga) + u(\alpha) c(\L)
\end{equation}
as operators on $\cob{G}{s}{X}$, for any $X \in \gvar{G}, \L \in \picard{G}{X}$ and $\alpha \in G^*$.

We are now in position to show that the projective bundle formula holds in $\cob{G}{s}{-}$. Suppose $X \in \gvar{G}$, $\cat{E}$ is a \glin{G}\ locally free sheaf of rank $r+1$ over $X$, 
$$\pi : \P \defeq \P(\cat{E}) \to X$$
be the projection map. Let 
$$\phi_i \defeq c(\O_{\P}(1))^{i} \circ \pi^* : \cob{G}{s}{X} \to \cob{G}{s}{\P}$$
for $0 \leq i \leq r$ and 
$$\Phi \defeq \oplus_{i=0}^r \phi_i : \oplus_{i=0}^r \cob{G}{s}{X} \to \cob{G}{s}{\P}.$$

\begin{thm}
\label{thm PBF}
Suppose $\char{k} = 0$. Then,
$$\Phi : \oplus_{i=0}^r \cob{G}{s}{X} \to \cob{G}{s}{\P}$$
is an isomorphism.
\end{thm}

\begin{proof}
For simplicity of notation, let $R \defeq \lazard_G^s(F)$. Within this proof, we say that a map $\psi : \cob{G}{s}{\P} \to \cob{G}{s}{X}$ is extendable if it can be written as the sum of compositions of $a \in R, c(\O_{\P}(1)), \pi_*$ or $\pi^*$. For a map $\Psi : \cob{G}{s}{\P} \to \oplus_{i=0}^r \cob{G}{s}{X}$, we say that it is extendable if it can be written as the sum of maps of the form
$$\Psi_1 \circ \Phi \circ \Psi_2 \circ \Phi \circ \cdots \circ \Psi_n$$
such that, for all $1 \leq i \leq n$, the map $\Psi_i$ is the direct sum of extendable maps $\cob{G}{s}{\P} \to \cob{G}{s}{X}$. 

The significance of extendable maps is that if $Z \in \gvar{G}$ is an \inv\ closed subscheme of $X$, $U \defeq X - Z, i : Z \embed X, j : U \embed X$ are the immersions and
$$f : \cob{G}{s}{\P|_Z} \to \oplus_{i=0}^r \cob{G}{s}{Z}$$ 
$$g : \cob{G}{s}{\P|_U} \to \oplus_{i=0}^r \cob{G}{s}{U}$$
are extendable, then there are extendable maps
$$f', g' : \cob{G}{s}{\P} \to \oplus_{i=0}^r \cob{G}{s}{X}$$
such that 
$$f' \circ i_* = i_* \circ f \text{ and } g \circ j^* = j^* \circ g'.$$

Instead of proving $\Phi$ to be an isomorphism, we will indeed show a stronger statement :
$$\Phi \text{ is surjective and has an extendable left inverse } \Psi.$$

Let us first consider the trivial case : $\cat{E} \cong \beta_0 \oplus \cdots \oplus \beta_r$ for some $\beta_i \in G^*$. Let $H_i \defeq \P(\beta_0 \oplus \cdots \oplus \hat{\beta_i} \oplus \cdots \oplus \beta_r)$ (omitting) and consider it as an \inv\ smooth divisor on $\P$. Define
$$\xi^i_j \defeq c(\O_{\P}(H_i)) \circ c(\O_{\P}(H_{i+1})) \circ \cdots \circ c(\O_{\P}(H_j))$$
if $0 \leq i \leq j \leq r$ and $\xi^i_j \defeq \id$ otherwise. Moreover, $\overline{\phi}_i \defeq \xi^0_{i-1} \circ \pi^*$ for $0 \leq i \leq r$ and 
$$\overline{\Phi} \defeq \oplus_{i=0}^r \overline{\phi}_i : \oplus_{i=0}^r \cob{G}{s}{X} \to \cob{G}{s}{\P}.$$

\medskip

\noindent Claim 1 : There exists an upper triangular matrix $M \in GL(r+1, R)$ (units on the diagonal) such that 
$$\Phi = \overline{\Phi} \circ M.$$

Notice that since $\O_{\P}(H_i) \otimes \beta_i \cong \O_{\P}(1)$ for all $i$, we have
$$c(\O(H_i) \otimes \gamma) \circ \pi^*= c(\O(H_0) \otimes \gamma') \circ \pi^* = (a + u c(\O(H_0))) \circ \pi^* = a \overline{\phi}_0 + u \overline{\phi}_1$$
(by equation (\ref{eqn special FGL})) for some $\gamma' \in G^*, a,u \in R$ such that $u$ is a unit. By induction, one can show that, for all $0 \leq i_j \leq r, \gamma_k \in G^*$ and $1 \leq n \leq r$,
$$c(\O(H_{i_1}) \otimes \gamma_1) \circ \cdots \circ c(\O(H_{i_n}) \otimes \gamma_n) \circ \pi^* = a_0 \overline{\phi}_0 + \cdots + a_{n-1} \overline{\phi}_{n-1} + u \overline{\phi}_n$$
for some $a_i, u \in R$ such that $u$ is a unit. Therefore, 
$$\phi_0 = u_{00} \overline{\phi}_0 \text{ with } u_{00}=1$$
and for $1 \leq n \leq r$,
\begin{eqnarray*}
\phi_{n} &=& c(\O_{\P}(1))^{n} \circ \pi^* \\
&=& c(\O_{\P}(H_0) \otimes \beta_0)^{n} \circ \pi^* \\
&=& a_{0,n} \overline{\phi}_0 + \cdots + a_{n-1,n} \overline{\phi}_{n-1} + u_{nn} \overline{\phi}_n
\end{eqnarray*}
and $a_{i,j}, u_{i,i}$ will be the entries of $M$. \claimend

\medskip

Next, we will define an extendable left inverse of $\overline{\Phi}$. Let $\overline{\psi}_0 \defeq \pi_* \circ \xi^1_r,$
$$\overline{\psi}_{n+1} \defeq \pi_* \circ \xi^{n+2}_r (\id - \sum_{k=0}^n \overline{\phi}_k \overline{\psi}_k)$$
for $0 \leq n \leq r-1$ (defined inductively) and $\overline{\Psi} \defeq \oplus_{i=0}^r \overline{\psi}_i$. Then it is not hard to see that $\overline{\Psi}$ is a left inverse of $\overline{\Phi}$ (see section 3.5.2 and 3.5.3 in \cite{universal alg cobor} for more details). Moreover, since
$$c(\O_{\P}(H_i)) = c(\O_{\P}(1) \otimes \dual{\beta_i}) = e(\dual{\beta_i}) + u(\dual{\beta_i}) c(\O_{\P}(1))$$
is extendable, $\overline{\psi}_0$ is extendable and so is $\overline{\psi}_{n+1}$ (by induction). Hence,
$$\Psi \defeq M^{-1} \overline{\Psi}$$
is an extendable left inverse of $\Phi$.

For the surjectivity of $\Phi$ (when $\cat{E}$ is trivial), we consider the following diagram (1) (not necessarily commutative) :

\begin{center}$\begin{CD}
\cob{G}{s}{H_r}    @>{i_*}>> \cob{G}{s}{\P} @>{j^*}>> \cob{G}{s}{\P - H_r} @>>> 0 \\
@A{\Phi}AA       @A{\Phi}AA @A{\pi^*}AA @. \\
\oplus_{i=0}^{r-1} \cob{G}{s}{X}   @>{M'}>> \oplus_{i=0}^{r} \cob{G}{s}{X} @>{a}>>  \cob{G}{s}{X} @.
\end{CD}$\end{center}
where 
$$M' = 
\begin{pmatrix}
e(\dual{\beta_r}) & 0 & 0 & \cdots & 0 \\
u(\dual{\beta_r}) & e(\dual{\beta_r}) & 0 & \cdots & 0 \\
0 & u(\dual{\beta_r}) & e(\dual{\beta_r}) & \cdots & 0 \\
\vdots & \vdots & \vdots & \ddots & \vdots \\
0 & 0 & 0 & \cdots & e(\dual{\beta_r}) \\
0 & 0 & 0 & \cdots & u(\dual{\beta_r}) \\
\end{pmatrix}$$
is a $(r+1) \x r$ matrix and $a$ is the projection of the $i=0$ part. The right square is commutative on the $i=0$ part.

\medskip

\noindent Claim 2 : The left square of diagram (1) is commutative.

By Theorem \ref{thm gen by geo cycle}, it is enough to consider elements of the form $[f : Y \to X] = f_*[\id_X]$. By the naturality of the definitions of $\P = \P(\beta_0 \oplus \cdots \oplus \beta_r)$ and $H_r = \P(\beta_0 \oplus \cdots \oplus \beta_{r-1})$, we may assume $X \in \gsmcat{G}$ and $f = \id_X$. Then,
\begin{eqnarray*}
i_* \circ \phi_i [\id_X] &=& i_* \circ c(\O_{H_r}(1))^i [\id_{H_r}]  \\
&=& c(\O_{\P}(1))^i \circ i_* [\id_{H_r}] \\
&=& c(\O_{\P}(1))^i c(\O_{\P}(H_r)) [\id_{\P}] \\
&=& c(\O_{\P}(1))^i (e(\dual{\beta_r}) + u(\dual{\beta_r}) c(\O_{\P}(1))) [\id_{\P}] \\
&=& (e(\dual{\beta_r}) c(\O_{\P}(1))^i + u(\dual{\beta_r}) c(\O_{\P}(1))^{i+1}) [\id_{\P}] \\
&=& (e(\dual{\beta_r}) \phi_i + u(\dual{\beta_r}) \phi_{i+1}) [\id_X]
\end{eqnarray*}
and we are done. \claimend

\medskip

Furthermore, for diagram (1), the top row is exact by the localization property. The right column is an isomorphism by the homotopy invariance property. The left column is an isomorphism by induction ($r=0$ case is trivial). The surjectivity of the middle column then follows from some diagram chasing. That handles the case when $\cat{E}$ is trivial.

For general $\cat{E}$, we need the following basic fact.

\medskip

\noindent Claim 3 : There exists a filtration of \inv, closed subschemes $X_j \in \gvar{G}$
$$\emptyset = X_0 \subset X_1 \subset \cdots \subset X_m = X$$
such that $\cat{E}|_{X_j - X_{j-1}}$ is trivial for all $1 \leq j \leq m$. 

By Noetherian induction, it is enough to show that there exists an \inv, open subscheme $U \subset X$ such that $\cat{E}|_U$ is trivial. W\withoutlog, we may assume $X$ is \girred{G}\ and \sm. The case when $\rank{\cat{E}} = r+ 1 = 1$ is handled by Proposition \ref{prop lb structure}.

Suppose $r \geq 1$. By Theorem \ref{thm splitting principle}, we may further assume that there exists an exact sequence
$$0 \to \L \to \cat{E} \to \cat{E}/\L \to 0$$
of \glin{G}\ locally free sheaves over $X$ of ranks 1, $r+1$ and $r$ respectively. By Proposition \ref{prop lb structure} and the induction assumption on $r$, we may assume $\L$ and $\cat{E} / \L$ are both trivial. It will then be enough to construct a $G$-\equi\ splitting map $\cat{E} / \L \to \cat{E}$.

By Proposition \ref{prop contain inv aff}, we may assume $X = \spec{A}$. Suppose $\cat{E} / \L \cong \oplus_{i = 1}^r A \beta_i$ for some $\beta_i \in G^*$. By considering the $\beta_i$-component of the exact sequence of $G$-\repn s, we have a surjective map
$$\cat{E}_{\beta_i} \to (\cat{E}/\L)_{\beta_i} \cong (A_{\epsilon}) \beta_i \oplus \bigoplus_{j \neq i} (A_{\dual{\beta_j} \beta_i}) \beta_j.$$
We may then pick a lifting $m_i \in \cat{E}_{\beta_i}$ of $\beta_i$ to define our splitting map and it will be $G$-\equi.  \claimend

\medskip

Let 
$$\emptyset = X_0 \subset X_1 \subset \cdots \subset X_m = X$$
be a filtration given by clam 3. Consider the following diagram (2) (not necessarily commutative) :
$$\begin{CD}
0 @>>> i_* \cob{G}{s}{\P|_Z} @>>> \cob{G}{s}{\P|_Y} @>{j^*}>> \cob{G}{s}{\P|_U} @>>> 0 \\
@. @A{\Phi}AA @A{\Phi}AA @A{\Phi}AA @. \\
0 @>>> {\oplus_{i=0}^r i_* \cob{G}{s}{Z}} @>>> {\oplus_{i=0}^r  \cob{G}{s}{Y}} @>{j^*}>> {\oplus_{i=0}^r  \cob{G}{s}{U}} @>>> 0 \\
\end{CD}$$
where $Y = X_j$, $Z = X_{j-1}$ and $U = Y - Z$ and the following diagram (3) (not necessarily commutative) :
$$\begin{CD}
\cob{G}{s}{\P|_Z} @>{i_*}>> \cob{G}{s}{\P|_Y}  \\
@A{\Phi}AA @A{\Phi}AA  \\
{\oplus_{i=0}^r  \cob{G}{s}{Z}} @>{i_*}>> {\oplus_{i=0}^r  \cob{G}{s}{Y}}
\end{CD}$$
By the localization property, the rows in diagram (2) are exact. Diagram (3) is clearly commutative. So the left column of diagram (2) is well-defined. By induction ($Z = X_0$ case is trivial), the left column of diagram (3) is surjective. Then so is the left column of diagram (2). The right column of (2) is surjective by the trivial case. Therefore, the middle column of diagram (2) is surjective.

By induction (not needed for the $Z = X_0$ case), the left column of (3) has an extendable left inverse $\Psi_Z$. So, we can define 
$$\overline{\Psi}_Z : \cob{G}{s}{\P|_Y} \to \oplus_{i=0}^r  \cob{G}{s}{Y}$$
as its extension, i.e., diagram (3) is commutative \wrt\ $\Psi_Z, \overline{\Psi}_Z$. So, the restriction of $\overline{\Psi}_Z$
$$\overline{\Psi}_Z : i_* \cob{G}{s}{\P|_Z} \to \oplus_{i=0}^r  i_*\cob{G}{s}{Z}$$
is well-defined. Therefore, the left square of diagram (2) \wrt\ $\overline{\Psi}_Z$ is commutative. By the trivial case, the right column of diagram (2) also has an extendable left inverse $\Psi_U$. So, we can define 
$$\overline{\Psi}_U : \cob{G}{s}{\P|_Y} \to \oplus_{i=0}^r  \cob{G}{s}{Y}$$
as its extension, i.e., the right square of diagram (2) \wrt\ $\Psi_U, \overline{\Psi}_U$ is commutative. By some diagram chasing, $\overline{\Psi}_Z$ is a left inverse of $\Phi$ on the left column of diagram (2). Moreover, $\overline{\Psi}_Z, \overline{\Psi}_U$ are clearly extendable.

Now, for an element $x \in {\oplus_{i=0}^r  \cob{G}{s}{Y}}$, we have
$$y \defeq x - \overline{\Psi}_U \circ \Phi(x) \in \kernel{j^*}$$
By the exactness of the bottom row of diagram (2), $y$ lies inside ${\oplus_{i=0}^r i_* \cob{G}{s}{Z}}$. Then,
$$y = \overline{\Psi}_Z \circ \Phi(y).$$
Therefore,
$$x - \overline{\Psi}_U \circ \Phi(x) = \overline{\Psi}_Z \circ \Phi(x - \overline{\Psi}_U \circ \Phi(x))$$
$$x  = \overline{\Psi}_Z \circ \Phi(x) + \overline{\Psi}_U \circ \Phi(x) - \overline{\Psi}_Z \circ \Phi \circ \overline{\Psi}_U \circ \Phi (x)$$
In other words, 
$$\Psi \defeq \overline{\Psi}_Z + \overline{\Psi}_U - \overline{\Psi}_Z \circ \Phi \circ \overline{\Psi}_U$$
is a left inverse of $\Phi$ of the middle column of diagram (2) and it is extendable. The result then follows by induction on $j$.
\end{proof}

Next, we will show that the extended homotopy property holds in the special theory $\cob{G}{s}{-}$. To simplify our computations, for an object $X \in \gvar{G}$, we will call a $G$-\equi\ \morp\ $f : F \to X$ a $G$-\equi\ torsor if there is an exact sequence of \glin{G}\ locally free sheaves $\cat{E}, \cat{E}'$ of finite ranks over $X$ :
$$0 \to \O_X \to \cat{E} \to \cat{E}' \to 0$$
such that $f$ is isomorphic to the projection $\P(\cat{E}) - \P(\cat{E}') \to X$.

\begin{prop}
\label{prop EH}
Suppose $\char{k} = 0$. For any $G$-\equi\ torsor $f : F \to X$ with $X \in \gvar{G}$, 
$$f^* : \cob{G}{s}{X} \to \cob{G}{s}{F}$$
is an isomorphism.
\end{prop}

\begin{proof}
By our definition, $f$ is isomorphic to $\P(\cat{E}) - \P(\cat{E}') \to X$ where
$$0 \to \O_X \to \cat{E} \to \cat{E}' \to 0$$
is exact. By considering $\P(\cat{E}')$ as an \inv\ \sm\ divisor on $\P(\cat{E})$, we have 
$$\O_{\P(\cat{E})}(\P(\cat{E}')) \cong \O_{\P(\cat{E})}(1)$$
(see the proof of Lemma 3.6.1 in \cite{universal alg cobor}). Let $\pi : \P(\cat{E}) \to X$, $\pi' : \P(\cat{E}') \to X$ be the projections and $i : \P(\cat{E}') \to \P(\cat{E})$ be the immersion. Then,
$$i_* \circ c(\O_{\P(\cat{E}')}(1))^i \circ \pi'^* = c(\O_{\P(\cat{E})}(1))^{i+1} \circ \pi^*$$
(see claim 2 of the proof of Theorem \ref{thm PBF}). By the localization property and the projective bundle formula, 
\begin{eqnarray*}
\cob{G}{s}{F} &\cong& \cob{G}{s}{\P(\cat{E})} / i_*\cob{G}{s}{\P(\cat{E}')} \\
&\cong& \oplus_{i=0}^r \cob{G}{s}{X} / i_* \oplus_{i=0}^{r-1} \cob{G}{s}{X} \\
&\cong& \cob{G}{s}{X}
\end{eqnarray*}
which is the $i=0$ level. So the isomorphism is given by $\pi^*$ as desired.
\end{proof}

As a consequence, we have the following splitting principle in $\cob{G}{s}{-}$.

\begin{prop}
\label{prop splitting principle for special}
Suppose $\char{k} = 0$. For any $Y \in \gsmcat{G}$ and \glin{G}\ locally free sheaf $\cat{E}$ of finite rank over $Y$, there exists a \sm\ morphism $f : Y' \to Y$ in $\gsmcat{G}$ such that $f^* \cat{E}$ is the direct sum of \glin{G}\ invertible sheaves and 
$$f^* : \cob{G}{s}{Y} \to \cob{G}{s}{Y'}$$
is injective.
\end{prop}

\begin{proof}
Basically, it follows from the projective bundle formula and the extended homotopy property (see Remark 4.1.2 in \cite{universal alg cobor}). But we will give some justifications here due to our slightly different definition on torsor.

Suppose
$$0 \to L \to E \to E/L \to 0$$
is an exact sequence of $G$-\equi\ vector bundles over $X$. Let $F \defeq Hom_{v.b.}(E \x_X L, L)$ and
$$j : Hom_{v.b.}(E,L) \to \P(F)$$
by sending $f$ to 
$$j(f)|_x (u,v) \defeq f|_x(u) + v$$ 
for all $x \in X, u \in E_x, v \in L_x$. Also, let
$$F' \defeq \{f \in F\ |\ f|_x(0,v)=0 \text{ for all } x \in X, v \in L_x \}.$$

Then, $F, F'$ are both $G$-\equi\ vector bundles over $X$, $\P(F')$ is an \inv\ closed subscheme of $\P(F)$, $j$ is a $G$-\equi\ immersion identifying $Hom_{v.b.}(E, L)$ to $\P(F) - \P(F')$. Moreover, if we define
$$\phi : X \x \A^1 \to F/F'$$
(trivial action on fibers of $X \x \A^1$) by sending $(x,a)$ to $f|_x(u,v) \defeq av$, then it can be shown that $\phi$ defines a $G$-\equi\ vector bundle isomorphism. Therefore,
$$0 \to F' \to F \to F/F' \cong X \x \A^1 \to 0$$
is exact. Hence, the splitting bundle $Hom_{v.b.}(E,L)$ is a $G$-\equi\ torsor in our sense. 
\end{proof}

\begin{rmk}
\label{rmk splitting principle for special}
{\rm
Suppose $\char{k} = 0$. By the standard arguments, we may further show that if 
$$0 \to \cat{E}' \to \cat{E} \to \cat{E}'' \to 0$$
is an exact sequence of \glin{G}\ locally free sheaves of finite ranks over $Y \in \gsmcat{G}$, then there exists a \sm\ morphism $f : Y' \to Y$ in $\gsmcat{G}$ such that $f^* \cat{E}', f^*\cat{E}''$ are both direct sums of \glin{G}\ invertible sheaves, $f^*\cat{E} \cong f^*\cat{E}' \oplus f^*\cat{E}''$ and 
$$f^* : \cob{G}{s}{Y} \to \cob{G}{s}{Y'}$$
is injective.
}
\end{rmk}

With the aid of the projective bundle formula, we can now define the higher Chern class operators for \glin{G}\ locally free sheaves of arbitrary finite ranks in $\cob{G}{s}{-}$.

Suppose $\char{k} = 0$. As pointed out in Remark \ref{rmk grading on equi lazard}, there is a natural grading on $\lazard_G(F)$, which induces a natural grading on $\lazard_G^s(F)$. Therefore, there is a canonical (cohomological) grading on $\cob{G}{s}{-}$ (see the proof of Theorem \ref{thm realization fct} for more details). Suppose $X \in \gvar{G}$, $\cat{E}$ is a \glin{G}\ locally free sheaf of rank $r$ over $X$ and $\pi : \P \defeq \P(\cat{E}) \to X$ be the projection. Then, by the projective bundle formula, we have an isomorphism
$$\oplus_{i = 0}^{r-1} c(\O_{\P}(1))^i \circ \pi^* : \oplus_{i = 0}^{r-1} \cob{G}{s,n-i}{X} \iso \cob{G}{s,n}{\P}$$
Therefore, we can define the higher Chern class operators of $\cat{E}$ 
$$c_i(\cat{E}) : \cob{G}{s,n}{X} \to \cob{G}{s,n+i}{X}$$
by the equality
$$\sum_{i=0}^r (-1)^i c(\O_{\P}(1))^{r-i} \circ \pi^* \circ c_i(\cat{E}) = 0$$
(as \homo s from $\cob{G}{s, n}{X}$ to $\cob{G}{s, n+r}{\P}$) with $c_0(\cat{E}) \defeq \id$.

\begin{prop}
\label{prop Chern class properties}
Suppose $X \in \gvar{G}, \L \in \picard{G}{X}, \cat{E}$ is a \glin{G}\ locally free sheaf of finite rank over $X$ and $f : X' \to X$ is a \morp\ in $\gvar{G}$. Then we have the following list of basic properties :

\noindent \begin{statementslist}
{\rm (1)} & $c_1(\L) = c(\L)$ (agrees with our usual first Chern class operator) \\
{\rm (2)} & If $f$ is \sm,  then $$c_i(f^* \cat{E}) \circ f^* = f^* \circ c_i(\cat{E})$$  \\
{\rm (3)} & If $f$ is \proj,  then $$c_i( \cat{E}) \circ f_* = f_* \circ c_i(f^* \cat{E})$$
\end{statementslist}
\end{prop}

\begin{proof}
Part (1) follows from the facts that $\pi : \P(\L) \to X$ is an isomorphism and $\pi^* \L \cong \O_{\P(\L)}(1)$. Parts (2) and (3) follow from the definition of higher Chern class operators and the projective bundle formula.
\end{proof}

\bigskip
\bigskip

\section{Comparison with other algebraic cobordism theories}
\label{sect comparison with other theories}

In this section, we will compare our \equi\ algebraic cobordism theory $\cob{}{G}{-}$ to other algebraic cobordism theories, namely, the non-\equi\ algebraic cobordism theory $\gO(-)$ as in \cite{universal alg cobor} and the \equi\ algebraic cobordism theory, which will be denoted by $\cob{\rm Tot}{G}{-}$ to avoid confusion, defined in \cite{homo equi alg cobor} (There is also a possibly equivalent definition in \cite{homo equi alg cobor 2}). 

Let us consider the theory $\Omega(-)$ first. We need to understand the relation between the universal representing ring $\lazard_G(F)$ and $\lazard$. We will use the same assumptions on $G$ and $k$ as in section \ref{sect notation}. In particular, we will not assume $\char{k} = 0$.

Recall definition 12.2 in \cite{equi FGL}. For a commutative ring $R$, denote the topological $R$-module obtained as the inverse limit of the free $R$-modules with basis $1 = y(V^0)$, $y(V^1)$, $y(V^2), \ldots, y(V^s)$ by $R\{\{F\}\}$. Then, a $(G,F)$-\fgl\ over a commutative ring $R$ is a topological $R$-module $R\{\{F\}\}$ with product given by
$$y(V^i) \cdot y(V^j) = \sum_{s \geq 0} b^{i,j}_s\, y(V^s)$$
with $b^{i,j}_s \in R$, a $G^*$-action given by 
$$l_{\ga}\, y(V^i) = \sum_{s \geq 0} d(\ga)^i_s\, y(V^s)$$
with $d(\ga)^i_s \in R$ and coproduct given by 
$$\gD y(V^i) = \sum_{s,t \geq 0} f^i_{s,t}\, y(V^s) \hat{\otimes} y(V^t)$$
with $f^i_{s,t} \in R$ satisfying the following conditions :

\medskip

\noindent \begin{statementslist}
{\rm (1)} & For all $i,s$, the coefficients $b^{i,j}_s = 0$ for sufficiently large $j$, and similarly with $i$ and $j$ exchanged. \\
{\rm (2)} & For all $\ga,s$, the coefficients $d(\ga)^i_s = 0$ for sufficiently large $i$. \\
{\rm (3)} & For all $s,t$, the coefficients $f^i_{s,t} = 0$ for sufficiently large $i$. \\
{\rm (R)} & The product is commutative, associative and unital. \\
{\rm (A)} & The action is through ring \homo s, associative and unital. \\
{\rm (T)} & The coproduct is through ring \homo s, \equi\ in the sense that $\gD \circ l_{\ga \gb} = ( l_{\ga}\, \hat{\otimes}\, l_{\gb} ) \circ \gD$, commutative, associative and unital. \\
{\rm (I)} & For all $i$, the ideal $(y(V^i))$ has additive topological basis $y(V^i)$, $y(V^{i+1}), \ldots$
\end{statementslist}

\medskip

\begin{rmk}
\label{rmk grading on equi lazard}
{\rm
By Corollary 14.3 in \cite{equi FGL}, one can show that there is a natural grading on $\lazard_G(F)$ given by
\begin{eqnarray}
\deg b^{i,j}_s &=& i + j - s, \nonumber\\
\deg d(\alpha)^i_s &=& i-s, \nonumber\\
\deg f^i_{s,t} &=& i-s-t. \nonumber
\end{eqnarray}
}
\end{rmk}

\medskip

Now, suppose $H$ is a closed subgroup of $G$ such that the pair $(H,k)$ is split (see section \ref{sect notation} for definition). We would like to define a canonical ring homo\morp\ $\lazard_G(F) \to \lazard_H(F)$. Since the induced map of character groups $G^* \to H^*$ is surjective, the complete $G$-universe $\cat{U}$ can be regarded as a complete $H$-universe and so is the complete $G$-flag $F$. Since the topological $\lazard_H(F)$-module $\lazard_H(F)\{\{F\}\}$ has product, $G^*$-action and coproduct that satisfies all the conditions listed, by the universal property of $\lazard_G(F)$, there is a canonical ring \homo
$$\gPH_{H \embed G} : \lazard_G(F) \to \lazard_H(F),$$ 
which sends structure constants in $\lazard_G(F)$ to structure constants in $\lazard_H(F)$, and it is surjective because $\lazard_H(F)$ is generated by structure constants (Corollary 14.3 in \cite{equi FGL}). 

For an object $X \in \gvar{G}$, we can then define a map
$$\gPS_{H \embed G} : \basicmod{G}{F}{}{X} \to \basicmod{H}{F}{}{X}$$ 
as the restriction of the map $\bigbasicmod{G}{F}{}{X} \to \bigbasicmod{H}{F}{}{X}$ which sends $a [Y \to X, \L_1, \ldots]$ to $\Phi_{H \embed G}(a) [Y \to X, \L_1, \ldots]$, by considering $Y \to X$ as a \proj\ map in $\gvar{H}$ and $\L_i$ as sheaves in $\picard{H}{Y}$. Notice that if we consider $\basicmod{H}{F}{}{X}$ as a $\lazard_G(F)$-module via the map $\Phi_{H \embed G}$, then $\Psi_{H \embed G}$ will be a $\lazard_G(F)$-module \homo.

\begin{prop}
Suppose $H$ is a closed subgroup of $G$ such that the pair $(H,k)$ is split. Then, for all $X \in \gvar{G}$, $\gPS_{H \embed G}$ defines a canonical $\lazard_G(F)$-module homo\morp 
$$\cob{}{G}{X} \to \cob{}{H}{X}$$ 
and it commutes with the four basic operations.
\end{prop}

\begin{proof}
The map $\gPS_{H \embed G}$ clearly commutes with the basic operations. The fact that it respects the \textbf{(Sect)} and \textbf{(EFGL)} axioms (by identifying $b^{i,j}_s$ with $\gPH_{H \embed G}(b^{i,j}_s)$ and similarly for other structure constants) follows immediately from the definition.
\end{proof}

As pointed out in example 12.3 (ii) in \cite{equi FGL}, if $G$ is the trivial group $\{1\}$, then the notion of ``$(G,F)$-\fgl'' agrees with the notion of ``\fgl'' as in \cite{universal alg cobor}. More precisely, $R\{\{F\}\} = R[[y]]$ with $y(V^i) = y^i$ and the coproduct 
$$\gD : R[[y]] \to R[[u]] \hat{\otimes} R[[v]] \cong R[[u,v]]$$
is given by
$$\gD(y) = F(u,v) = \sum_{s,t \geq 0} a_{s,t}\, u^s v^t.$$
Therefore, there is a canonical ring iso\morp\ $\overline{\gPH} : \lazard_{\{1\}}(F) \iso \lazard$ which sends $b^{i,j}_s$ to $\gd^{i+j}_s$, $d(\gep)^i_s$ to $\gd^i_s$ and $f^1_{s,t}$ to $a_{s,t}$. In addition, since 
$$(\sum_{s,t \geq 0}\ f^1_{s,t}\, y^s \otimes y^t)^i = (\gD(y))^i = \gD(y^i) = \sum_{s,t \geq 0}\ f^i_{s,t}\, y^s \otimes y^t,$$
the map $\overline{\gPH}$ will send $f^i_{s,t}$ to $a^i_{s,t}$ if we define elements $a^i_{s,t} \in \lazard$ by the equation :
$$(\sum_{s,t \geq 0}\ a_{s,t}\, u^s v^t)^i = \sum_{s,t \geq 0}\ a^i_{s,t}\, u^s v^t.$$
Thus, for any object $X \in \gvar{G}$, we have a canonical $\lazard$-module \homo
$$\overline{\gPS} : \basicmod{\{1\}}{F}{}{X} \to \Omega(X)$$ 
which sends $\sum_I a_I [Y \to X, V^{i_1}_{S_1}(\L_1), \ldots, V^{i_r}_{S_r}(\L_r)]$ to 
$$\sum_I \overline{\Phi}(a_I) [Y \to X, \L_1, \ldots, \L_1, \L_2, \ldots, \L_2, \ldots, \L_r, \ldots, \L_r]$$ 
(for each $1\leq j \leq r$, there are $i_j - \# S_j$ copies of $\L_j$). Notice that it is a finite sum because of the \textbf{(Dim)} axiom in $\Omega(-)$.

\begin{prop}
\label{prop iso to non equi theory}
Suppose $G$ is the trivial group. Then, for all $X \in \gvar{G}$, $\overline{\gPS}$ defines a canonical $\lazard$-module iso\morp\ $\cob{}{G}{X} \iso \Omega(X)$ and it commutes with the four basic operations.
\end{prop}

\begin{proof}
The map $\overline{\Psi}$ clearly commutes with the basic operations and respects the \textbf{(Sect)} axiom. For the \textbf{(EFGL)} axiom, 

\begin{eqnarray}
\overline{\gPS}(V^i(\L)V^j(\L)[\id_Y]) &=& c(\L)^i c(\L)^j[\id_Y] \nonumber\\
&=& \sum_{s \geq 0} \gd^{i+j}_s\, c(\L)^s[\id_Y] \nonumber\\
&=& \overline{\gPS}(\sum_{s \geq 0}\ b^{i,j}_s\, V^s(\L)[\id_Y]) \nonumber
\end{eqnarray}
and the rest is similar. Hence, $\overline{\gPS}$ defines a canonical $\lazard$-module \homo\ $\cob{}{G}{X} \to \Omega(X)$.

Its inverse is also naturally defined :
$$\overline{\gPS}^{-1} : \lazard \otimes_{\Z} Z(X) \to \cob{}{G}{X},$$
which sends $a[Y \to X, \L_1, \ldots, \L_r]$ to $\overline{\Phi}^{-1}(a) [Y \to X, \L_1, \ldots, \L_r]$.

Recall that $\Omega(-)$ is defined by imposing the \textbf{(Dim)}, \textbf{(Sect)} and \textbf{(FGL)} axioms on $\lazard \otimes_{\Z} Z(-)$. For the \textbf{(Dim)} axiom, we need to show the following claim.

\medskip

\noindent Claim 1 : Suppose $Y \in \smcat$ (the category of \sm, \qproj\ schemes over $k$) is irreducible, $r > \dim Y$ and $\L_1, \ldots, \L_r$ are invertible sheaves over $Y$. 
$$\overline{\gPS}^{-1}([\id_Y, \L_1, \ldots, \L_r]) = 0.$$

Let $\cat{M}_1$, $\cat{M}_2$ be two very ample invertible sheaves over $Y$ such that $\L_r \cong \cat{M}_1 \otimes \dual{\cat{M}_2}$. Then,
\begin{eqnarray}
\overline{\gPS}^{-1}([\id_Y, \L_1, \ldots, \L_r]) &=& c(\L_r)[\id_Y, \L_1, \ldots, \L_{r-1}] \nonumber\\
&=& c(\cat{M}_1 \otimes \dual{\cat{M}_2})[\id_Y, \L_1, \ldots, \L_{r-1}] \nonumber\\
&=& c(\cat{M}_1)[\id_Y, \L_1, \ldots, \L_{r-1}] + \gs \circ c(\cat{M}_2)[\id_Y, \L_1, \ldots, \L_{r-1}] \nonumber
\end{eqnarray}
for some $\gs$. So, w\withoutlog, we may assume $\L_1, \ldots, \L_r$ are all very ample. Then the result follows from an \equi\ version (with trivial group $G$) of Lemma 2.3.9(1) in \cite{universal alg cobor}. \claimend

\medskip

The map $\overline{\gPS}^{-1}$ clearly respects the \textbf{(Sect)} axiom. For the \textbf{(FGL)} axiom, $\overline{\gPS}^{-1}$ sends $c(\L \otimes \cat{M})[\id_Y]$ and $\sum_{s,t \geq 0}\, a_{s,t}\, c(\L)^s c(\cat{M})^t[\id_Y]$ to $c(\L \otimes \cat{M})[\id_Y]$ and $\sum_{s,t \geq 0}\, f^1_{s,t}\, V^s(\L) V^t(\cat{M})[\id_Y]$ respectively. Hence, it defines a canonical map 
$$\overline{\gPS}^{-1} : \gO(X) \to \cob{}{G}{X}.$$

Clearly, $\overline{\gPS} \circ \overline{\Psi}^{-1}$ is the identity. By claim 1, for any $Y \in$ $\gsmcat{G}$, $\L \in \picard{G}{Y}$ and $S$ as in (\ref{eqn basic element}), we have
$$V^n_S(\L)[\id_Y] = 0,$$
as elements in $\cob{}{G}{Y}$, for sufficiently large $n$. Therefore, $\overline{\Psi}^{-1}$ is surjective and that finishes the proof.
\end{proof}

\begin{cor}
\label{cor map to non equi theory}
For all $X \in \gvar{G}$, the ``forgetful map'' given by $\overline{\gPS} \circ \gPS_{\{1\} \embed G}$ defines a $\lazard_G(F)$-module \homo
$$\cob{}{G}{X} \to \Omega(X)$$
and it commutes with the four basic operations.
\end{cor}
\begin{flushright}
$\square$
\end{flushright}

Our second goal is to compare our theory $\cob{}{G}{-}$ to the \equi\ algebraic cobordism theory $\cob{\rm Tot}{G}{-}$ defined in \cite{homo equi alg cobor} using Totaro's approximation of $EG$. Following the basic assumptions in \cite{homo equi alg cobor}, $G$ will be a split torus of rank $n$ and $\char{k} = 0$ for the rest of this section.

Recall that $\cob{\rm Tot}{G}{X}$ is defined to be the inverse limit of $\Omega(X \x^G U_i)$ where $\{(W_i,U_i)\}$ is a good system of $G$-\repn s (in particular, the $G$-actions on $U_i$ are free) and $X \x^G U_i \defeq (X \x U_i) / G$ (see definition 1 in \cite{homo equi alg cobor} for details). When $G = (\torus)^n$, there is a simple choice of $\{(W_i, U_i)\}$. Let $W_i \defeq (\A^i)^n$ be a $G$-\repn\ with action given by 
$$(g_1, \ldots, g_n) \cdot (a_1, \ldots, a_n) \defeq (g_1 a_1, \ldots, g_n a_n)$$
and $U_i \defeq (\A^i - 0)^n$. Then $\{(W_i, U_i)\}$ forms a good system of $G$-\repn s. Also notice that $U_i/G \cong (\P^{i-1})^n$. For $1 \leq j \leq n$, let 
$$D_{ij} \defeq \P^{i-1} \x \cdots \x \P^{i-1} \x H \x \P^{i-1} \x \cdots \P^{i-1},$$
where $H \subset \P^{i-1}$ is a hyperplane in the $j$-th copy of $\P^{i-1}$, and we will consider $D_{ij}$ as a \sm\ divisor on $U_i/G$.

Let $\hat{\lazard}_G(F)$ be the completion of $\lazard_G(F)$ \wrt\ the ideal generated by the Euler classes and $\gamma_1, \ldots, \gamma_n$ be the natural set of generators of $G^* \cong \Z^n$. By Theorem 6.5 in \cite{equi FGL 2}, there is a ring isomorphism $\hat{\lazard}_G(F) \iso \lazard[[z_1, \ldots, z_n]]$ which sends $e(\gamma_j)$ to $z_j$. Therefore, we have a ring homo\morp
$$\phi : \lazard_G(F) \to \lazard[[z_1, \ldots, z_n]],$$ 
which sends $e(\gamma_j)$ to $z_j$. 

Now, for an object $X \in \gvar{G}$, we define an abelian group homo\morp 
$$\Psi_{\rm Tot} : \basicmod{G}{F}{}{X} \to \cob{\rm Tot}{G}{X}$$
by sending $a [f : Y \to X, \L_1, \ldots, \L_r]$ to
$$\sum_K a'_K\, c(\O(\hat{D}_{i1}))^{k_1} \cdots c(\O(\hat{D}_{in}))^{k_n} [\hat{f} : Y \x^G U_i \to X \x^G U_i, \hat{\L}_1, \ldots, \hat{\L}_r],$$
where $a'_K \in \lazard$ is given by the equation $\phi(a) = \sum_K a'_K z_1^{k_1} \cdots z_n^{k_n}$ with $K$ as multi-index, $\hat{D}_{ij}$ is the pull-back of $D_{ij}$ via $X \x^G U_i \to U_i/G$, $\hat{f} = f \x^G \id_{U_i}$ and $\hat{\L}_j$ are the sheaves naturally induced by $\L_j$. Notice that, by the \textbf{(Dim)} axiom in $\Omega(-)$, this gives a well-defined element in $\Omega(X \x^G U_i)$, for all $i$, and the case of infinite cycle is covered. Moreover, by its naturality, $\Psi_{\rm Tot}$ is also compatible with the inverse system. Hence, the map $\Psi_{\rm Tot}$ is well-defined.

\begin{prop}
\label{prop map to homo equi cobor theory}
Suppose $G$ is a split torus and $\char{k} = 0$. Then $\Psi_{\rm Tot}$ defines an abelian group homo\morp
$$\Psi_{\rm Tot} : \cob{}{G}{X} \to \cob{\rm Tot}{G}{X},$$
for any $X \in \gvar{G}$, and it commutes with the \proj\ push-forward, \sm\ pull-back and (first) Chern class operator.
\end{prop}

\begin{proof}
The only non-trivial part is the well-definedness of 
$$\Psi_{\rm Tot} : \cob{}{G}{X} \to \Omega(X \x^G U_i)$$
for all $i$. 

Recall that for a character $\alpha$, we call the element $e(\alpha) \defeq d(\alpha)^1_0 \in \lazard_G(F)$ the Euler class of $\alpha$. Let $I \subset \lazard_G(F)$ be the ideal generated by all the Euler classes and $\widehat{\lazard{\rm Z}}^{G,F}(X)$, $\cobcomp{}{G}{X}$ be the analogues of $\basicmod{G}{F}{}{X}$, $\cob{}{G}{X}$ respectively, defined by using coefficient ring $\hat{\lazard}_G(F)$ instead of $\lazard_G(F)$. Notice that there is an analogue 
$$\hat{\Psi}_{\rm Tot} : \widehat{\lazard{\rm Z}}^{G,F}(X) \to \Omega^G_{\rm Tot}(X)$$
of $\Psi_{\rm Tot}$ and the map $\Psi_{\rm Tot}$ factors through $\widehat{\lazard{\rm Z}}^{G,F}(X)$. Hence, it will be enough to show the well-definedness of 
$$\hat{\Psi}_{\rm Tot} : \cobcomp{}{G}{X} \to \Omega(X \x^G U_i).$$
Also, it is not hard to see that it respects the \textbf{(Sect)} axiom. So it remains to show that it respects the \textbf{(EFGL)} axiom.

Let $n$ be the rank of $G$. By Theorem 6.5 in \cite{equi FGL 2}, 
$$\hat{\lazard}_G(F)\{\{F\}\} \cong \lazard[[z_1, \ldots, z_n]]\{\{F\}\} \cong \lazard[[z_1, \ldots, z_n, \overline{y}]].$$
Following the notation in \cite{equi FGL} (see the proof of Lemma 13.1), for a character $\beta$, we let
$$y(\beta) = l_{\beta} y(V^1) = l_{\beta} y(\epsilon)$$
as an element in $\lazard_G(F)\{\{ F \}\}$ (or $\hat{\lazard}_G(F)\{\{F\}\}$).

\medskip

\noindent Claim 1 : For any character $\beta$, there are unique elements $g(\beta)_j \in \lazard[[z_1, \ldots, z_n]]$ such that 
\begin{eqnarray}
\label{eqn9}
y(\beta) = \sum_j g(\beta)_j \overline{y}^j.
\end{eqnarray}

By Theorem 6.5 in \cite{equi FGL 2},
\begin{eqnarray}
\label{eqn11}
\hat{\lazard}_G(F)\{\{F\}\} \cong \hat{\lazard}_G(F)[[\overline{y}]]
\end{eqnarray}
where $\overline{y}$ corresponds to $y(\epsilon)$. For any character $\beta$,
$$y(\beta) = e(\beta) + \overline{y}\, \sum_{i \geq 1} d(\beta)^1_i\, y(\alpha_2) \cdots y(\alpha_i).$$
For each $y(\alpha_j)$,  we can apply the above equation with $\beta = \alpha_j$. By repeating this argument, we got
$$y(\beta) = \sum_{i \geq 0} g(\beta)_i\, \overline{y}^i$$
for some elements $g(\beta)_i \in \hat{\lazard}_G(F)$ (Each $g(\beta)_i$ is given by a finite sum of elements in $\hat{\lazard}_G(F)$ because the flag $F$ is complete and $y(\epsilon) = \overline{y}$). These elements are unique by equation (\ref{eqn11}). \claimend

\medskip

\noindent Claim 2 : For all $n \geq 0$, $x \in \cob{}{G}{X}$, the element $c(\L)^m (x)$ lies inside $I^n \cob{}{G}{X}$ for sufficiently large $m$.

By Theorem \ref{thm gen by geo cycle}, we may assume $x = [f : Y \to X]$. If $\dim Y = 0$, then $f^* \L \cong \beta$ for some character $\beta$ and the statement is clearly true. The result then follows from Theorem \ref{thm gen by geo cycle} and the induction assumption on $\dim Y$. \claimend

\medskip

\noindent Claim 3 : For all \girred{G}\ $Y \in$ $\gsmcat{G}$, $\L \in \picard{G}{Y}$, the operator $\sum_j g(\beta)_j\, c(\L)^j$ is well-defined (in the theory $\cobcomp{}{G}{-}$) and 
\begin{eqnarray}
\label{eqn 37}
c(\L \otimes \beta) = \sum_j g(\beta)_j\, c(\L)^j.
\end{eqnarray}

For its well-definedness, since
$$\cobcomp{}{G}{X} = \lim_{\stackrel{\longleftarrow}{n}} \cob{}{G}{X}\, /\, I^n \cob{}{G}{X},$$
it suffices to show its well-definedness as a map $\cobcomp{}{G}{X} \to \cob{}{G}{X}\, /\, I^n \cob{}{G}{X}$ and it commutes with the maps in the inverse system. Then it follows from claim 2.

For the second part, in $\cobcomp{}{G}{-}$,
$$c(\L \otimes \beta) = e(\beta) + c(\L)\, \sum_{i \geq 1} d(\beta)^1_i\, c(\L \otimes \alpha_2) \cdots c(\L \otimes \alpha_i).$$
Then the result follows from a similar argument as in claim 1. \claimend

\medskip

Observe that $y(V^i) = y(\alpha_1) \cdots y(\alpha_i)$, which can be expressed in terms of $\overline{y}$ by equation (\ref{eqn9}). Therefore, the equation
\begin{eqnarray}
\label{eqn8}
y(V^i) y(V^j) = \sum_s b^{i,j}_s y(V^s) 
\end{eqnarray}
can be expressed in terms of $\overline{y}$ (the right hand side in terms of $\overline{y}$ is well-defined because the flag $F$ is complete). On the other hand, $V^i(\L) = c(\L \otimes \alpha_1) \cdots c(\L \otimes \alpha_i)$, which can also be expressed in terms of $c(\L)$ by equation (\ref{eqn 37}). Hence, by equation (\ref{eqn8}), we have
$$V^i(\L) V^j(\L)[\id_Y] = \sum_s b^{i,j}_s V^s(\L)[\id_Y]$$
(well-definedness while expressed in terms of $c(\L)$ follows from claim 2). The rest of the \textbf{(EFGL)} axiom follows from similar arguments.
\end{proof}

\begin{cor}
\label{cor lazard iso theory over a pt}
Suppose $G$ is a split torus and $\char{k} = 0$. If the completion map $\lazard_G(F) \to \hat{\lazard}_G(F)$ is injective, then the canonical ring homo\morp\
$$\lazard_G(F) \to \cob{}{G}{\pt}$$
is an iso\morp.
\end{cor}

\begin{proof}
The surjectivity is given by Theorem \ref{thm gen by lazard}. For the injectivity, let $n$ be the rank of $G$ and $f : \lazard_G(F) \to \cob{}{G}{\pt}$ be the canonical map. As pointed out in section 3.3 in \cite{homo equi alg cobor}, 
$$\cob{\rm Tot}{G}{\pt} \cong \lazard[[z_1, \ldots, z_n]]$$ 
with $z_j$ corresponding to the element $[D_{ij} \embed U_i/G] \in \Omega(U_i/G)$.  Then, by Proposition \ref{prop map to homo equi cobor theory}, we have a composition of maps
$$\lazard_G(F) \stackrel{f}{\longto} \cob{}{G}{\pt} \stackrel{\Psi_{\rm Tot}}{\longto} \cob{\rm Tot}{G}{\pt} \cong \lazard[[z_1, \ldots, z_n]],$$
which is nothing but the completion map $\lazard_G(F) \to \hat{\lazard}_G(F)$. Hence $f$ is injective.
\end{proof}

\begin{rmk}
{\rm
Notice that there is an analogue of Corollary \ref{cor lazard iso theory over a pt} in Topology. In his paper \cite{equi FGL 2}, Greenlees conjectured that, for any compact abelian Lie group $G$, the canonical ring \homo
$$\nu : \lazard_G(F) \to MU_G,$$
where $MU_G$ is Tom Dieck's equivariant cobordism ring, is an isomorphism. It is shown in \cite{equi FGL 2} that $\nu$ is surjective (Theorem 13.1) and the completion of $\nu$ \wrt\ the ideal generated by Euler classes is an isomorphism (Proposition 13.3). Hence the analogue of Corollary \ref{cor lazard iso theory over a pt} is also true in Topology, i.e., if $G$ is a torus and the completion map $\lazard_G(F) \to \hat{\lazard}_G(F)$ is injective, then $\nu$ is an isomorphism.
}
\end{rmk}

\bigskip
\bigskip

\section{More on $\cob{\rm Tot}{G}{-}$}
\label{sect more on Totaro cobor}

In Heller and Malag\'{o}n-L\'{o}pez's paper \cite{homo equi alg cobor}, they define an equivariant algebraic cobordism theory $\cob{\rm Tot}{G}{-}$ for connected linear algebraic group $G$ over a field of characteristic zero. In order to have a full comparison with our theory, we need to first extend their definition to allow $G$ to be a split diagonalizable group (not connected). Since Corollary 3 in \cite{homo equi alg cobor} still holds, $\cob{\rm Tot}{G}{-}$ is well-defined for such $G$ and its definition is still independent of the choice of good system of \repn s. In this section, we will compute $\cob{\rm Tot}{G}{\pt}$ and generalize Proposition \ref{prop map to homo equi cobor theory}\ and Corollary \ref{cor lazard iso theory over a pt}. As in \cite{homo equi alg cobor}, we will assume $\char{k} = 0$.

Let us first compute $\cob{\rm Tot}{G}{\pt}$ when $G$ is a cyclic group of order $n$. Let $\alpha$ be a generator of $G^*$. For all $i > 0$, let $W_i$ be $\A^i$ with $G$-action given by $g \cdot w = \alpha(g) w$ and $U_i \defeq W_i - 0$. Then $\{(W_i, V_i)\}$ forms a good system of \repn s. Let $E_i$ be the \equi\ line bundle over $\P(W_i)$ corresponding to the sheaf $\O_{\P^i}(n)$ and $s : \P(W_i) \to E_i$ be the zero section. 

Now, $W_i = \spec{k[x_1, \ldots, x_i]}$ and $U_i \supset D(x_j) = \spec{k[x_1, \ldots, x_i][x_j^{-1}]}$. On the other hand, $\P(W_i)|_{D(x_j)} = \spec{k[\frac{x_1}{x_j}, \ldots, \frac{x_i}{x_j}]}$ and 
$$E_i|_{D(x_j)} = \spec{k[\frac{x_1}{x_j}, \ldots, \frac{x_i}{x_j}][v_j]}$$ 
where $v_j$ corresponds to the section $x_j^{-n}$. Notice that the $G$-actions on $\P(W_i)$ and $E_i$ are both trivial. Define an \equi\ map
$$q_j : k[\frac{x_1}{x_j}, \ldots, \frac{x_i}{x_j}][v_j][v_j^{-1}] \to k[x_1, \ldots, x_i][x_j^{-1}]$$
by sending $\frac{x_k}{x_j}$ to $\frac{x_k}{x_j}$ and $v_j$ to $x_j^{-n}$ and let
$$p_j : D(x_j) \to (E_i - s(\P(W_i)))|_{D(x_j)}$$
be the corresponding \equi\ map. 

\begin{lemma}
The \equi\ maps $p_j$ patch together to define an \equi\ map
$$p : U_i \to E_i - s(\P(W_i))$$
and it is isomorphic to the quotient map $U_i \to U_i/G$.
\end{lemma}

\begin{proof}
The patching follows from the naturality of the definition of $p_j$. Also, it is not hard to see that $q_j$ is injective and its image is precisely $(k[x_1, \ldots, x_i][x_j^{-1}])^G$. Hence, $p_j$ is isomorphic to the quotient map and the result then follows.
\end{proof}

By the projective bundle formula, 
$$\Omega(\P(W_i)) = \oplus_{k=0}^{i-1}\, \lazard \cdot c(\O_{\P(W_i)}(1))^k [\id_{\P(W_i)}].$$
By the localization property, the sequence
$$\Omega(\P(W_i)) \stackrel{s_*}{\longto} \Omega(E_i) \longto \Omega(U_i/G) \longto 0$$
is exact. By the extended homotopy property, 
$$\Omega(E_i) = \oplus_{k=0}^{i-1}\, \lazard \cdot c(\L_i)^k [\id_{E_i}]$$
where $\pi : E_i \to \P(W_i)$ is the projection and $\L_i \defeq \pi^* \O_{\P(W_i)}(1)$. Notice that
$$s_*[\id_{\P(W_i)}] = c(\pi^* O_{\P(W_i)}(n)) [\id_{E_i}] = F^n(c(\L_i))[\id_{E_i}],$$
where $F^n(u) \in \lazard[[u]]$ is defined inductively by $F^0(u) \defeq 0$ and $F^{n+1}(u) \defeq F(F^n(u),u)$. Therefore,
$$\Omega(U_i/G) \cong (\oplus_{k=0}^{i-1}\, \lazard \cdot c(\L_i)^k [\id_{E_i}])\ /\ (c(\L_i)^l \circ F^n(c(\L_i))[\id_{E_i}] \ |\ 0 \leq l \leq i-1)$$
as $\lazard$-modules. Moreover, these isomorphisms commute with the maps in the inverse system. Hence, 
$$\cob{\rm Tot}{G}{\pt} = \lim_{\stackrel{\leftarrow}{i}} \Omega(U_i/G) \cong \lazard [[t]] / (F^n(t))$$
where $t^j$ is identified with $c(\L_i)^j [\id_{E_i}]$.

In general, suppose 
$$G \cong G_f \x G_t = (\prod_{j=1}^s G_j) \x G_t$$ 
where $G_j$ is a cyclic group of order $n_j$ and $G_t$ is a split torus of rank $r$. Let $\beta_j$ be a generator of $G_j^*$, $\gamma_1, \ldots, \gamma_r$ be the standard set of generators of $G_t^*$. Also, let $W(\beta_j)_i$, $W(\gamma_k)_i$ be $\A^i$ with $G$-actions given by $\beta_j$, $\gamma_k$ respectively, $U(\beta_j)_i \defeq W(\beta_j)_i - 0$, $U(\gamma_k)_i \defeq W(\gamma_k)_i - 0$, $W_i \defeq \prod_{j=1}^s W(\beta_j)_i \x \prod_{k=1}^r W(\gamma_k)_i$ and $U_i \defeq \prod_{j=1}^s U(\beta_j)_i \x \prod_{k=1}^r U(\gamma_k)_i$. Then $\{(W_i, U_i)\}$ forms a good system of \repn s. Similarly, we have 
$$U_i/ G = (\prod_{j=1}^s U(\beta_j)_i / G_j) \x (\prod_{k=1}^r U(\gamma_k)_i / \torus) \cong (\prod_{j=1}^r (E(\beta_j)_i - s(\P(W^j_i)))) \x \prod_{k=1}^r \P(W(\gamma_k)_i)$$
where $E(\beta_j)_i$ is the line bundle over $\P(W(\beta_j)_i)$ corresponding to the sheaf $\O_{\P(W(\beta_j)_i)}(n_j)$ and $s : \P(W(\beta_j)_i) \to E(\beta_j)_i$ is the zero section. As before, let 
$$\pi : F_i \defeq (\prod_{j=1}^r E(\beta_j)_i) \x \prod_{k=1}^r \P(W(\gamma_k)_i) \to \prod_{j=1}^r \P(W(\beta_j)_i),$$ 
$$\pi' : F_i  \to \prod_{k=1}^r \P(W(\gamma_k)_i)$$ 
be the projections, $\L(\beta_j)_i \defeq \pi^* \circ \pi_j^* \O_{\P(W(\beta_j)_i)}(1)$ and $\L(\gamma_k)_i \defeq {\pi'}^* \circ \pi_k^* \O_{\P(W(\gamma_k)_i)}(1)$. By similar calculations, we have 
\begin{eqnarray}
\label{eqn Totaro ring}
\cob{\rm Tot}{G}{\pt} \cong \lazard [[t_1, \ldots, t_s, z_1, \ldots, z_r]] / (F^{n_1}(t_1), \ldots, F^{n_s}(t_s))
\end{eqnarray}
\noindent where $t_j^p$ is identified with $c(\L(\beta_j)_i)^p [\id_{F_i}]$ and $z_k^q$ is identified with $c(\L(\gamma_k)_i)^q [\id_{F_i}]$.

Now, as in section \ref{sect comparison with other theories}, let 
$$\phi : \lazard_G(F) \to \hat{\lazard}_G(F) \cong \lazard [[t_1, \ldots, t_s, z_1, \ldots, z_r]] / (F^{n_1}(t_1), \ldots, F^{n_s}(t_s))$$
which sends $e(\beta_j)$, $e(\gamma_k)$ to $t_j$, $z_k$ respectively. For an object $X \in \gvar{G}$, we define an abelian group homo\morp 
$$\Psi_{\rm Tot} : \basicmod{G}{F}{}{X} \to \cob{\rm Tot}{G}{X}$$
by sending $a [f : Y \to X, \cat{M}_1, \ldots]$ to
$$\sum_{PQ} a'_{PQ}\, \prod_{j=1}^s c(\O(\widehat{\L(\beta_j)}_i))^{p_j} \circ \prod_{k=1}^r  c(\O(\widehat{\L(\gamma_k)}_i))^{q_k} [\hat{f} : Y \x^G U_i \to X \x^G U_i, \hat{\cat{M}}_1, \ldots],$$
where $a'_{PQ} \in \lazard$ is given by the equation $\phi(a) = \sum_{PQ} a'_{PQ} t_1^{p_1} \cdots t_s^{p_s} z_1^{q_1} \cdots z_r^{q_r}$ with $P, Q$ as multi-indices, $\widehat{\L(\beta_j)}_i$, $\widehat{\L(\gamma_k)}_i$ are the pull-backs of $\L(\beta_j)_i$, $\L(\gamma_k)_i$ respectively, via $X \x^G U_i \to U_i/G$, $\hat{f} = f \x^G \id_{U_i}$ and $\hat{\cat{M}}_l$ are the sheaves naturally induced by $\cat{M}_l$. For the same reason as in section \ref{sect comparison with other theories}, the map $\Psi_{\rm Tot}$ is well-defined.

\begin{prop}
\label{prop map to homo equi cobor theory 2}
Suppose $\char{k} = 0$. Then $\Psi_{\rm Tot}$ defines an abelian group homo\morp
$$\Psi_{\rm Tot} : \cob{}{G}{X} \to \cob{\rm Tot}{G}{X},$$
for any $X \in$ $\gvar{G}$, and it commutes with the \proj\ push-forward, \sm\ pull-back and (first) Chern class operator.
\end{prop}

\begin{proof}
See the proof of Proposition \ref{prop map to homo equi cobor theory}.
\end{proof}

\begin{cor}
\label{cor lazard iso theory over a pt 2}
Suppose $\char{k} = 0$. If the completion map $\lazard_G(F) \to \hat{\lazard}_G(F)$ is injective, then the canonical ring homo\morp\
$$\lazard_G(F) \to \cob{}{G}{\pt}$$
is an iso\morp.
\end{cor}

\begin{proof}
See the proof of Corollary \ref{cor lazard iso theory over a pt}.
\end{proof}

\bigskip

\bigskip

\section{Comparison with the \equi\ K-theory}
\label{sect K theory}

In this section, we will compare our \equi\ algebraic cobordism theory to the \equi\ K-theory. Recall the following definition of \equi\ K-theory from \cite{equi K theory}. Suppose $G$ is a group scheme over $k$ and $X$ is in $\gvar{G}$. Denote the abelian category of $G$-\equi, coherent sheaves over $X$ by $M(G;X)$. Then define 
$$K'_n(G;X) \defeq K_n(M(G;X)).$$
Also, denote the abelian category of $G$-\equi, locally free coherent sheaves over $X$ by $P(G;X)$ and define 
$$K_n(G;X) \defeq K_n(P(G;X)).$$
We then have the following list of basic results (see section 2 in \cite{equi K theory}) :

\noindent \begin{statementslist}
{\rm (1)} & $K'_n(G;X)$ has flat pull-back and \proj\ push-forward. \nonumber\\
{\rm (2)} & If $G$ is the trivial group, then $K'_n(G;X) = K'_n(X)$ (the ordinary $K$-theory). \nonumber\\
{\rm (3)} & There is a natural iso\morp\ $R(G) \iso K'_0(G;\pt)$. \nonumber\\
{\rm (4)} & If $X$ is \sm\ and \qproj\ over $k$, then the natural \homo\ $K_n(G;X) \to K'_n(G;X)$ is an iso\morp\ (Proposition 2.20 in \cite{equi K theory}). \nonumber
\end{statementslist}

\begin{rmk}
{\rm
In this paper, we only focus on the $K_0'(G;-)$ and $K_0(G;-)$ theories. In order to have \proj\ push-forward, we need to consider $K_0'(G;-)$. But for external product, we need the ring structure on $K_0(G;-)$. Hence, we will focus on the category $\gsmcat{G}$.
}
\end{rmk}

For an object $X \in$ $\gsmcat{G}$\ and $\L \in \picard{G}{X}$, define 
$$c_K(\L) \defeq ([\O_X] - [\dual{\L}]) v^{-1}$$
and $V^i_{K,S}(\L)$ analogously as elements in $K_0(G;X)[v,v^{-1}]$. Then $K_0(G;-)[v,v^{-1}]$ is a theory on $\gsmcat{G}$\ with four basic operations, i.e.,
\begin{eqnarray}
f_![\cat{E}] &\defeq& v^{\dim f}\, \sum_i\, (-1)^i [R^i f_* \cat{E}] \nonumber\\
&& \text{(when $f$ is \proj\ and equidimensional),} \nonumber\\
f^*[\cat{E}] &\defeq& [f^* \cat{E}] \nonumber\\
&& \text{(when $f$ is flat),} \nonumber\\
c_K(\L)[\cat{E}] &\defeq& c_K(\L) \cdot [\cat{E}], \nonumber\\
\ [\cat{E_1}] \x [\cat{E_2}] &\defeq &[\gp_1^* \cat{E_1}] \cdot [\gp_2^* \cat{E_2}] \nonumber
\end{eqnarray}
where $\cat{E_1}$, $\cat{E_2}$ are \glin{G}\ locally free coherent sheaves over $X_1$, $X_2$ respectively.

Also recall the notion of ``multiplicative \fgl'' in section 7 of \cite{equi FGL 2} (also see \cite{multi FGL}). Suppose $G$ is an abelian compact Lie group and $F$ is a complete $G$-flag. A $(G,F)$-\equi\ \fgl\ (over a ring $R$) is multiplicative if its coproduct has the property
$$\gD y(\gep) = y(\gep) \otimes 1 + 1 \otimes y(\gep) - v\, y(\gep) \otimes y(\gep)$$
for some element $v \in R$. In other words, $f^1_{1,1} = -v$ and $f^1_{s,t} = 0$ if $s$ or $t > 1$. Denote its representing ring by $\lazard^m_G(F)$. Then we have a natural surjective map $\lazard_G(F) \to \lazard^m_G(F)$ (by sending $f^1_{1,1}$ to $-v$ and $f^1_{s,t}$ to zero if $s$ or $t > 1$). In addition, by Proposition 4.5 in \cite{multi FGL}, there is a natural iso\morp
$$\lazard^m_G(F)[v^{-1}] \iso R(G)[v,v^{-1}]$$
which sends Euler classes $d(\ga)^1_0 = e(\ga)$ to Euler classes $e_K(\ga) \defeq (1 - \dual{\ga})v^{-1}$ (by Remarks \ref{rmk FGL}, $\lazard^m_G(F)$ is generated by $v$ and the Euler classes, so the map is uniquely determined). Hence, there is a natural ring \homo\
$$\gPH_K : \lazard_G(F) \to \lazard^m_G(F) \to R(G)[v,v^{-1}].$$

\begin{rmk}
{\rm
Readers should be aware that the definitions of the Euler classes in the equivariant K-theory in our paper and \cite{multi FGL} (also \cite{equi FGL} and \cite{equi FGL 2}) are different. In our paper,
$$e_K(\alpha) = (1 - \dual{\alpha})v^{-1},$$
while in \cite{multi FGL},
$$\tilde{e}_K(\alpha) =  (1 - \alpha)v^{-1}$$
(see remark 3.4). That means the isomorphism 
$$\lazard^m_G(F)[v^{-1}] \iso R(G)[v,v^{-1}]$$
(Proposition 4.5 in \cite{multi FGL}) in fact sends $e(\alpha) = d(\alpha)^1_0$ to $\tilde{e}_K(\alpha) = (1- \alpha)v^{-1}$. But it can be easily remedied by composing with the isomorphism $R(G)[v,v^{-1}] \iso R(G)[v,v^{-1}]$ which sends $\alpha$ to $\dual{\alpha}$. Since we would like to have $e_K(\alpha) = c_K(\alpha)$ and the axiom \textbf{(Sect)} in the equivariant K-theory, this disagreement is inevitable (see the proof of Lemma \ref{lemma Sect in K theory}). 
}
\end{rmk}

Our first objective in this section is to define a canonical map from our \equi\ algebraic cobordism theory $\cob{}{G}{-}$ to the \equi\ K-theory $K_0(G;-)[v,v^{-1}]$. For the rest of this section, we will use the same assumptions on $G$ and $k$ as in section \ref{sect notation} (so that our theory $\cob{}{G}{-}$ is well-defined). 

We will start with showing some basic results in $K_0(G;-)[v,v^{-1}]$. As in $\cob{}{G}{-}$, let $\Endofin{K_0(G;-)[v,v^{-1}]}$ be the $R(G)[v,v^{-1}]$-subalgebra of $\Endo{K_0(G;-)[v,v^{-1}]}$ generated by $c_K(\L)$.

\begin{lemma}
\label{lemma basic axiom hold in K theory}
Axioms \textbf{(A1)-(A8)} hold in the \equi\ K-theory.
\end{lemma}

\begin{proof}
Follows from the definitions and some basic facts about $G$-\equi\ sheaves.
\end{proof}

\begin{lemma}
\label{lemma FGL in K theory}
For any $X \in$ $\gsmcat{G}$, $\L$, $\cat{M} \in \picard{G}{X}$ and character $\alpha$, 

\noindent \begin{statementslist}
{\rm (1)} & $c_K(\O_X) = 0$ \\
{\rm (2)} & $c_K(\alpha) = e_K(\alpha)$ \\
{\rm (3)} & $c_K(\L \otimes \cat{M}) = c_K(\L) + c_K(\cat{M}) - v\, c_K(\L) c_K(\cat{M})$ \\
{\rm (4)} & $c_K(\dual{\L}) = \gs_K \cdot c_K(\L)$ for some $\gs_K \in \Endofin{K(G;X)[v,v^{-1}]}$.
\end{statementslist}
\end{lemma}

\begin{proof}
Part (1), (2), (3) follow directly from the definition. Part (4) is an analogue of Proposition \ref{prop FGL inverse}.
\end{proof}

We certainly need the \textbf{(Sect)} axiom to hold in the \equi\ K-theory.

\begin{lemma}
\label{lemma Sect in K theory}
Suppose $X \in$ $\gsmcat{G}$\ is \girred{G}\ and $\L$ is a sheaf in $\picard{G}{X}$. Moreover, if there exists an invariant section $s \in \gsection{X}{\L}^G$ that cuts out an invariant \sm\ divisor $Z$ on $X$, then the following equality holds in $K_0(G;X)[v,v^{-1}]$ {\rm :}
$$c_K(\L) = i_![\O_Z]$$
where $i : Z \embed X$ is the immersion.
\end{lemma}

\begin{proof}
It follows from the facts that $\L \cong \O_X(Z)$, the functor $i_*$ is exact and the following sequence is exact :
$$0 \to \O_X(-Z) \to \O_X \to i_* \O_Z \to 0.$$
\end{proof}

Next, we will show that the double point relation holds in $K_0(G;-)[v,v^{-1}]$ by following the same recipe as in section \ref{sect basic properties}.

\begin{lemma}
\label{lemma DPR residue in K theory}
Suppose $Y$ is an object in $\gsmcat{G}$, $E_1$, $E_2$ are two invariant \sm\ divisors on $Y$ with transverse intersection $D$. Then, as elements in $K_0(G;D)[v,v^{-1}]$,
$$-v = - p_! [\O_{\P_D}]$$
where $p : \P_D \defeq \P(\O_D \oplus \O_D(E_1)) \to D$.
\end{lemma}

\begin{proof}
This result is an analogue of Lemma \ref{lemma FGL remaining terms}. First of all, the surjective \morp\ $\O_D \oplus \O_D(E_1) \to \O_D(E_1)$ of sheaves over $D$ defines an \inv\ section $D \embed \P_D$. Let $Y'$ be the blow up of $\P_D \x \P^1$ (trivial action on $\P^1$) along $D \x 0$. Then we have a map $Y' \to \P_D \x \P^1 \to \P^1.$ By Lemma \ref{lemma Sect in K theory}, we have
\begin{eqnarray}
{i_{\infty}}_![\O_{Y'_{\infty}}] &=& c_K(\O_{Y'}(Y'_{\infty})) \nonumber\\
&& \text{where $i_{\infty} : Y'_{\infty} \embed Y'$} \nonumber\\
&=& c_K(\O_{Y'}(Y'_0)) \nonumber\\
&=& c_K(\O_{Y'}(\P_D + E)) \nonumber\\
&& \text{where $E$ is the exceptional divisor} \nonumber\\
&=& c_K(\O_{Y'}(\P_D)) + c_K(\O_{Y'}(E)) - v\, c_K(\O_{Y'}(\P_D)) c_K(\O_{Y'}(E)), \nonumber
\end{eqnarray}
which is, by Lemma \ref{lemma Sect in K theory}, equal to 
$${i_{\P_D}}_![\O_{\P_D}] + {i_E}_![\O_E] - v\, c_K(\O_{Y'}(\P_D)) \cdot {i_E}_![\O_E]$$
where $i_{\P_D}$ and $i_E$ are the corresponding immersions. Moreover, 
$$c_K(\O_{Y'}(\P_D)) \cdot {i_E}_![\O_E] = {i_E}_! (c_K(\O_E(\P_D))) = {i_E}_! \circ {i_{\P_D \cap E}}_![\O_{\P_D \cap E}].$$
So, we have
$${i_{\infty}}_![\O_{Y'_{\infty}}] = {i_{\P_D}}_![\O_{P_D}] + {i_E}_![\O_E] - v\, {i_E}_! \circ {i_{\P_D \cap E}}_![\O_{\P_D \cap E}].$$
Notice that $Y'_{\infty} \cong E \cong \P_D$ and $\P_D \cap E \cong D$. By pushing the above equality down to $K_0(G;D)[v,v^{-1}]$, we have
$$p_![\O_{P_D}] = p_![\O_{P_D}] + p_![\O_{P_D}] - v$$
and we are done.
\end{proof}

\begin{lemma}
\label{lemma DPR in K theory}
Suppose $A$, $B$, $C$ are invariant \sm\ divisors on $Y \in$ $\gsmcat{G}$\ such that $A + B \sim C$, $C$ is disjoint from $A \cup B$ and $A + B + C$ is a \rsncd. Then, as elements in $K_0(G;Y)[v,v^{-1}]$,
$${i_C}_![\O_C] = {i_A}_![\O_A] + {i_B}_![\O_B] - {i_D}_! \circ p_![\O_{\P_D}]$$
where $D \defeq A \cap B$, $i_A$, $i_B$, $i_C$ and $i_D$ are the corresponding immersions and $p : \P_D \defeq \P(\O_D \oplus \O_D(A)) \to D$.
\end{lemma}

\begin{proof}
By Lemma \ref{lemma Sect in K theory}, we have
\begin{eqnarray}
{i_C}_![\O_C] &=& c_K(\O_Y(C)) \nonumber\\
&=& c_K(\O_Y(A + B)) \nonumber\\
&=& c_K(\O_Y(A)) + c_K(\O_Y(B)) - v\,c_K(\O_Y(A))c_K(\O_Y(B)) \nonumber\\
&=& {i_A}_![\O_A] + {i_B}_![\O_B] - v\, c_K(\O_Y(A)) \cdot {i_B}_![\O_B] \nonumber\\
&=& {i_A}_![\O_A] + {i_B}_![\O_B] - v\, {i_B}_! \circ j_![\O_D] \nonumber
\end{eqnarray}
where $j : D \embed B$. The result then follows from Lemma \ref{lemma DPR residue in K theory}.
\end{proof}

Hence, the double point relation holds in $K_0(G;-)[v,v^{-1}]$, so does the blow up relation and the extended double point relation. We also need an analogue of Proposition \ref{prop Nilp axiom}.

\begin{lemma}
\label{lemma Nilp axiom in K theory}
Suppose $\char{k} = 0$. For any \girred{G}\ $Y \in$ $\gsmcat{G}$, $\L \in \picard{G}{Y}$ and $S$ as in (\ref{eqn basic element}),
$$V^n_{K,S}(\L) = 0,$$
as elements in $K_0(G;Y)[v,v^{-1}]$, for sufficiently large $n$.
\end{lemma}

\begin{proof}
See the proof of Proposition \ref{prop Nilp axiom}. 
\end{proof}

Now, suppose $\char{k} = 0$. For any $X \in$ $\gsmcat{G}$, since $K_0(G;X)$ can be considered as a $K_0(G;\pt)$-module and $K_0(G;\pt) \cong R(G)$, we have a natural map 
$$\gPS_K : R(G)[v,v^{-1}] \otimes_{\lazard_G(F)} \basicmod{G}{F}{}{X} \to K_0(G;X)[v,v^{-1}]$$
defined by sending $b \otimes \sum_I a_I [f : Y \to X, V^{i_1}_{S_1}(\L_1), \ldots, V^{i_r}_{S_r}(\L_r)]$ to 
$$\sum_I b\, \Phi_K(a_I)\, f_! (V^{i_1}_{K,S_1}(\L_1) \cdots V^{i_r}_{K,S_r}(\L_r)),$$
which is actually a finite sum by Lemma \ref{lemma Nilp axiom in K theory}.
 
Our main objective in this section is to prove the \equi\ Conner-Floyd isomorphism, i.e., $\Psi_K$ defines a map
$$R(G)[v,v^{-1}] \otimes_{\lazard_G(F)} \cob{}{G}{X} \to K_0(G;X)[v,v^{-1}]$$
and it is an isomorphism. We will handle the well-definedness and surjectivity of $\Psi_K$ first.

\begin{prop}
\label{prop map to K theory}
Suppose $\char{k} = 0$. Then $\gPS_K$ defines a canonical, surjective, $R(G)[v,v^{-1}]$-module homo\morp 
$$\gPS_K : R(G)[v,v^{-1}] \otimes_{\lazard_G(F)} \cob{}{G}{X} \to K_0(G;X)[v,v^{-1}],$$
for any $X \in$ $\gsmcat{G}$, and it commutes with the four basic operations.
\end{prop}

\begin{proof}
First of all, it is not hard to see that $\gPS_K$ commutes with the four basic operations. By Lemma \ref{lemma Sect in K theory}, the map $\gPS_K$ respects the \textbf{(Sect)} axiom. For the \textbf{(EFGL)} axiom, we need to show that, as elements in $K_0(G;Y)[v,v^{-1}]$ where $Y \in$ $\gsmcat{G}$\ is \girred{G},
\begin{eqnarray}
V^i_K(\L) \cdot V^j_K(\L) &=& \sum_{s \geq 0}\,b^{i,j}_s\,V^s_K(\L), \nonumber\\
V^i_K(\L \otimes \ga) &=& \sum_{s \geq 0}\,d(\ga)^i_s\,V^s_K(\L), \nonumber\\
V^i_K(\L \otimes \cat{M}) &=& \sum_{s,t \geq 0}\,f^i_{s,t}\,V^s_K(\L) \cdot V^t_K(\cat{M}), \nonumber
\end{eqnarray}
where $b^{i,j}_s$, $d(\ga)^i_s$, $f^i_{s,t}$ are considered to be elements in $R(G)[v,v^{-1}]$ via $\gPH_K$.

Notice that
\begin{eqnarray}
\label{eqn 26}
&& V^i_K(\L) \cdot V^j_K(\L) \\
&=& V^{i-1}_K(\L) \cdot c_K(\L \otimes \ga_i) \cdot V^j_K(\L) \nonumber\\
&=& V^{i-1}_K(\L) \cdot c_K(\L \otimes \ga_{j+1} \otimes \gb) \cdot V^j_K(\L) \nonumber\\
&& \text{where $\gb \defeq \dual{\ga_{j+1}} \otimes \ga_i$} \nonumber\\
&=& V^{i-1}_K(\L) \cdot (c_K(\L \otimes \ga_{j+1}) + e_K(\gb) - v\, e_K(\gb) c_K(\L \otimes \ga_{j+1})) \cdot V^j_K(\L) \nonumber\\
&=&  e_K(\gb) \cdot V^{i-1}_K(\L) \cdot V^j_K(\L) + (1 - v\, e_K(\gb)) \cdot V^{i-1}_K(\L) \cdot V^{j+1}_K(\L). \nonumber
\end{eqnarray}
So, inductively, we can write 
$$V^i_K(\L) \cdot V^j_K(\L) = \sum_{s \geq 0} B^{i,j}_s V^s_K(\L)$$
for some elements $B^{i,j}_s \in R(G)[v,v^{-1}]$ uniquely determined by the above process. Thus, it is enough to show $B^{i,j}_s = b^{i.j}_s$.

Consider the multiplicative \fgl\ $\lazard^m_G(F)[v^{-1}]\{\{F\}\}.$ By Lemma 2.1 in \cite{multi FGL}, 
$$l_{\ga}\, y(\gep) = y(\ga) = e(\ga) + (1 - v e(\ga))y(\gep).$$
In other words, $d(\ga)^1_0 = e(\ga)$, $d(\ga)^1_1 = 1 - v e(\ga)$ and $d(\ga)^1_s = 0$ when $s > 1$. Also, if we apply $l_{\gb}$ on the above equation, we have
$$l_{\gb}\, y(\ga) = e(\ga) + (1 - v e(\ga)) l_{\gb}\, y(\gep) =  e(\ga) + (1 - v e(\ga)) y(\gb).$$
Therefore,
\begin{equation}
\label{eqn3}
l_{\ga}\, y(\gb) = y(\ga \otimes \gb) = l_{\gb}\, y(\ga) =e(\ga) + (1 - v e(\ga)) y(\gb).
\end{equation}

Now, we apply an analogue of the procedure in (\ref{eqn 26}) to the element $y(V^i) \cdot y(V^j) \in \lazard^m_G(F)[v^{-1}]\{\{F\}\}$ :
\begin{eqnarray}
y(V^i) \cdot y(V^j) &=& y(V^{i-1}) \cdot y(\ga_i) \cdot y(V^j) \nonumber\\
&=& y(V^{i-1}) \cdot y(\gb \otimes \ga_{j+1}) \cdot y(V^j) \nonumber\\
&& \text{where $\gb \defeq \dual{\ga_{j+1}} \otimes \ga_i$} \nonumber\\
&=& y(V^{i-1}) \cdot l_{\gb}\, y(\ga_{j+1}) \cdot y(V^j) \nonumber\\
&=& y(V^{i-1}) \cdot (e(\gb) + (1 - v e(\gb)) y(\ga_{j+1})) \cdot y(V^j) \nonumber\\
&& \text{by equation (\ref{eqn3})} \nonumber\\
&=&  e(\gb) \cdot y(V^{i-1}) \cdot y(V^j) + (1 - v e(\gb)) \cdot y(V^{i-1}) \cdot y(V^{j+1}). \nonumber
\end{eqnarray}
Again, inductively, we can write 
$$y(V^i) \cdot y(V^j) = \sum_{s \geq 0} \overline{B}^{i,j}_s y(V^s)$$
for some elements $\overline{B}^{i,j}_s \in \lazard^m_G(F)[v^{-1}]$. Since $e(\gb)$ is identified with $e_K(\gb)$, we have $B^{i,j}_s = \overline{B}^{i,j}_s$. Also, 
$$y(V^i) \cdot y(V^j) = \sum_{s \geq 0} b^{i,j}_s y(V^s)$$
by definition. Since $\{y(V^s)\}$ is a basis, $b^{i,j}_s = \overline{B}^{i,j}_s = B^{i,j}_s$.

Similarly, we have
\begin{eqnarray}
V^i_K(\L \otimes \ga) &=& V^{i-1}_K(\L \otimes \ga) \cdot c_K(\L \otimes \gb) \nonumber\\
&& \text{where $\gb \defeq \ga \otimes \ga_i$} \nonumber\\
&=& e_K(\gb) \cdot V^{i-1}_K(\L \otimes \ga) + (1 - v e_K(\gb)) \cdot V^{i-1}_K(\L \otimes \ga) \cdot c_K(\L). \nonumber
\end{eqnarray}
By induction, $V^{i-1}_K(\L \otimes \ga)$ can be expressed by $\{V^s_K(\L)\}$ with uniquely determined coefficients and 
$$V^s_K(\L) \cdot c_K(\L) = V^s_K(\L) \cdot V^1_K(\L) = \sum_t b^{s,1}_t V^t_K(\L).$$ 
Hence,
$$V^i_K(\L \otimes \ga) = \sum_{s \geq 0}\,D(\ga)^i_s\,V^s_K(\L)$$
for some uniquely determined elements $D(\ga)^i_s \in R(G)[v,v^{-1}]$. On the other hand, in $\lazard^m_G(F)[v^{-1}]\{\{F\}\}$, we have
\begin{eqnarray}
l_{\ga}\, y(V^i) &=& l_{\ga}\, y(V^{i-1}) \cdot l_{\ga}\, y(\ga_i) \nonumber\\
&=& l_{\ga}\, y(V^{i-1}) \cdot l_{\gb}\, y(\gep) \nonumber\\
&& \text{where $\gb \defeq \ga \otimes \ga_i$} \nonumber\\
&=& e(\gb) \cdot l_{\ga}\, y(V^{i-1}) + (1 - v e(\gb)) \cdot l_{\ga}\, y(V^{i-1}) \cdot y(V^1). \nonumber
\end{eqnarray}
Hence, by the same reason, $D(\ga)^i_s = d(\ga)^i_s$.

Again, 
\begin{eqnarray}
V^i_K(\L \otimes \cat{M}) &=& V^{i-1}_K(\L \otimes \cat{M}) \cdot c_K(\L \otimes \cat{M} \otimes \ga_i) \nonumber\\
&=& e_K(\ga_i) \cdot V^{i-1}_K(\L \otimes \cat{M}) + (1 - v e_K(\ga_i)) \cdot V^{i-1}_K(\L \otimes \cat{M}) \cdot c_K(\L \otimes \cat{M}). \nonumber
\end{eqnarray}
By induction, we can express $V^{i-1}_K(\L \otimes \cat{M})$ in terms of $V^s_K(\L) V^t_K(\cat{M})$. Also, 
$$c_K(\L \otimes \cat{M}) = V^1_K(\L) + V^1_K(\cat{M}) - v V^1_K(\L) V^1_K(\cat{M}).$$
Express the products by $b^{i,j}_s$ as before and we get
$$V^i_K(\L \otimes \cat{M}) = \sum_{s,t \geq 0}\,F^i_{s,t}\,V^s_K(\L) V^t_K(\cat{M})$$
for some uniquely determined elements $F^i_{s,t}$. On the other hand, in $\lazard^m_G(F)[v^{-1}]\{\{F\}\}$, we have
\begin{eqnarray}
\gD y(V^i) &=& \gD y(V^{i-1}) \cdot \gD y(\ga_i) \nonumber\\
&=& \gD y(V^{i-1}) \cdot \gD ( e(\ga_i) + (1 - v e(\ga_i)) y(\gep) ) \nonumber\\
&=& e(\ga_i) \cdot \gD y(V^{i-1}) + (1 - v e(\ga_i)) \cdot \gD y(V^{i-1}) \cdot \gD y(\gep) \nonumber
\end{eqnarray}
(The coproduct $\Delta$ is a unital, $\lazard_G(F)$-module homomorphism). 

Again, by induction, we can express $\gD y(V^{i-1})$ in terms of $y(V^s) \otimes y(V^t)$. Also, we have
$$\gD y(V^1) =  \gD y(\gep) = y(\gep) \otimes 1 + 1 \otimes y(\gep) - v \cdot y(\gep) \otimes y(\gep).$$
Hence, $F^i_{s,t} = f^i_{s,t}$ by the same reason (because $\{y(V^s) \otimes y(V^t)\}$ form a basis). That finishes the proof of the well-definedness of 
$$\gPS_K : R(G)[v,v^{-1}] \otimes_{\lazard_G(F)} \cob{}{G}{X} \to K_0(G;X)[v,v^{-1}].$$

For the surjectivity of $\Psi_K$, we proceed by induction on the dimension of $X$. W\withoutlog, we may assume $X$ to be \girred{G}. As a $R(G)[v,v^{-1}]$-module, $K_0(G;X)[v,v^{-1}]$ is generated by elements of the form $[\cat{E}]$ where $\cat{E}$ is a \glin{G}\ locally free sheaf of finite rank over $X$. If $\dim X = 0$, then $\cat{E}$ splits. So, we may assume the rank of $\cat{E}$ is 1. Then, we have 
\begin{eqnarray}
\label{eqn6}
\Psi_K([\id_X]- v\, [\id_X, \dual{\cat{E}}]) = [\O_X] - v\, c_K(\dual{\cat{E}}) = [\cat{E}]. 
\end{eqnarray}
That handles the $\dim X = 0$ case.

Suppose $\dim X > 0$. By the blow up relation in the \equi\ K-theory, if $Z \subset X$ is an \inv, \sm\ closed subscheme, then we have
$$\pi_![\O_{\tilde{X}}] - [\O_X] = - {i_Z}_! \circ {p_1}_![ \O_{\P_1}] + {i_Z}_! \circ {p_2}_! [\O_{\P_2}]$$
where $\pi : \tilde{X} \to X$ is the blow up of $X$ along $Z$, $i_Z : Z \embed X$ is the immersion, $p_1$ and $p_2$ are the maps $\P_1 \defeq \P(\O \oplus \dual{\nbundle{Z}{X}}) \to Z$ and $\P_2 \defeq \P(\O \oplus \O(1)) \to \P(\dual{\nbundle{Z}{X}}) \to Z$ respectively. Since $\dim \pi = \dim (i_Z \circ p_1) = \dim (i_Z \circ p_1) = 0$, if we multiply the above equation by $[\cat{E}]$, we have
\begin{eqnarray}
\pi_! [\pi^* \cat{E}] - [\cat{E}] &=& - [\cat{E}] \cdot {i_Z}_! \circ {p_1}_! [\O_{\P_1}] + [\cat{E}] \cdot {i_Z}_! \circ {p_2}_! [\O_{\P_2}] \nonumber\\
&=& - {i_Z}_! \circ {p_1}_!\, [p_1^*(\cat{E}|_Z)] + {i_Z}_! \circ {p_2}_!\, [p_2^*(\cat{E}|_Z)]. \nonumber
\end{eqnarray}
Therefore,
$$\pi_! [\pi^* \cat{E}] - [\cat{E}] = {i_Z}_!(x)$$
for some $x \in K_0(G;Z)[v,v^{-1}]$. By the induction assumption, the right hand side is in the image of $\Psi_K$. So it is enough to show that $[\pi^* \cat{E}]$ lies inside the image of $\Psi_K$. By Theorem \ref{thm splitting principle}, we may assume $\cat{E}$ splits. Hence, it is enough to show the surjectivity when $\rank{\cat{E}} = 1$, which follows from equation (\ref{eqn6}). That finishes the proof of Proposition \ref{prop map to K theory}.
\end{proof}

For the injectivity of $\Psi_K$, we will follow the same procedure as in section 4.2 in \cite{universal alg cobor}, start by constructing a left inverse of $\Psi_K$ called the Chern character. Since its definition is based on the higher Chern class operators (see section \ref{sect special theory}), we will assume $\char{k} = 0$ for the rest of this section.

For simplicity of notation, let 
$$A(-) \defeq R(G)[v,v^{-1}] \otimes_{\lazard_G(F)} \cob{G}{}{-}$$
be a theory on $\gvar{G}$ with the four basic operations. Then, $A(-)$ will have
$$c(\L \otimes \cat{M}) = c(\L) + c(\cat{M}) - vc(\L) c(\cat{M})$$
(as operators on $A(X)$) for any $X \in \gvar{G}$ and $\L, \cat{M} \in \picard{G}{X}$ from the multiplicative \fgl. Moreover, as pointed out in section \ref{sect special theory}, there is a canonical surjective ring \homo
$$\lazard_G^s(F) \to \lazard_G^m(F)[v^{-1}] \cong R(G)[v,v^{-1}].$$
Therefore, the projective bundle formula (Theorem \ref{thm PBF}) and the extended homotopy property (Proposition \ref{prop EH}) hold in $A(-)$. Furthermore, we have the splitting principle (Proposition \ref{prop splitting principle for special} and Remark \ref{rmk splitting principle for special}) and the higher Chern class operators for \glin{G}\ locally free sheaves are well-defined in $A(-)$.

Other than the list of properties in Proposition \ref{prop Chern class properties}, we will also need the Whitney sum formula to define our Chern character. To this end, we first need the following Lemma.

\begin{lemma}
\label{lemma invertible u}
For all $X \in \gvar{G}$, $\L \in \picard{G}{X}$, the endomorphism
$$u(\L) \defeq 1 - v c(\L) : A(X) \to A(X)$$
has an inverse of the form
$$\sum_{i = 0}^n a_i u(\L)^i$$
for some $n \geq 0$ and $a_i \in R(G)[v,v^{-1}]$.
\end{lemma}

\begin{proof}
We will prove this by induction on $\dim X$ following a similar procedure as in the proof of Theorem \ref{thm PBF}. In particular, we will show that $u(\L)$ is surjective and has a left inverse of the given form (which is also a right inverse). 

Suppose $\L$ is trivial, i.e., $\L \cong \beta$ for some $\beta \in G^*$. Then, 
$$u(\L) = 1 - v c(\L) = 1 - v e(\beta) = \dual{\beta} \in R(G)[v,v^{-1}]$$
has an inverse of the form $a_0 \defeq \beta$. That handles the cases when $\L$ is trivial or $\dim X = 0$ (by Proposition \ref{prop lb structure}).

Now, suppose $\dim X \geq 1$. Similar to the proof of Theorem \ref{thm PBF}, we have a filtration of \inv, closed subschemes $X_j \in \gvar{G}$
$$\emptyset = X_0 \subset X_1 \subset \cdots \subset X_m = X$$
such that $\L|_{X_j - X_{j-1}}$ is trivial for all $1 \leq j \leq m$. Let $Y \defeq X_j, Z \defeq X_{j-1}, U \defeq Y - Z$, $i : Z \embed X$ and $j : U \embed X$ be the immersions. Then, we have the following diagram (1) :

$$\begin{CD}
0 @>>> i_* A(Z) @>>> A(X) @>{j^*}>> A(U) @>>> 0 \\
@. @A{u(\L)}AA @A{u(\L)}AA @A{u(\L|_U)}AA @. \\
0 @>>> i_* A(Z) @>>> A(X) @>{j^*}>> A(U) @>>> 0 \\
\end{CD}$$
and the following diagram (2) :
$$\begin{CD}
A(Z) @>{i_*}>> A(X)  \\
@A{u(\L|_Z)}AA @A{u(\L)}AA  \\
A(Z) @>{i_*}>> A(X)
\end{CD}$$

As in the proof of Theorem \ref{thm PBF}, the rows of diagram (1) are exact and both diagrams are commutative. By induction and the trivial case, the left and right columns of diagram (1) is surjective, so is its middle column. Moreover, by induction, the left inverse of $u(\L|_Z)$ is of the form 
$$\psi_Z \defeq \sum_{i = 0}^n a_i u(\L|_Z)^i.$$
So we can extend it to a map 
$$\overline{\psi}_Z \defeq \sum_{i = 0}^n a_i u(\L)^i : A(X) \to A(X),$$
i.e., $i_* \circ \psi_Z = \overline{\psi}_Z \circ i_*$. By the trivial case, the left inverse of $u(\L|_U)$ is of the form 
$$\psi_U \defeq \sum_{i = 0}^m b_i u(\L|_U)^i,$$
which can be extended to a map
$$\overline{\psi}_U \defeq \sum_{i = 0}^m b_i u(\L)^i : A(X) \to A(X),$$
i.e., $j^* \circ \overline{\psi}_U = \psi_U \circ j^*$. Hence,
$$\psi \defeq \overline{\psi}_Z + \overline{\psi}_U - \overline{\psi}_Z \circ u(\L) \circ \overline{\psi}_U$$
is a left inverse of $u(\L)$, which is of the form desired.
\end{proof}

\begin{prop}
\label{prop Whitney sum}
Suppose $X \in \gvar{G}$ and $\cat{E}, \cat{E}', \cat{E''}$ are \glin{G}\ locally free sheaves of finite ranks over $X$.

\noindent \begin{statementslist}
{\rm (1)} & If there is an exact sequence $0 \to \cat{E}' \to \cat{E} \to \cat{E}'' \to 0$, then
$$c_i(\cat{E}) = \sum_{j=0}^i c_j(\cat{E}') c_{i-j}(\cat{E}'')$$
\\
{\rm (2)} & $$c_i(\cat{E}) c_j(\cat{E}') = c_j(\cat{E}') c_i(\cat{E}) $$
\end{statementslist}
\end{prop}

\begin{proof}
\noindent (1) : \tab By Theorem \ref{thm gen by geo cycle}, it is enough to consider geometric cycle. Since
$$c_i(\cat{E})[f : Y \to X] = c_i(\cat{E}) f_* [\id_Y] = f_* c_i(f^*\cat{E}) [\id_Y],$$
it is enough to consider $[\id_X]$ where $X \in \gsmcat{G}$ is \girred{G}. By splitting principle (Proposition \ref{prop splitting principle for special} and Remark \ref{rmk splitting principle for special}), it is enough to show that
$$c_i(\cat{E})[\id_X] = S_i(c(\L_1), \ldots, c(\L_r)) [\id_X]$$
where $\cat{E} \cong \oplus_{i=1}^r \L_i$ and $S_i$ is the symmetric polynomial with degree $i$.

Let 
$$\pi : \P \defeq \P(\cat{E}) \to X$$
be the projection and $H_i \defeq \P(\oplus_{j \neq i} \L_j)$ be an \inv\ \sm\ divisor on $\P$. Then, 
$$\O_{\P}(H_i) \otimes \pi^* \L_i \cong \O_{\P}(1).$$
Therefore,
\begin{eqnarray*}
0 &=& c(\O(H_1)) c(\O(H_2)) \cdots c(\O(H_r)) [\id_{\P}] \\
&=& \prod_{i=1}^r c(\O(1) \otimes \pi^* \dual{\L_i})  [\id_{\P}] \\
&=& \prod_{i=1}^r (\xi + c(\pi^* \dual{\L_i}) - v \xi c(\pi^* \dual{\L_i}))  [\id_{\P}],
\end{eqnarray*}
where $\xi \defeq c(\O_{\P}(1))$, which is equal to 
$$\prod_{i=1}^r (u(\pi^* \dual{\L_i})\xi + c(\pi^* \dual{\L_i}))  [\id_{\P}].$$
By Lemma \ref{lemma invertible u}, $u(\pi^* \dual{\L_i})$ an isomorphism. So we have
$$0 = \prod_{i=1}^r (\xi + u(\pi^* \dual{\L_i})^{-1} c(\pi^* \dual{\L_i}))  [\id_{\P}].$$

\medskip

\noindent Claim 1 : $u(\dual{\L})^{-1} c(\dual{\L}) = -c(\L)$ as endomorphisms on $A(X)$ for all $X \in \gvar{G}$ and $\L \in \picard{G}{X}$.

It is enough to show
$$c(\dual{\L}) = - u(\dual{\L}) c(\L).$$
As before, it is enough to consider $[\id_X]$ where $X \in \gsmcat{G}$ is \girred{G}. Then,
$$0 = c(\L \otimes \dual{\L}) [\id_X] = (c(\L) + c(\dual{\L}) - vc(\L) c(\dual{\L}))[\id_X] = (c(\dual{\L}) + u(\dual{\L})c(\L)) [\id_X]$$
and the result follows. \claimend

\medskip

By Claim 1, we have
\begin{eqnarray*}
0 &=& \prod_{i=1}^r (\xi - c(\pi^* \L_i))  [\id_{\P}] \\
&=& \sum_{i=0}^r (-1)^i \xi^{r-i} S_i(c(\pi^* \L_1), \ldots, c(\pi^* \L_r)) [\id_P] \\
&=& \sum_{i=0}^r (-1)^i c(\O_{\P}(1))^{r-i} \circ \pi^* \circ S_i(c(\L_1), \ldots, c(\L_r)) [\id_X].
\end{eqnarray*}
The result then follows from the definition of the higher Chern class operators and the projective bundle formula.

\medskip

\noindent (2) : \tab It follows from the splitting principle (Proposition \ref{prop splitting principle for special}) and part (1).
\end{proof}

Next, similar to \cite{universal alg cobor}, we need the following technical results (see section 4.2.1 in \cite{universal alg cobor}).

\begin{prop}
\label{prop tower in special theory}
Suppose $Y \in$ $\gsmcat{G}$\ is \girred{G}, $\P \to Y$ is a \qadtower. Then
$$[\pi : \P \to Y] = v^{\dim \pi} [\id_Y]$$
in $A(Y)$. The same holds in the \equi\ K-theory, i.e., 
$$\pi_{!} [\O_{\P}] = v^{\dim \pi} [\O_Y]$$
in $K_0(G;Y)[v,v^{-1}]$.
\end{prop}

\begin{proof}
Let us consider the $A(-)$ theory first. The arguments for the \equi\ K-theory will be analogous.

This proof is very similar to the proof of Proposition \ref{prop tower reduction} and Theorem \ref{thm gen by lazard}, so we will only give a sketch. We will prove the statement by induction on $\dim Y$ and then $\dim \P$. First of all, by the blow up relation,
$$[\blowup{Y}{Z} \to Y] - [\id_Y] = -[\P_1 \to Z \embed Y] + [\P_2 \to Z \embed Y]$$
for some \qadtower s $\P_1, \P_2$. By induction, we have $[\blowup{Y}{Z} \to Y] = [\id_Y]$. We can then apply $\pi_* \circ \pi^*$ on both sides to get $[\tilde{\P} \to \blowup{Y}{Z} \to Y] = [\P \to Y]$. By the same arguments in step 1 of the proof of Proposition \ref{prop tower reduction}, we may assume $\P$ is given by invertible sheaves (Step 1). In fact, we may further assume $\P$ is a \adtower\ (Step 2).

By lemma \ref{lemma tower twisting} (and remark \ref{rmk twisting tower same structure}), 
$$ [\P' \to Y] - [\P \to Y] = [Q_0 \to D \embed Y] - [Q_1 \to D \embed Y] +\ [Q_2 \to D \embed Y] - [Q_3 \to D \embed Y]$$
where $\P'$ is $\P$ twisted by some \inv\ \sm\ divisor $D \subset Y$ and $Q_i$ are \qadtower s over $D$. By the induction assumption, $[\P' \to Y] = [\P \to Y]$. So we may assume $\P \cong Q \x Y$ for some \adtower\ $Q \to \pt$. By the induction assumption, $[Q] = v^{\dim Q}[\id_{\pt}]$ and we are done.

Now, we need to handle the case when $\dim Y = 0$. In this case, $\P \cong Q \x Y$ for some \adtower\ $Q \to \pt$. So we reduce to the case when $Y = \pt$. The result then follows from the same arguments (with multiplicative \fgl) in step 2, 3 and claim 4 in the proof of Theorem \ref{thm gen by lazard}.

For the \equi\ K-theory, since the blow up relation, analogues of Lemma \ref{lemma tower twisting} and remark \ref{rmk twisting tower same structure} still hold, we can apply the same arguments to reduce to the case when $\P \cong Q \x Y$ for some \adtower\ $Q \to \pt$. Then,
$$\pi_! [\O_{\P}] = \pi_! \circ \pi_1^* [\O_Q] = \pi_Y^* \circ \pi_{Q!} [\O_Q]$$
where $\pi_Q, \pi_Y$ are the structure morphisms. That reduces to the case when $Y = \pt$. The rest then follows from the same arguments (with multiplicative \fgl) in step 2, 3 and claim 4 in the proof of Theorem \ref{thm gen by lazard}.
\end{proof}

\begin{cor}
\label{cor blow up in special theory}
Suppose $Z$ is an \inv\ closed subscheme of $Y$ such that $Z$, $Y$ are both in $\gsmcat{G}$. Then, 
$$[\blowup{Y}{Z} \to Y] = [\id_Y] $$
in $A(Y)$. The same holds in the \equi\ K-theory.
\end{cor}

\begin{proof}
By the blow up relation and Proposition \ref{prop tower in special theory}.
\end{proof}

Now, similar to section 4.2.2 in \cite{universal alg cobor}, we define the Chern character 
$$\tilde{ch} : K_0(G;X)[v,v^{-1}] \to \Endo{A(X)}$$
$$\tilde{ch}(\cat{E}v^n) \defeq (\rank{\cat{E}} - v c_1(\dual{\cat{E}}))v^n$$
for $X \in \gsmcat{G}$ (Well-defined because of Proposition \ref{prop Whitney sum}) and 
$$ch : K_0(G;X)[v,v^{-1}] \to A(X)$$
$$ch(\cat{E}v^n) \defeq \tilde{ch}(\cat{E}v^n)[\id_X].$$

\begin{prop}
\label{prop ch properties}
Suppose $f : X' \to X$ is a \morp\ in $\gsmcat{G}$, $\L, \cat{M} \in \picard{G}{X}$ and $\cat{E}, \cat{F}$ are \glin{G}\ locally free sheaves over $X$.

\noindent \begin{statementslist}
{\rm (1)} & $$\tilde{ch}([\cat{E}]) \circ f_! = f_* \circ \tilde{ch}([f^* \cat{E}]) \text{\tab if $f$ is \proj}$$
$$\tilde{ch}([f^* \cat{E}]) \circ f^* = f^* \circ \tilde{ch}([\cat{E}]) \text{\tab if $f$ is \sm}$$ \\
{\rm (2)} & If $Z \in \gsmcat{G}$ is an \inv\ closed subscheme of $X$ and $\pi : \blowup{X}{Z} \to X$ is the blow up, then
$$ch([\cat{E}]) = \pi_* ch([\pi^* \cat{E}])$$ \\
{\rm (3)} & $$\tilde{ch}(\L \otimes \cat{M}) = \tilde{ch}(\L) \circ \tilde{ch}(\cat{M})$$ \\
{\rm (4)} & $$\tilde{ch}(c_K(\L) \cdot [\cat{M}]) = c(\L) \circ \tilde{ch}(\cat{M})$$ \\
{\rm (5)} & $$\tilde{ch}(\cat{E} \otimes \cat{F}) = \tilde{ch}(\cat{E}) \circ \tilde{ch}(\cat{F})$$ \\
{\rm (6)} & $$\tilde{ch}(c_K(\L) \cdot [\cat{E}]) = c(\L) \circ \tilde{ch}(\cat{E})$$ \\
\end{statementslist}
\end{prop}

\begin{proof}
Part (1) follows from the basic properties of the higher Chern class operators (Proposition \ref{prop Chern class properties}). Part (2) follows from Corollary \ref{cor blow up in special theory}. Parts (3) and (4) simply follow from the multiplicative \fgl. Parts (5) and (6) follow from the splitting principle (Proposition \ref{prop splitting principle for special}) and parts (3) and (4) respectively.
\end{proof}

\begin{cor}
$\tilde{ch}$ is a ring homomorphism and $ch$ commutes with \sm\ pull-backs, the first Chern class operators for \glin{G}\ invertible sheaves and external products.
\end{cor}

\begin{proof}
By Proposition \ref{prop Chern class properties} and \ref{prop ch properties}.
\end{proof}

\begin{prop}
\label{prop ch comm with push forward}
$ch$ commutes with \proj\ push-forwards.
\end{prop}

\begin{proof}
As this proof is highly similar to the proof of Proposition 4.2.9 in \cite{universal alg cobor}, we will only give a sketch here.

Since the splitting principle holds in $A(-)$ (Proposition \ref{prop splitting principle for special}) and the \textbf{(Sect)} axiom holds in the \equi\ K-theory (Lemma \ref{lemma Sect in K theory}), the \equi\ analogue of Lemma 4.2.8 in \cite{universal alg cobor} holds, by the same arguments. As in the proof of Proposition 4.2.9 in \cite{universal alg cobor}, by Proposition \ref{prop equi embed}, it is enough to consider the cases when $i : Z \embed X$ in $\gsmcat{G}$ is a closed immersion or when $\pi_2 : \P(V) \x X \to X$ is a projection. 

For the case $\pi : \P(V) \x X \to X$, by the projection formula and the projective bundle formula in the \equi\ K-theory (see Theorem 10 in \cite{equi K theory 2}) and the projective bundle formula in $A(-)$, one can reduce to showing 
$$ch(\pi_! (c_K(\O(1))^n)) = \pi_* ch( c_K(\O(1))^n)$$
with $X = \pt$ and $0 \leq n \leq \dim V - 1$. 

Let $V \cong \beta_0 \oplus \cdots \oplus \beta_r$ for some $\beta_i \in G^*$, $\P \defeq \P(V)$, $H_i \defeq \P(\oplus_{j \neq i} \beta_j)$ as an \inv\ \sm\ divisor on $\P$ and
$$\xi^i_j \defeq c(\O_{\P}(H_i)) \circ c(\O_{\P}(H_{i+1})) \circ \cdots \circ c(\O_{\P}(H_j))$$
if $0 \leq i \leq j \leq r$ (as in $A(-)$ or $K_0(G;-)[v,v^{-1}]$). Since $\O_{\P}(1) \cong \O_{\P}(H_i) \otimes \beta_i$, we have
$$c_K(\O(1)) = c_K(\O(H_r) \otimes \beta_r) = e(\beta_r) + (1 - ve(\beta_r))c(\O(H_r)) = a + b \xi^r_r$$
for some $a,b \in R(G)[v,v^{-1}]$. Inductively, it can be shown that, for $0 \leq n \leq r$, 
$$c_K(\O(1))^n = a_0 + a_1 \xi^r_r + \cdots + a_n \xi^{r-n+1}_r$$
for some $a_i \in R(G)[v,v^{-1}]$. Since $ch(\xi^i_j) = \xi^i_j [\id_{\P}]$, the statement will be true if we can show that
$$ch(\pi_! [\O_{\P(V)}]) = \pi_* ch(\O_{\P(V)})$$
for arbitrary $G$-\repn\ $V$. By Proposition \ref{prop tower in special theory}, both sides of the equation are $v^{\dim \P(V)}$. So we are done.

The case $i : Z \embed X$ follows from the exact same arguments as in the proof of Proposition 4.2.9 in \cite{universal alg cobor}.
\end{proof}

We are now ready to prove one of our major results in this paper : the equivariant Conner-Floyd isomorphism.

\begin{thm}
\label{thm equi Conner Floyd}
Suppose $\char{k} = 0$ and $k$ contains a primitive $e$-th \rou, where $e$ is the exponent of $G_f$. Then there is a canonical $R(G)[v,v^{-1}]$-module isomorphism 
$$\gPS_K : R(G)[v,v^{-1}] \otimes_{\lazard_G(F)} \cob{}{G}{X} \to K_0(G;X)[v,v^{-1}],$$
for any $X \in$ $\gsmcat{G}$, and it commutes with the four basic operations.
\end{thm}

\begin{proof}
By Proposition \ref{prop map to K theory}, it remains to show that $\Psi_K$ is injective. As mentioned before, we will show that the Chern character is a left inverse of $\Psi_K$. By Theorem \ref{thm gen by geo cycle}, it is enough to show that
$$ch \circ \gPS_K [f : Y \to X] = [f : Y \to X]$$
for arbitrary geometric cycle $[f : Y \to X]$. But since $\Psi_K$ and $ch$ both commute with push-forwards (Proposition \ref{prop map to K theory} and \ref{prop ch comm with push forward}), we have
\begin{eqnarray*}
ch \circ \gPS_K [f : Y \to X] &=& ch \circ \gPS_K \circ f_* [\id_Y] \\
&=& f_* \circ ch \circ \gPS_K [\id_Y] \\
&=& f_* [\id_Y] \\
&=& [f : Y \to X]
\end{eqnarray*}
as desired.
\end{proof}

\bigskip

\bigskip

\section{Realization functor}
\label{sect realization functor}

In this section, we will establish a realization functor from our equivariant algebraic cobordism theory to Tom Dieck's equivariant complex cobordism theory, when $k = \C$ and $G$ is a compact abelian Lie group.

Recall from section 1 in \cite{complex cobor} that, for a compact Lie group $G$ and a pointed $G$-manifold $X$, 
$$\widetilde{MU}_G^{2k}(X) \defeq \lim_{\stackrel{\longto}{V}} [S^V \wedge X, MU(\dim_{\C} V + k, G)]$$
where $V$ runs through all finite dimensional complex $G$-\repn s, $S^V$ is the one-point compactification of $V$, $MU(r,G)$ is the Thom space of the universal rank $r$ complex $G$-vector bundle and $[-,-]$ denotes pointed $G$-homotopy set. 

Then, the Tom Dieck's \equi\ complex cobordism abelian group for a $G$-manifold $X$ is defined to be
$$MU_G^{2k}(X) \defeq \widetilde{MU}_G^{2k}(X^+)$$
where $X^+$ is $X$ with a separate base point.

This theory has a naturally defined pull-back for any $G$-map and a cup product :
$$MU_G^{2r}(X) \x MU_G^{2s}(X) \stackrel{\cup}{\longto} MU_G^{2r+2s}(X).$$ 
With this cup product, $MU_G(X)$ becomes a commutative ring, or $MU_G$-algebra, with unity $[X^+ \to MU(0,G) = S^0]$. Also, pull-back is a ring \homo. For any \equi\ line bundle $L \to X$, we have an element 
$$c_1(L) \defeq [X^+ \stackrel{a}{\to} M(L) \stackrel{b}{\to} MU(1, G)] \in MU_G^2(X)$$ 
where $M(L)$ is the Thom space of $L$, $a$ is given by the zero section and $b$ is given by the classifying map. We can then define a (first) Chern class operator
$$c_1(L) : MU_G^{2k}(X) \to MU_G^{2k + 2}(X)$$
by sending $x$ to $c_1(L) \cup x$. It also has a naturally defined external product :
$$MU_G^{2r}(X) \x MU_G^{2s}(Y) \stackrel{\wedge}{\longto} MU_G^{2r+2s}(X \x Y).$$ 
It then forms a multiplicative \equi\ cohomology theory. 

For any $r$-dimensional complex $G$-\repn\ $W$, we have a suspension iso\morp
$$MU_G^{2k}(X) \iso \widetilde{MU}_G^{2k + 2r}(S^W \wedge X^+).$$
More generally, for any rank $r$ \equi\ vector bundle $E \to X$, the map $[M(E) \to MU(r, G)]$ given by the classifying map defines an element in $\widetilde{MU}_G^{2r}(M(E))$, which is called the Thom class of $E$ and will be denoted by $Th(E)$. 

Suppose further that $G$ is abelian. Then, we have the Thom iso\morp
$$MU_G^{2k}(X) \iso \widetilde{MU}_G^{2k + 2r}(M(E))$$
defined by sending $x$ to $d^*(x \wedge Th(E))$ where $d : M(E) \to X^+ \wedge M(E)$ is given by projection (see section 2 in \cite{completion of complex cobor}). Moreover, this theory also has a canonical $MU_G$-orientation (following the terminology in \cite{equi Riemann Roch}) : 
$$Th(\O_{\P(\cat{U})}(1)) \in \widetilde{MU}_G^2(M(\O_{\P(\cat{U})}(1))) = \widetilde{MU}_G^2(MU(1,G))$$
where $\O_{\P(\cat{U})}(1)$ is the universal complex $G$-line bundle. Indeed, $MU_G(-)$ is the universal complex oriented cohomology theory with orientation in degree 2 (when $G$ is abelian, see Theorem 1.2 in \cite{universality of complex cobor}). 

As pointed out in Example 11.3 in \cite{equi FGL}, for a compact abelian Lie group G and a complex oriented cohomology theory $E_G^*(-)$, the pair $E_G$, $E_G(\P(\cat{U}))$ with $y(\epsilon)$ being the complex orientation, defines a $G$-\fgl. In particular, for a complete $G$-flag $F$, the pair $MU_G$, $MU_G(\P(\cat{U})) \cong MU_G\{\{F\}\}$ (as $MU_G$-module, by Theorem 9.6 in \cite{equi FGL 2}) forms a $(G,F)$-\fgl. Hence, there is a canonical ring \homo
$$\nu : \lazard_G(F) \to MU_G.$$
By Theorem 13.1 in \cite{equi FGL 2}, $\nu$ is surjective and it is conjectured to be an iso\morp\ for all compact abelian Lie group (see section 13 in \cite{equi FGL 2}). 

Our goal in this section is to define a realization functor :
$$\Psi_{\rm Top} : \cob{G}{}{-} \to MU_G^{}(-),$$
which commutes with the four basic operations. To this end, we first need to have a \proj\ push-forward (also known as Gysin \homo) in $MU_G(-)$. For simplicity, we will assume $G$ to be abelian for the rest of this section.

Recall the following definition of Gysin \homo\ from section 2 in \cite{equi Riemann Roch}. Suppose $f : X \to Y$ is a map in $\gsmcat{G}$\ such that $X$, $Y$ are both equidimensional and \proj\ (In \cite{equi Riemann Roch}, $X$, $Y$ are required to be $MU_G$-oriented, i.e., their tangent bundles are $MU_G$-oriented. But in $MU_G(-)$, all \equi\ vector bundles have Thom classes and hence, $MU_G$-oriented). Define the Gysin \homo
$$f_! : MU_G^{2k}(X) \to MU_G^{2k - 2 \dim f}(Y)$$
as the composition
$$MU_G^{2k}(X) \stackrel{a}{\iso} \widetilde{MU}_G^{2k + 2r}(M(N_f)) \stackrel{b^*}{\to} \widetilde{MU}_G^{2k + 2r}(S^{V_X} \wedge Y^+) \stackrel{c}{\iso} MU_G^{2k - 2 \dim f}(Y)$$
where $V_X$ is a $G$-\repn, $X \embed V_X$ is an \equi\ embedding (in the sense of Topology), $N_f$ is the rank $r$ normal bundle of the embedding $X \embed Y \x V_X$, which will be identified with the tubular neighborhood of $X$ inside $Y \x V_X$ and it is assumed to be inside $Y \x {\rm Int\,} D(V_X)$ (interior of the unit disk), $a$ is the Thom iso\morp, $b$ is the collapsing map and $c$ is the suspension iso\morp. By Lemma 2.2 in \cite{equi Riemann Roch}, the above definition is independent of all choices made.

Since we would like to consider $MU_G(-)$ as a theory with four basic operations : \proj\ push-forward, \sm\ pull-back, (first) Chern class operator and external product, we will focus on the full subcategory of $\gsmcat{G}$\ consisting of \proj\ objects, denoted by $\gprojsmcat{G}$, for the rest of this section.

\begin{prop}
\label{prop MU basic properties}
For the theory $MU_G(-)$ on $\gprojsmcat{G}$,

\noindent \begin{statementslist}
{\rm (1)} & $c(f^*L) = f^*(c(L))$ \\
{\rm (2)} & If $x \in MU_G(X)$ and $y \in MU_G(Y)$, then $x \wedge y = \pi_1^*(x) \cup \pi_2^*(y)$ \\
{\rm (3)} & If $x, y \in MU_G(X)$, then $x \cup y = \Delta^* (x \wedge y)$ where $\Delta$ is the diagonal map \\
{\rm (4)} & $f_!(x \cup f^*(y)) = f_!(x) \cup y$ \\
{\rm (5)} & {\rm \textbf{(D1)-(D4), (A1)-(A8)} hold (see \cite{universal alg cobor})}
\end{statementslist}
\end{prop}

\begin{proof}
(1), (2), (3) follow directly from definition. \textbf{(D2), (D3), (D4)} are what we call pull-back, (first) Chern class operator and external product, which we just defined. 

\noindent \textbf{(D1)} and (4) : From Lemma 2.2 in \cite{equi Riemann Roch}.

\noindent \textbf{(A1)} : Follows from the fact that $MU_G(-)$ is a cohomology theory.

\noindent \textbf{(A2)} : Consider the following diagram
\squarediagramword{W}{X}{Y}{Z}{g'}{f'}{f}{g}
which is Cartesian and $W, X, Y, Z$ are objects in $\gprojsmcat{G}$. Choose equivariant embeddings $e_X : X \embed V_X$ and $e_Y : Y \embed V_Y$. Then, we have a Cartesian diagram
\squarediagramword{X \x_{Z \x V_X} (Y \x V_X)}{X}{Y \x V_X}{Z \x V_X}{a}{b}{c}{d}
where $d = g \x \id_{V_X}$, $c = f \x e_X$ is an equivariant embedding and so is $b$. Notice that 
$$X \x_{Z \x V_X} (Y \x V_X) \cong X \x_Z Y = W$$
and $a$ is isomorphic to $g'$. Also, if we denote the normal bundle of $c : X \embed Z \x V_X$ by $N_f$, then  and the normal bundle of $b : W \embed Y \x V_X$ will be $g'^* N_f$.

To define $N_{f'}$ for the Gysin \homo, we let $V_W \defeq V_X \x V_Y$ and consider the equivariant embedding 
$$Y \x V_X \embed Y \x V_X \x V_Y = Y \x V_W$$
which sends $(y,v)$ to $(y,v,e_Y(y))$. Then the composition
$$W \stackrel{b}{\embed} Y \x V_X \embed Y \x V_W,$$
which sends $(x,y)$ to $(y, e_X(x), e_Y(y))$, will be an equivariant embedding and the corresponding normal bundle $N_{f'}$ of $W$ is isomorphic to $(g'^* N_f) \x V_Y$.

Now, for an element $t = [S^V \wedge X^+ \to MU(k,G)] \in MU_G(X)$,
\begin{eqnarray*}
&& f_!' \circ {g'}^*(t) \\
&=& [S^V \wedge Y^+ \wedge S^{V_X} \wedge S^{V_Y} \to S^V \wedge M(N_{f'}) \to S^V \wedge W^+ \wedge M(N_{f'}) \to MU(r'+k,G)] 
\end{eqnarray*}
where $r' \defeq \rank{N_{f'}}$. On the other hand,
\begin{eqnarray*}
g^* \circ f_!(t) &=&  [S^V \wedge Y^+ \wedge S^{V_X} \to S^V \wedge Z^+ \wedge S^{V_X}  \\
&& \to S^V \wedge M(N_f) \to S^V \wedge X^+ \wedge M(N_f) \to MU(r+k,G)] 
\end{eqnarray*}
where $r = \rank{N_f}$, which is equal to 
$$[S^V \wedge Y^+ \wedge S^{V_X} \wedge S^{V_Y} \to \cdots \to MU(r+k,G) \wedge S^{V_Y} \to MU(r'+k,G)] $$
in the direct system. By following the definitions, it can be seen that the above two maps agree.

\noindent \textbf{(A3)} : Follows from (1) and (4).

\noindent \textbf{(A4)} : Follows from (1) and the fact that pull-back is a ring \homo.

\noindent \textbf{(A5)} : Follows from the fact that $MU_G(X)$ is a commutative ring.

\noindent \textbf{(A6)} : If $X \embed V_X$, $Y \embed V_Y$, then $X \x Y \embed V_X \x V_Y$ and $N_{f \x g} \cong (\pi_1^* N_f) \oplus (\pi_2^* N_g)$.

\noindent \textbf{(A7)} : Follows from definition.

\noindent \textbf{(A8)} : Follows from (2).
\end{proof}

Next, we need an analogue of Proposition \ref{prop Nilp axiom}\ in $MU_G(-)$. For an equivariant line bundle $L$ over $Y \in$ $\gprojsmcat{G}$, let 
$$V^n(L) \defeq c(L \otimes \alpha_1) \circ \cdots \circ c(L \otimes \alpha_n)$$
(Here we are considering $\alpha_i$ as an equivariant line bundle over $Y$).

\begin{lemma}
\label{lemma axioms hold in MU}
\textbf{(Sect)}, \textbf{(EFGL)}, double point relation, blow up relation, extended double point relation hold in $MU_G(-)$. Moreover, for all \girred{G}\ $Y \in$ $\gprojsmcat{G}$, $L$ \equi\ line bundle over $Y$, finite set $S$ as in (\ref{eqn basic element}), 
$$V^n_S(L)(1_Y) = 0$$
in $MU_G(Y)$, for sufficiently large $n$.
\end{lemma}

\begin{proof}
For \textbf{(Sect)}, let $D \subset Y \in$ $\gprojsmcat{G}$\ be an \inv\ \sm\ divisor and $L$ is the line bundle corresponding to $\O_Y(D)$. We need to show $c(L) = i_!(1_D)$ where $i : D \embed Y$ is the closed immersion. Choose an equivariant embedding $D \embed V_D$. Let $r$ be the rank of its normal bundle $N_i$. Then,
$$i_!(1_D) = [S^{V_D} \wedge Y^+ \to M(N_i) \to MU(r,G)]$$
and 
$$c(L) = [Y^+ \to M(L) \to MU(1,G)] = [S^{V_D} \wedge Y^+ \to M(L \x V_D) \to MU(r,G)].$$
So they agree.

For \textbf{(EFGL)}, see remark \ref{rmk why EFGL}. The other properties follow from the same proofs as in section \ref{sect basic properties}\ and Proposition \ref{prop Nilp axiom}, with the correspondence 
$$[f : Y \to X] = f_! (1_Y).$$
\end{proof}

Now, for all $X \in$ $\gprojsmcat{G}$, we define a $\lazard_G(F)$-module \homo
$$\Psi_{\rm Top} : \lazard {\rm Z}_{G,F}(X) \to MU_G(X)$$
by sending $a [f : Y \to X, \L_1, \ldots, \L_r]$ to $\nu(a) f_!(c(L_1) \cdots c(L_r))$ where $L_i$ is the \equi\ line bundle corresponding to $\L_i$. The case of infinite cycles is covered because of Lemma \ref{lemma axioms hold in MU}.

\begin{thm}
\label{thm realization fct}
$\Psi_{\rm Top}$ descends to a canonical $\lazard_G(F)$-\homo
$$\cob{G}{}{X} \to MU_G(X),$$
for all $X \in$ $\gprojsmcat{G}$, and it commutes with the four basic operations. When $X$ is equidimensional, there is a canonical grading on $\cob{G}{}{X}$ and $\Psi_{\rm Top} : \cob{G}{*}{X} \to MU_G^{2*}(X).$
\end{thm}

\begin{proof}
It clearly commutes with the operations. By Lemma \ref{lemma axioms hold in MU}, \textbf{(Sect)} and \textbf{(EFGL)} hold in $MU_G(-)$ and hence, $\Psi_{\rm Top} : \cob{G}{}{X} \to MU_G(X)$ is well-defined.

For the (cohomological) grading, first of all, $\lazard_G(F)$ has a natural grading, by Remark \ref{rmk grading on equi lazard}. We then define 
$$\deg [f : Y \to X, \L_1, \ldots, \L_r] \defeq r - \dim f.$$ 
Since $\cob{}{G}{-} \cong \cob{}{G}{-}_{\rm fin}$ (as in section \ref{sect fundamental properties}) and this grading is preserved by \textbf{(Blow)}, \textbf{(Nilp)}, \textbf{(Sect)}, \textbf{(EFGL)}, it defines a grading on $\cob{}{G}{-}$ (It is not well-defined in $\basicmod{G}{F}{}{-}$ because of the infinite cycles). For a homogeneous element $a \in \lazard_G(F)$,
$$\deg (a [f : Y \to X, \L_1, \ldots, \L_r]) = \deg a + r - \dim f$$
and 
$$\deg (\nu(a) f_!(c(L_1) \cdots c(L_r))) = \deg \nu(a) + 2r - 2\dim f.$$
Therefore, it is enough to show $\nu : \lazard_G(F)^* \to MU_G^{2*}$. Notice that $\lazard_G(F)$ is generated by $e(\alpha)$ and $f^1_{s,t}$. Moreover, $\deg e(\alpha) = 1$ and 
$$\nu(e(\alpha)) = \Psi_{\rm Top}(c(\alpha)[\id_{\pt}]) = c(\alpha),$$
which is of degree 2. It remains to show the statement for $f^1_{s,t}$.

Let $X \defeq \P(\dual{\alpha_1} \oplus \dual{\alpha_2})$ and $Y \defeq X \x X$. Then, in $\cob{}{G}{Y}$, we have
\begin{eqnarray}
&& V^1(\O(1,1))[\id_Y] \nonumber\\
&=& \sum_{s,t} f^1_{s,t} V^s(\O(1,0)) V^t(\O(0,1)) [\id_Y] \nonumber\\
&=& c(\O(1,0)) [\id_Y] + c(\O(0,1)) [\id_Y] + f^1_{1,1} V^1(\O(1,0)) V^1(\O(0,1)) [\id_Y] \nonumber\\
&& +\ f^1_{2,1} V^2(\O(1,0)) V^1(\O(0,1)) [\id_Y] + \cdots \nonumber\\
&=& [\P(\dual{\alpha_2}) \x X \embed Y] + [X \x \P(\dual{\alpha_2}) \embed Y] + f^1_{1,1} [\P(\dual{\alpha_2}) \x \P(\dual{\alpha_2}) \embed Y] \nonumber
\end{eqnarray}
Push this equality down to $\pt$, we have
$$[Y, \O(1,1)] = 2 [X] + f^1_{1,1}.$$
Hence,
$$\nu(f^1_{1,1}) = \Psi_{\rm Top}([Y, \O(1,1)] - 2 [X]).$$
It shows that $\deg \nu(f^1_{1,1}) = -2 = 2 \deg{f^1_{1,1}}$. The result for general $f^1_{s,t}$ follows from similar arguments, inductively.
\end{proof}
 
\begin{rmks} 
{\rm 
The proof of Theorem \ref{thm realization fct}\ also gives a geometric description to the elements $f^1_{s,t} \in \cob{}{G}{\pt}$ and $\nu(f^1_{s,t}) \in MU_G$.
}
\end{rmks}

\bigskip

\bigskip

\section{Flag dependency}
\label{sect flag indep}

In this last section, we will show that our definition of $\cob{}{G}{-}$ is indeed independent of the choice of the complete $G$-flag $F$. We will use the same notations and assumptions on $G$, $k$ and $F$ as in section \ref{sect notation}, except that we will not assume $\alpha_1 = \epsilon$ anymore.

Throughout this paper, we have been using a flag-dependent definition of equivariant formal group law over a commutative ring $R$. But there is indeed a flag-independent definition, which is called $G$-\equi\ formal group law over $R$ (see section 11 of \cite{equi FGL} for details). Since it is shown in section 13 of \cite{equi FGL} that it is equivalent to the flag-dependent definition we have been using, we will not include the details about this notion here. 

Suppose $F'$ is another complete $G$-flag given by 
$$0 = W^0 \subset W^1 \subset W^2 \subset \cdots.$$ 
Denote the 1-dimensional characters $W^i/W^{i-1}$ by $\beta_i$. Again, we will not assume $\beta_1 = \epsilon$. 

According to the results in section 13 of \cite{equi FGL}, there is a canonical map
$$\phi : \lazard_G(F) \to \lazard_G(F')$$
defined as follow.

By definition, $\lazard_G(F')\{\{F'\}\}$ is a $(G,F')$-\equi\ formal group law over $\lazard_G(F')$. By Lemma 13.1 in \cite{equi FGL}, it defines a $G$-\equi\ formal group law over $\lazard_G(F')$ (flag-independent). By Lemma 13.2 in \cite{equi FGL}, it then defines a $(G,F)$-\equi\ formal group law over $\lazard_G(F')$. Then, $\phi : \lazard_G(F) \to \lazard_G(F')$ is the unique map given by the universal property of $\lazard_G(F)$. By symmetry, there is also a canonical map $\phi' : \lazard_G(F') \to \lazard_G(F)$. 

Notice that we have the following commutative diagram :

\medskip

\begin{center}
$\begin{CD}
R\{\{F\}\} @>{\overline{\phi}}>> R'\{\{F\}\} @>{\psi}>> R'\{\{F'\}\} @>{\overline{\phi'}}>> R\{\{F'\}\} @>{\psi'}>> R\{\{F\}\} \\
@AAA @AAA @AAA @AAA @AAA \\
R @>{\phi}>> R' @= R' @>{\phi'}>> R @= R
\end{CD}$
\end{center}

\medskip

\noindent where $R \defeq \lazard_G(F)$, $R' \defeq \lazard_G(F	')$, the maps $\psi$, $\psi'$ are the iso\morp s given by Lemma 13.1 and 13.2 in \cite{equi FGL} and $\overline{\phi}$, $\overline{\phi'}$ are the maps given by the universal properties of $R$, $R'$ respectively. Observe that $\overline{\phi}, \overline{\phi'}, \psi$ and $\psi'$ all preserve product, $G^*$-action and coproduct. 

\begin{prop}
\label{prop flag coef iso}
The compositions $\psi' \circ \overline{\phi'} \circ \psi \circ \overline{\phi}$ and $\phi' \circ \phi$ are both identity maps. In addition, $\psi \circ \overline{\phi}$ and $\phi$ are both iso\morp s.
\end{prop}

\begin{proof}
First of all, as the notation suggests, $\psi$ will send
$$y(\gamma) \defeq l_{\gamma \dual{\alpha_1}} y(\alpha_1) = l_{\gamma \dual{\alpha_1}} y(V^1) \in R'\{\{ F \}\}$$
to 
$$y(\gamma) \defeq l_{\gamma \dual{\beta_1}} y(\beta_1) = l_{\gamma \dual{\beta_1}} y(W^1) \in R'\{\{ F' \}\}$$
and similarly for $\psi'$. In other words, $\psi,\psi'$ fix elements of the form $y(\gamma)$.

Let $f \defeq \psi' \circ \overline{\phi'} \circ \psi \circ \overline{\phi}$. By Lemma 13.2 in \cite{equi FGL}, 
$$y(V^i) = y(\alpha_1) \cdots y(\alpha_i),$$
the map $f$ fixes $y(V^i)$. In $R\{\{F\}\}$, we have
$$y(V^i)y(V^j) = \sum_s b^{i,j}_s y(V^s).$$
Then, $f$ fixes the left hand side and sends the right hand side to $\sum_s (\phi' \circ \phi (b^{i,j}_s))\, y(V^s)$. Therefore,
$$\sum_s b^{i,j}_s y(V^s) = \sum_s (\phi' \circ \phi (b^{i,j}_s))\, y(V^s)$$
and hence, $\phi' \circ \phi (b^{i,j}_s) = b^{i,j}_s$. By similar arguments, $\phi' \circ \phi$ fixes all structure constants and hence is an identity. By symmetry, $\phi \circ \phi'$ is also an identity. So $\phi$ is an iso\morp. Since $f$ fixes $y(V^i)$ and elements in $R$, $f$ is an identity. Again, by symmetry, $\psi \circ \overline{\phi}$ is an iso\morp.
\end{proof}

Let us denote the \equi\ algebraic cobordism theories corresponding to $F$ and $F'$ by $\cob{}{G,F}{-}$ and $\cob{}{G,F'}{-}$ respectively. Moreover, define the theories $\bigcob{}{G,F}{-}$ and $\bigcob{}{G,F'}{-}$ by imposing the \textbf{(Sect)} and \textbf{(EFGL)} axioms on $\bigbasicmod{G}{F}{}{-}$ and $\bigbasicmod{G}{F'}{}{-}$ respectively. Notice that since the $\lazard_G(F)$-submodule of $\bigbasicmod{G}{F}{}{-}$ corresponding to imposing the axioms is actually a submodule of $\basicmod{G}{F}{}{-}$, we have $\cob{}{G,F}{-} \subset \bigcob{}{G,F}{-}$ and similarly, $\cob{}{G,F'}{-} \subset \bigcob{}{G,F'}{-}$.

For an object $X \in \gvar{G}$, let 
$$\overline{\Psi}_{F,F'} : \bigbasicmod{G}{F}{}{X} \to \bigbasicmod{G}{F'}{}{X}$$
be the canonical map which sends $a[f : Y \to X, \ldots]$ to $\phi(a) [f : Y \to X, \ldots]$, which induces a map 
$$\Psi_{F,F'} : \basicmod{G}{F}{}{X} \embed \bigbasicmod{G}{F}{}{X} \stackrel{\overline{\Psi}_{F,F'}}{\longto} \bigbasicmod{G}{F'}{}{X} \to \bigcob{}{G,F'}{X}.$$
Our goal is to show that it descends to a map $\cob{}{G,F}{X} \to \bigcob{}{G,F'}{X}$ with image inside $\cob{}{G,F'}{X}$.

First of all, for each infinite Chern class operator on $\bigbasicmod{G}{F}{}{-}$, we need to define an associated infinite Chern class operator on $\bigbasicmod{G}{F'}{}{-}$. For any $i \geq 0$ and $S$ as in (\ref{eqn basic element}), let 
$$y(V^i_S) \defeq \prod_{\stackrel{1 \leq j \leq i}{j \notin S}} y(\alpha_j)$$
as an element in $\lazard_G(F)\{\{F\}\}$, if $i \geq \max{S}$. Otherwise, set it to zero. Also, let $a^{i,S}_j \in \lazard_G(F')$ be the unique coefficients satisfying the following equation :
\begin{eqnarray}
\label{eqn 27}
\psi \circ \overline{\phi} (y(V^i_S)) = \sum_{j \geq 0} a^{i,S}_j\, y(W^j).
\end{eqnarray}

\begin{lemma}
\label{lemma large coef is zero}
For all $j \geq 0$ and $S$ as in (\ref{eqn basic element}), 
$$a^{i,S}_j = 0$$
for sufficiently large $i$.
\end{lemma}

\begin{proof}
Since $F'$ is a complete $G$-flag, for sufficiently large $i$,
$$y(V^i_S) = y(W^{j+1}) y(\gamma_1) \cdots y(\gamma_n)$$
for some $n$ and characters $\gamma_k$. As the map $\psi \circ \overline{\phi}$ fixes elements of the form $y(\gamma)$, it also fixes the element 
$$y(W^{j+1}) = y(\beta_1) \cdots y(\beta_{j+1}).$$
Therefore,
\begin{eqnarray}
\psi \circ \overline{\phi}(y(V^i_S)) &=& \psi \circ \overline{\phi}(y(W^{j+1}) y(\gamma_1) \cdots y(\gamma_n)) \nonumber\\
&=& y(W^{j+1}) \psi \circ \overline{\phi}(y(\gamma_1) \cdots y(\gamma_n)) \nonumber\\
&=& y(W^{j+1}) \sum_{k \geq 0} a_k' y(W^k), \nonumber
\end{eqnarray}
for some $a_k' \in \lazard_G(F')$. Therefore,
$$\psi \circ \overline{\phi}(y(V^i_S)) = \sum_{k,\, l \geq 0} a_k'\, {b'}^{j+1,k}_l\, y(W^l).$$
The result then follows from the fact that $b'^{j+1,k}_j = 0$ for all $k$ (by Proposition 14.1 in \cite{equi FGL}).
\end{proof}

Now, we define an operator $W^i_S(\L)$ on $\bigbasicmod{G}{F'}{}{X}$ as an analogue of $V^i_S(\L)$ (use $\beta_j$ instead of $\alpha_j$). Then, for each infinite Chern class operator $\sigma = \sum_I a_I V^{i_1}_{S_1}(\L_1) \cdots V^{i_r}_{S_r}(\L_r)$ on $\bigbasicmod{G}{F}{}{X}$, we define an associated infinite Chern class operator on $\bigbasicmod{G}{F'}{}{X}$ :
$$\Psi_{F,F'}(\sigma) \defeq \sum_J \sum_I \phi(a_I) a^{i_1, S_1}_{j_1} \cdots a^{i_r, S_r}_{j_r}\, W^{j_1}(\L_1) \cdots W^{j_r}(\L_r)$$ 
(it is well-defined by Lemma \ref{lemma large coef is zero}). 

\begin{lemma}
\label{lemma flag indep}
For all $i \geq 0$ and $S$ as in (\ref{eqn basic element}), 
$$V^i_S(\L) = \sum_{j \geq 0} a^{i,S}_j\, W^j(\L)$$
as operators on $\bigcob{}{G,F'}{X}$.
\end{lemma}

\begin{proof}
Note that, as an operator on $\bigcob{}{G,F'}{X}$, 
$$V^i_S(\L) = c(\L \otimes \alpha_1)  c(\L \otimes \alpha_2) \cdots c(\L \otimes \alpha_{i})$$
omitting $\L \otimes \alpha_k$ whenever $k \in S$.

If $i < \max{S}$, then the statement is trivially true. Suppose $i \geq \max{S}$. Denote the indices from 1 to $i$ which is not in $S$ by $j_1, \ldots, j_n$. Then, on one hand,
\begin{eqnarray}
\psi \circ \overline{\phi} (y(V^i_S)) &=& \psi \circ \overline{\phi} (y(\alpha_{j_1}) \cdots y(\alpha_{j_n})) \nonumber\\
&=& l_{\alpha_{j_1} \dual{\beta_1}}\, y(W^1) \cdots l_{\alpha_{j_n} \dual{\beta_1}}\, y(W^1) \nonumber\\
&=& (\sum_{k_1}\, d'(\alpha_{j_1} \dual{\beta_1})^1_{k_1}\, y(W^{k_1})) \cdots (\sum_{k_n}\, d'(\alpha_{j_n} \dual{\beta_1})^1_{k_n}\, y(W^{k_n})), \nonumber
\end{eqnarray}
which can then be expressed in terms of $\{y(W^j)\}$ by the equation
$$y(W^k) y(W^l) = \sum_p b'^{k,l}_p y(W^p).$$
On the other hand,
\begin{eqnarray}
V^i_S(\L) &=& c(\L \otimes \alpha_{j_1}) \cdots c(\L \otimes \alpha_{j_n}) \nonumber\\
&=& W^1(\L \otimes \alpha_{j_1} \dual{\beta_1}) \cdots W^1(\L \otimes \alpha_{j_n} \dual{\beta_1}) \nonumber\\
&=& (\sum_{k_1}\, d'(\alpha_{j_1} \dual{\beta_1})^1_{k_1}\, W^{k_1}(\L)) \cdots (\sum_{k_n}\, d'(\alpha_{j_n} \dual{\beta_1})^1_{k_n}\, W^{k_n}(\L)), \nonumber
\end{eqnarray}
which can be expressed in terms of $\{W^j(\L)\}$ by the equation
$$W^k(\L) W^l(\L) = \sum_p b'^{k,l}_p W^p(\L).$$
The result then follows from matching the coefficients and equation (\ref{eqn 27}).
\end{proof}

By the map $\phi : \lazard_G(F) \to \lazard_G(F')$, we may consider $\cob{}{G,F'}{-}$ as a $\lazard_G(F)$-module.

\begin{prop}
\label{prop flag independent}
For any $X \in \gvar{G}$, $\Psi_{F,F'}$ defines a canonical $\lazard_G(F)$-module iso\morp 
$$\Psi_{F,F'} : \cob{}{G,F}{X} \iso \cob{}{G,F'}{X}$$ 
and it commutes with the \proj\ push-forward, \sm\ pull-back, external product and the infinite Chern class operator, i.e., $\Psi_{F,F'} \circ \sigma = \Psi_{F,F'}(\sigma) \circ \Psi_{F,F'}$.
\end{prop}

\begin{proof}
The map 
$$\Psi_{F,F'} : \basicmod{G}{F}{}{-} \to \bigcob{}{G,F'}{-}$$ 
clearly commutes with the \proj\ push-forward, \sm\ pull-back and external product. For the infinite Chern class operator, 
\begin{eqnarray}
\Psi_{F,F'} \circ \sigma [\id_Y] &=& \Psi_{F,F'}(\sum_I a_I V^{i_1}_{S_1}(\L_1) \cdots V^{i_r}_{S_r}(\L_r)[\id_Y]) \nonumber\\
&=& \sum_I \phi(a_I) V^{i_1}_{S_1}(\L_1) \cdots V^{i_r}_{S_r}(\L_r)[\id_Y] \nonumber\\
&=& \sum_I \sum_J \phi(a_I) a^{i_1, S_1}_{j_1} \cdots a^{i_r, S_r}_{j_r} W^{j_1}(\L_1) \cdots W^{i_r}(\L_r)[\id_Y] \nonumber
\end{eqnarray}
by Lemma \ref{lemma flag indep}, which is equal to $\Psi_{F,F'}(\sigma) \circ \Psi_{F,F'} [\id_Y]$. So, $\Psi_{F,F'}$ commutes with the basic operations. Clearly, it respects the \textbf{(Sect)} axiom. Moreover, by Lemma \ref{lemma flag indep} and the dictionary between $y(V^i_S)$ and $V^i_S(\L)$, the map $\Psi_{F,F'}$ also respects the \textbf{(EFGL)} axiom. Therefore, $\Psi_{F,F'}$ descends to a map $\cob{}{G,F}{-} \to \bigcob{}{G,F'}{-}$. Since $\cob{}{G,F}{X}$ is generated by elements of the form $f_* \circ \sigma [\id_Y]$ and $\Psi_{F,F'}$ commutes with the basic operations, the image of $\Psi_{F,F'}$ lies inside $\cob{}{G,F'}{X}$. Hence, we have a canonical $\lazard_G(F)$-module \homo
$$\Psi_{F,F'} : \cob{}{G,F}{X} \to \cob{}{G,F'}{X}.$$ 
By symmetry, we also have $\Psi_{F',F} : \cob{}{G,F'}{X} \to \cob{}{G,F}{X}.$ Then, $\Psi_{F,F'}$ is an iso\morp\ because $\Psi_{F',F} \circ \Psi_{F,F'}$ and $\Psi_{F,F'} \circ \Psi_{F',F}$ are both identity maps (by Proposition \ref{prop flag coef iso}).
\end{proof}

\bigskip

\bigskip


\end{document}